\documentclass[10pt,reqno]{amsart}
\usepackage{amssymb,mathrsfs,graphicx}
\usepackage{ifthen}
\usepackage{caption}
\usepackage{rotating,dsfont}
\usepackage{enumitem}
\DeclareMathOperator*{\esssup}{ess\,sup}

\usepackage[normalem]{ulem}
\usepackage{soul}
\usepackage{cancel}

\usepackage[top=1in, bottom=1in, left=1in, right=1in]{geometry}

\usepackage[dvipsnames]{xcolor} 

\usepackage{tikz}

\usepackage[colorlinks=true, pdfstartview=FitV, linkcolor=RoyalBlue, citecolor=ForestGreen, urlcolor=blue]{hyperref}

\definecolor{labelkey}{rgb}{0,0,1}

\title[ Finite-energy weak solutions]{Finite-energy weak solutions and relaxation for a compressible kinetic--fluid system with locally averaged Brinkman force}

\author[Choi]{Young-Pil Choi}
\address{Department of Mathematics, Yonsei University, Seoul 03722, Republic of Korea}
\email{ypchoi@yonsei.ac.kr}

\author[Shvydkoy]{Roman Shvydkoy}
\address{851 S Morgan St, M/C 249, Department of Mathematics, Statistics and Computer Science, University of Illinois at Chicago, Chicago, IL 60607}
\email{shvydkoy@uic.edu}

\numberwithin{equation}{section}

\newtheorem{theorem}{Theorem}[section]
\newtheorem{lemma}{Lemma}[section]

\newtheorem{proposition}{Proposition}[section]
\newtheorem{remark}{Remark}[section]
\newtheorem{definition}{Definition}[section]

\newcommand{\R}{\mathbb R}

\newcommand{\N}{\mathbb N}

\newcommand{\T}{\mathbb T}

\newcommand{\bq}{\begin{equation}}
\newcommand{\eq}{\end{equation}}
\newcommand{\e}{\varepsilon}
\newcommand{\lt}{\left}
\newcommand{\rt}{\right}
\newcommand{\lal}{\langle}
\newcommand{\ral}{\rangle}
\newcommand{\pa}{\partial}

\newcommand{\intr}{\int_{\R^d}}
\newcommand{\intt}{\int_{\T^d}}
\newcommand{\inttr}{\iint_{\T^d \times \R^d}}

\newcommand{\calA}{\mathcal A}

\newcommand{\calD}{\mathcal D}
\newcommand{\calE}{\mathcal E}

\newcommand{\calG}{\mathcal G}
\newcommand{\calH}{\mathcal H}
\newcommand{\calI}{\mathcal I}

\newcommand{\calM}{\mathcal M}

\newcommand{\calP}{\mathcal P}

\newcommand{\calR}{\mathcal R}
\newcommand{\calS}{\mathcal S}

\newcommand{\calX}{\mathcal X}

\newcommand{\rd}{\textnormal{d}}
\newcommand{\dd}{\textnormal{d}}
\newcommand{\dx}{\textnormal{d}x}

\newcommand{\dt}{\textnormal{d}t}
\newcommand{\dv}{\textnormal{d}v}
\newcommand{\dy}{\textnormal{d}y}

\newcommand{\dz}{\textnormal{d}z}

\newcommand{\ds}{\textnormal{d}s}
\newcommand{\dmu}{\textnormal{d}\mu}

\newcommand{\jap}[1]{\left\langle #1 \right\rangle}
\newcommand{\I}{\ensuremath{\mathbb{I}}}   

\makeatletter
\def\moverlay{\mathpalette\mov@rlay}
\def\mov@rlay#1#2{\leavevmode\vtop{%
   \baselineskip\z@skip \lineskiplimit-\maxdimen
   \ialign{\hfil$\m@th#1##$\hfil\cr#2\crcr}}}
\newcommand{\charfusion}[3][\mathord]{
    #1{\ifx#1\mathop\vphantom{#2}\fi
        \mathpalette\mov@rlay{#2\cr#3}
      }
    \ifx#1\mathop\expandafter\displaylimits\fi}
\makeatother

\def \ddt {\frac{\dd}{\dt}}
\def \free { {\rm free}}
\def \BD { {\rm BD}}
\def \aux { {\rm aux}}
\def \vis { {\rm vis}}
\def \drag {{\rm drag}}

\def \cI {\mathcal{I}}
\def \cD {\mathcal{D}}

\def \st {\mathrm{s}}

\def \n {\nabla}

\def \rmb {\mathrm{b}}

\begin{document}
\allowdisplaybreaks

\date{\today}


\begin{abstract}
We study a kinetic--fluid system in which a Vlasov or Vlasov--Fokker--Planck equation is coupled to the compressible Navier--Stokes equations with density-dependent viscosities through a locally averaged Brinkman force. The averaging is chosen in a conservative form so that the coupled system preserves the total momentum and satisfies a natural energy-dissipation balance, while avoiding the pointwise evaluation of a possibly rough fluid velocity. The first main result is a conditional exponential relaxation estimate for sufficiently regular solutions. The estimate is proved under positive and negative density moment bounds and a Muckenhoupt $\calA_2$ condition on a power of the density, which replace uniform pointwise upper and lower bounds. The proof combines a modulated energy and hypocoercivity analysis with a compensating functional for the density fluctuation.  A key ingredient is a weighted
Calder\'on--Zygmund estimate for the associated elliptic corrector, which allows us to control the viscous contribution under the $\calA_2$ condition. The second main result concerns global finite-energy Bresch--Desjardins entropy weak solutions in admissible low-dimensional regimes. Using the additional entropy structure associated with the density-dependent viscosities, we verify the density assumptions required by the conditional relaxation theorem and obtain exponential relaxation for the weak solutions. In the diffusionless case, the particle distribution aligns toward a mono-kinetic state, while in the diffusive case the relaxation is towards a Maxwellian equilibrium. 
\end{abstract}

 \maketitle  \centerline{\date}

\tableofcontents

%
%
%
%

\section{Introduction}

Sprays and particle-laden flows arise in many physical and engineering situations, including combustion, aerosols, sedimentation, cooling processes, and dispersed multiphase flows. They consist of a large number of droplets, dust particles, or other small inclusions suspended in a surrounding gas or fluid. Depending on the concentration of the dispersed phase and on the interaction mechanisms, several levels of description are possible, ranging from microscopic particle systems to mesoscopic kinetic equations and macroscopic multiphase flow models; see, for instance, \cite{Wil58,ORo81,Des10}.

At the mesoscopic level, the dispersed phase is often described by a distribution function in phase space, while the carrier fluid is described by macroscopic balance laws. A central mechanism in such models is the momentum exchange between the particles and the fluid through a drag force. This force tends to align the particle velocity with the local fluid velocity and leads to coupled Vlasov--Navier--Stokes or Vlasov--Fokker--Planck--Navier--Stokes type systems. In this paper, we study a kinetic--fluid coupling in a compressible viscous regime, with emphasis on finite-energy weak solutions and their long-time relaxation.

%
%
%
%
%

\subsection{A kinetic--fluid model with a locally averaged Brinkman force}

We consider a kinetic--fluid model in which a Vlasov or Vlasov--Fokker--Planck particle phase is coupled to a compressible viscous fluid with density-dependent viscosities. The spatial domain is the flat torus $\T^d$, normalized by $|\T^d|=1$. The particle distribution is denoted by $f=f(t,x,v)\ge0$, where $(x,v)\in\T^d\times\R^d$, while $(\rho,u)=(\rho(t,x),u(t,x))$ denotes the density and velocity of the compressible fluid. 

The coupling is through a locally averaged Brinkman force. Let $\theta_\e\in C^\infty(\T^d)$ be a fixed symmetric mollifier satisfying
\[ 
\theta_\e(x)>0 \quad \text{for every }x\in\T^d, \quad \intt \theta_\e\,\dx=1. 
\] 
For a function $g=g(x)$, we write 
\[ 
g_\e:=\theta_\e*g.
\]

We also introduce the kinetic density and momentum
\[
\rho_f(t,x):= \intr f(t,x,v)\,\dv, \quad m_f(t,x):= \intr vf(t,x,v)\,\dv.
\]
The coupled Vlasov--compressible Navier--Stokes system is
\begin{align}\label{main_sys}
\begin{aligned}
&\pa_t f + v \cdot \nabla f = \rho_\e \nabla_v \cdot (\sigma \nabla_v f +  (v-u_\e^\e)f), \quad (t,x,v)\in \R_+\times\T^d\times\R^d,\cr
 &\pa_t \rho  + \nabla \cdot (\rho  u) = 0,\cr
 &\pa_t (\rho  u) + \nabla  \cdot (\rho  u \otimes u) + \nabla p(\rho ) = \nabla \cdot \mathbb S(\rho,u) +\rho F_\e,
 \end{aligned}
\end{align}
where 
\[
\mathbb S(\rho,u) := 2\nu(\rho)D u + \lambda(\rho)(\nabla\cdot u)\mathbb I, \quad D u:=\frac12\lt(\nabla u+\nabla u^{\mathsf T}\rt),
\]
and
\[
F :=  \intr (v - u_\e^\e)f\,\dv, \quad u_\e^\e:=\frac{(\rho u)_\e}{\rho_\e}.
\]
The system is supplemented with initial data
\[
f|_{t=0}=f_0,\quad \rho|_{t=0}=\rho_0,\quad (\rho u)|_{t=0}=\rho_0u_0.
\]
We impose mass normalizations
\[
\inttr f_0\,\dx\dv=1, \quad \intt \rho_0\,\dx=1.
\]
By the conservation of the kinetic and fluid masses, see Section \ref{ssec:cons} below, these normalizations are preserved along the flow.  Since $\theta_\e$ is strictly positive, we have
\[ 
\rho_\e(t,x)\ge\min_{\T^d}\theta_\e>0. 
\] 
Hence $u_\e^\e$ is well defined and represents the locally averaged fluid velocity seen by the particles.

In this paper, the viscosity coefficients and pressure law are given by 
\[ 
\nu(\rho)=\rho^\alpha, \quad \lambda(\rho)=2(\alpha-1)\rho^\alpha, \quad p(\rho)=\rho^\gamma, 
\] 
where $\alpha>0$ and $\gamma>1$. These viscosity coefficients satisfy the Bresch--Desjardins relation
\bq\label{eq:bd-rel}
\lambda(\rho) = 2\lt(\rho\nu'(\rho)-\nu(\rho)\rt).
\eq
This structural relation will be used to derive an additional entropy estimate for the density.

The inclusion of the Brinkman force in an averaged form in \eqref{main_sys} has both modeling and analytical motivations. The classical Brinkman coupling is based on the relative velocity $v-u$ and corresponds, at least formally, to a pointwise interaction between particles and the fluid velocity at the same spatial point. This is a useful idealization, but it is not the only natural one from the viewpoint of modeling or approximation. In particle descriptions, the particle phase is often represented by empirical measures, and singular particle--fluid interactions are commonly replaced by smooth blobs or localized averages. From this perspective, the velocity $u_\e^\e$ can be interpreted as the effective fluid velocity seen by particles after averaging over a small spatial scale.
The averaged coupling is also consistent with particle approximation and scaling-limit viewpoints for Vlasov--Navier--Stokes type systems. In such problems, suitable modifications or regularizations of the microscopic drag interaction naturally appear in order to pass from a particle system to a kinetic--fluid equation; see, for instance, \cite{FLR19,FLR21}. Related derivations of Vlasov--Navier--Stokes type models for sprays or aerosol flows from kinetic theory can also be found in \cite{BDGR17, BDGR18}. Thus a regularized or averaged interaction may be viewed as a natural mesoscopic approximation of the particle--fluid coupling, rather than merely as a technical device.

Let us also mention that locally averaged Brinkman forces have recently been used in the analysis of incompressible kinetic--fluid systems. In \cite{ST25}, a Fokker--Planck--Navier--Stokes system with locally averaged Brinkman drag was analyzed, and global weak solutions together with unconditional alignment were obtained. In that setting, the averaging plays an important role in stabilizing the hypocoercivity estimates and in proving relaxation toward a global Maxwellian equilibrium. The present paper follows this averaged-interaction viewpoint in a compressible viscous setting.

From the analytical point of view, the averaged Brinkman force is particularly useful in the compressible finite-energy setting. In the presence of density-dependent viscosities and possible vacuum, the natural energy controls $\sqrt\rho u$ rather than $u$ itself. Thus the pointwise velocity field may not be sufficiently regular to be inserted directly into the kinetic equation at the level of weak solutions. By contrast, since $\rho_\e$ is bounded from below by a positive constant depending only on $\e$, the averaged velocity
\[
u_\e^\e=\frac{(\rho u)_\e}{\rho_\e}
\]
is well defined even when the density has vacuum regions.

Moreover, the regularization is chosen in a conservative form. The factor $\rho_\e$ in the kinetic equation and the mollified forcing term $\rho F_\e$ in the fluid equation are compatible with each other. Consequently, the drag interaction preserves the total momentum and produces a nonnegative drag dissipation in the energy balance. Thus, the averaged coupling retains the basic physics of the system while avoiding the singular pointwise evaluation of a possibly rough fluid velocity.

%
%
%
%
%
 
\subsection{Main results}

We now state the main results of the paper.

%
%
%
%
%
\subsubsection{Conditional exponential relaxation}
The first result gives a conditional exponential relaxation estimate for sufficiently regular solutions. We begin by fixing the notation used in its formulation.

For $\sigma>0$, let
\[
\mu_{a,b,\sigma}(v) := \frac{a}{(2\pi\sigma)^{\frac d2}} \exp\lt(-\frac{|v-b|^2}{2\sigma}\rt)
\]
denote the Maxwellian with mass $a$, center $b$, and temperature $\sigma$. For $\sigma \ge 0$, we use the following convention for the kinetic part of the free energy:
\[
H
:=
\begin{cases}
\displaystyle \sigma\inttr f\log\frac{f}{\mu_{1,0,\sigma}}\,\dx\dv, & \sigma>0,\\[3mm]
\displaystyle \frac12\inttr |v|^2 f\,\dx\dv, & \sigma=0.
\end{cases}
\]
Thus, when $\sigma=0$, the relative entropy part is understood as the kinetic energy. We define the kinetic and fluid momenta by
\[ 
\bar{v}(t):=\inttr vf(t,x,v)\,\dx\dv, \quad \bar{m}(t):=\intt \rho(t,x)u(t,x)\,\dx. 
\]
We also recall the definition of the Muckenhoupt class $\calA_2(\T^d)$. A nonnegative locally integrable function $\omega$ belongs to $\calA_2(\T^d)$ if
\[
[\omega]_{\calA_2} := \sup_{Q\subset\T^d} \lt(\frac1{|Q|}\int_Q\omega\,\dx\rt) \lt(\frac1{|Q|}\int_Q\omega^{-1}\,\dx\rt) <\infty,
\]
where the supremum is taken over all cubes $Q\subset\T^d$. We refer to $[\omega]_{\calA_2}$ as the $\calA_2$-characteristic of $\omega$.

\begin{theorem}[Conditional relaxation in any dimension]\label{thm:main1}
Let $(f,\rho,u)$ be a sufficiently regular solution to \eqref{main_sys} on $[0,\infty)$ satisfying
\[
H(0) +\frac12\intt \rho_0|u_0|^2\,\dx+\frac1{\gamma-1}\intt \rho_0^\gamma\,\dx<\infty.
\]
Assume that
\[
1-\frac1d\le\alpha\le\gamma-1.
\]
Let $r\in(1,2)$ satisfy
\[
r>\frac{2d}{d+2}  \text{ when $d\ge2$, and define }  
r^\star:=
\begin{cases}
\dfrac{dr}{d-r}, & d\ge2,\\[2mm]
\infty, & d=1.
\end{cases}
\]
Assume that there exist exponents $p,s>0$ such that
\bq\label{eq:mom-exp}
\frac2p+\frac2{r^\star}\le1, \quad s\ge\frac{\alpha r}{2-r},
\eq
and, when $d\ge2$ and $\gamma<d$,
\bq\label{eq:pot-exp-c}
2p'\le\frac{d\gamma}{d-\gamma}.
\eq
Suppose that there exists a constant $M>0$ such that
\bq\label{eq:pos-nm-c}
\sup_{t\ge0}\intt \lt(\rho(t,x)^p+\rho(t,x)^{-s}\rt)\dx\le M
\eq
and
\bq\label{eq:A2-dc}
\sup_{t\ge0}[\rho(t)^\alpha]_{\calA_2}\le M.
\eq
Then there exist constants $C>0$ and $\lambda>0$ such that in the diffusionless case $\sigma=0$ we have
\bq\label{eq:cond-rel-zero}
\inttr|v-\bar{v}|^2f\,\dx\dv+\intt\rho|u-\bar{m}|^2\,\dx+\|\rho-1\|_{L^1(\T^d)}^2+|\bar{v}-\bar{m}|^2\le Ce^{-\lambda t},
\eq
while for $\sigma>0$ we have relaxation to the momentum-centered Maxwellian equilibrium
\bq\label{eq:cond-rel-pos}
\inttr \lt| f - \mu_{1,\bar{v},\sigma} \rt| \dx\dv+\intt\rho|u-\bar{m}|^2\,\dx+\|\rho-1\|_{L^1(\T^d)}^2+|\bar{v}-\bar{m}|^2\le Ce^{-\lambda t}.
\eq
In the case $\sigma > 0$, in particular, the particle density $\rho_f$ tends to a uniform distribution.
\end{theorem}

\begin{remark} 
The conclusion of Theorem \ref{thm:main1} can be strengthened by using the power bounds assumption in \eqref{eq:pos-nm-c}. Indeed, for every $1\le q<p$, interpolation between $L^1(\T^d)$ and $L^p(\T^d)$ yields
\[
\|\rho(t)-1\|_{L^q(\T^d)}\le\|\rho(t)-1\|_{L^1(\T^d)}^{\tau}\|\rho(t)-1\|_{L^p(\T^d)}^{1-\tau},
\quad
\tau:=\frac{p-q}{q(p-1)}.
\]
Consequently, for every $1\le q<p$, there exist constants $C_q>0$ and $\lambda_q>0$ such that
\[
\|\rho(t)-1\|_{L^q(\T^d)}^2\le C_qe^{-\lambda_qt}.
\]
In particular, if $p>2$, then
\[
\|\rho(t)-1\|_{L^2(\T^d)}^2\le C_2e^{-\lambda_2t}.
\]
\end{remark}

\begin{remark}[Asymptotic velocity and mono-kinetic relaxation] For $\sigma>0$, Theorem \ref{thm:main1} identifies the limiting kinetic profile as the Maxwellian equilibrium. In the diffusionless case $\sigma=0$, one instead expects mono-kinetic relaxation in velocity, although not in the strong $L^1$ topology. Let us elaborate on this point.
   
By conservation of the total momentum (see Section \ref{ssec:cons} below), we have
\[
\bar{v}(t)+\bar{m}(t) = \inttr vf_0\,\dx\dv+\intt \rho_0u_0\,\dx =:P_0.
\]
Since Theorem \ref{thm:main1} gives $|\bar{v}(t)-\bar{m}(t)|\leq Ce^{-\lambda t}$, it follows that
\[
\bar{v}(t)\to v_\infty, \quad \bar{m}(t)\to v_\infty, \quad v_\infty:=\frac12P_0,
\]
and the convergence is exponential.

Furthermore, the kinetic distribution relaxes to a mono-kinetic state in velocity. Indeed, for any bounded Lipschitz test function $\varphi=\varphi(x,v)$ with $\|\varphi\|_{L^\infty}\leq1$ and $\|\varphi\|_{\mathrm{Lip}}\leq1$, we have
\begin{align*}
\lt| \inttr \varphi(x,v)f(t,x,v)\,\dx\dv - \intt \varphi(x,\bar{v}(t))\rho_f(t,x)\,\dx \rt| & \leq \inttr |v-\bar{v}(t)|f(t,x,v)\,\dx\dv \\
& \leq \lt( \inttr |v-\bar{v}(t)|^2f(t,x,v)\,\dx\dv \rt)^{\frac12}.
\end{align*}
Thus
\[
\dd_{\rm BL}\lt(f(t,x,v)\,\dx\dv,\, \rho_f(t,x)\delta_{\bar{v}(t)}(\dv)\dx\rt) \leq Ce^{-\lambda t}.
\]
Using $\bar{v}(t)\to v_\infty$, we also obtain
\[
\dd_{\rm BL}\lt(f(t,x,v)\,\dx\dv,\, \rho_f(t,x)\delta_{v_\infty}(\dv)\dx\rt) \leq Ce^{-\lambda t}.
\]
\end{remark}

\begin{remark}[Asymptotic traveling profile of $\rho_f$] Again, for $\sigma>0$, the only possible limiting distribution of the particle density $\rho_f$ is constant. For $\sigma = 0$, such profile is not the only possibility.  In this case exponential relaxation implies that the particle density converges weakly to a traveling wave. Indeed, integrating the kinetic equation with respect to $v$ gives
\bq\label{eq:conti}
\pa_t\rho_f+\nabla\cdot m_f=0.  
\eq
Since $\bar{v}(t)\to v_\infty$ and
\bq\label{eq:kin_dec}
\intt |m_f-\rho_fv_\infty|\,\dx \leq \inttr |v-\bar{v}|f\,\dx\dv + |\bar{v}-v_\infty| \leq Ce^{-\lambda t},
\eq
we have
\[
\pa_t\rho_f+v_\infty\cdot\nabla\rho_f
=
-\nabla\cdot(m_f-\rho_fv_\infty),
\]
with an exponentially decaying right-hand side in a negative norm.

Equivalently, we may formulate this convergence in a moving frame. Set
\[
\widetilde\rho_f(t,x):=\rho_f(t,x+v_\infty t).
\]
Then, for any $\varphi\in W^{1,\infty}(\T^d)$, by the change of variables $y=x+v_\infty t$, we have
\[
\intt \widetilde\rho_f(t,x)\varphi(x)\,\dx = \intt \rho_f(t,y)\varphi(y-v_\infty t)\,\dy .
\]
Differentiating in time and using \eqref{eq:conti}, we obtain
\begin{align*}
\frac{\dd}{\dt}
\intt \widetilde\rho_f(t,x)\varphi(x)\,\dx &= \frac{\dd}{\dt} \intt \rho_f(t,y)\varphi(y-v_\infty t)\,\dy \\
&= \intt \pa_t\rho_f(t,y)\varphi(y-v_\infty t)\,\dy -\intt \rho_f(t,y)v_\infty\cdot\nabla\varphi(y-v_\infty t)\,\dy \\
&= \intt m_f(t,y)\cdot\nabla\varphi(y-v_\infty t)\,\dy -\intt \rho_f(t,y)v_\infty\cdot\nabla\varphi(y-v_\infty t)\,\dy \\
&= \intt \lt(m_f-\rho_f v_\infty\rt)(t,y) \cdot\nabla\varphi(y-v_\infty t)\,\dy .
\end{align*}
Thus, by \eqref{eq:kin_dec}, 
\bq\label{eq:rho-f-mov}
\lt| \frac{\dd}{\dt} \intt \widetilde\rho_f(t,x)\varphi(x)\,\dx \rt| \leq \|m_f-\rho_f v_\infty\|_{L^1(\T^d)} \|\nabla\varphi\|_{L^\infty(\T^d)} \leq
Ce^{-\lambda t}\|\nabla\varphi\|_{L^\infty(\T^d)}.
\eq
This estimate implies that $\widetilde\rho_f(t)$ is Cauchy in the dual space $(W^{1,\infty}(\T^d))'$. Indeed, for $0<s<t$, integrating \eqref{eq:rho-f-mov} over $[ s,t ]$ gives
\[
\lt| \intt \lt(\widetilde\rho_f(t,x)-\widetilde\rho_f(s,x)\rt) \varphi(x)\,\dx \rt| \leq C\lt(\int_s^t e^{-\lambda \tau}\,\dd\tau \rt) \|\nabla\varphi\|_{L^\infty(\T^d)} \leq Ce^{-\lambda s}\|\varphi\|_{W^{1,\infty}(\T^d)}.
\]
Taking the supremum over $\|\varphi\|_{W^{1,\infty}(\T^d)}\leq1$, we obtain
\[
\|\widetilde\rho_f(t)-\widetilde\rho_f(s)\|_{(W^{1,\infty})'} \leq Ce^{-\lambda s}.
\]
Hence there exists a limit $\rho_f^\infty\in (W^{1,\infty}(\T^d))'$ such that
\[
\widetilde\rho_f(t)\to \rho_f^\infty \quad \text{in }(W^{1,\infty}(\T^d))' \quad\text{as }t\to\infty.
\]
Since each $\widetilde\rho_f(t)$ is a nonnegative measure with total mass one, the limit $\rho_f^\infty$ is also a probability measure on $\T^d$.

Equivalently, for every $\varphi\in W^{1,\infty}(\T^d)$,
\[
\intt \rho_f(t,x)\varphi(x-v_\infty t)\,\dx \to \intt \varphi(x)\,\rho_f^\infty(\dx) \quad\text{as }t\to\infty.
\]
In this sense,
\[
\rho_f(t,x)\rightharpoonup \rho_f^\infty(x-v_\infty t)
\]
in the moving frame. Thus, in the absence of spatial diffusion, the particle density need not become spatially homogeneous; rather, it converges weakly to a profile travelling with the asymptotic velocity $v_\infty$.
\end{remark}

%
%
%
%
%
\subsubsection{BD entropy weak solutions and relaxation}
Theorem \ref{thm:main1} is conditional in the sense that it assumes uniform positive and negative density moment bounds, together with the weighted
$\calA_2$ condition. These assumptions replace uniform pointwise upper and lower bounds on $\rho$ by integral and weighted density conditions, but they
do not follow from the free energy estimate alone. In the admissible low-dimensional regimes considered below, they can be verified by using the additional entropy structure associated with the Bresch--Desjardins relation \eqref{eq:bd-rel}. This leads to our second result, which concerns global finite-energy BD entropy weak solutions and their relaxation.

We introduce the energy and dissipation functionals used in the weak formulation. We suppress their time dependence in the notation whenever
no confusion can arise. 

The fluid part of the free energy is
\[
E := \intt \lt[ \frac12\rho|u|^2+\pi(\rho) \rt]\dx,
\]
where the pressure potential $\pi$ is given by
\[
\pi(r) = r\int_1^r\frac{s^\gamma-1}{s^2}\,\ds = \frac{r^\gamma-\gamma r+\gamma-1}{\gamma-1}, \quad r\ge0.
\]
We define
\[
\calE_{\rm free}:=H +E.
\]

The drag dissipation is
\[
\calD_{\rm drag} := \inttr \frac{\rho_\e}{f} |\sigma\nabla_v f+(v-u_\e^\e)f |^2 \,\dx\dv.
\]
When $\sigma=0$, this reduces to
\[
\calD_{\rm drag}= \inttr \rho_\e |v-u_\e^\e|^2 f\,\dx\dv.
\]
The viscous dissipation is
\[
\calD_{\rm vis}:= \intt \lt( 2\nu(\rho)|D u|^2 + \lambda(\rho)|\nabla\cdot u|^2 \rt)\dx.
\]

We also introduce the BD entropy quantities. Let $\phi$ be determined by
\[
\phi'(\rho)=\frac{2\nu'(\rho)}{\rho} = 2\alpha\rho^{\alpha-2}.
\]
Equivalently,
\bq\label{eq:def-phi}
\sqrt{\rho}\nabla\phi(\rho)
=
\begin{cases}
\dfrac{2\alpha}{\alpha-\frac12}\nabla\rho^{\alpha-\frac12}, & \alpha\neq\frac12,\\[2mm]
\nabla\log\rho, & \alpha=\frac12.
\end{cases}
\eq
The BD entropy is defined by
\[
\calE_{\rm BD} := H + \intt \lt[ \frac12\rho|u+\nabla\phi(\rho)|^2+\pi(\rho) \rt]\dx.
\]
The corresponding BD and rotational dissipations are
\[
\calD_{\rm BD}:= \intt \nabla p(\rho)\cdot\nabla\phi(\rho)\,\dx, \quad \calD_{\rm rot} := 2\intt \nu(\rho)|A u|^2\,\dx,
\]
where
\[
A u := \frac12\lt(\nabla u-\nabla u^{\mathsf T}\rt).
\]
Finally, for $\eta>0$, we set
\[
\calE_\eta:= \calE_{\rm free} +\eta\calE_{\rm BD}.
\]

We now state the existence of weak solutions result. The definition is formulated for both the diffusionless case $\sigma=0$ and the diffusive case $\sigma>0$. When $\sigma=0$, the kinetic component is naturally allowed to be measure-valued, while for $\sigma>0$ we work in the corresponding
finite-entropy class. The diffusive case is further discussed in Remark
\ref{rem:diffusive-case} below and in Section \ref{ssec:diff-ext}.

We denote by $\calP_2(\T^d\times\R^d)$ the space of probability measures on $\T^d\times\R^d$ with finite second velocity moment:
\[
\calP_2(\T^d\times\R^d) := \lt\{ f\in\calP(\T^d\times\R^d): \inttr |v|^2 f\,\dx\dv<\infty \rt\}.
\]
Since $\T^d$ is compact, the boundedness of the second moment with respect to any product distance on $\T^d\times\R^d$ is equivalent to the boundedness of the second velocity moment above. We write $W_2$ for the corresponding quadratic Wasserstein distance. When the kinetic component is only measure-valued, expressions such as $\inttr \varphi f\,\dx\dv$ are understood as duality pairings.  
For $1\le p<\infty$ and $q\ge0$, we use the convention
\[
\|g\|_{L^p_q}^p := \inttr \langle v\rangle^q |g(x,v)|^p\,\dx\dv, \quad \langle v\rangle=(1+|v|^2)^{\frac12}.
\]

\begin{definition}\label{def:weak-sol}
Fix $\e>0$, $\sigma\ge0$, and $T>0$. A triple $(f,\rho,u)$ is called a finite-energy BD entropy weak solution to \eqref{main_sys}  on $[0,T]$ if the following properties hold.

\medskip
\noindent
\textup{(i) Regularity.}
We have
\[
f\in
\begin{cases}
L^\infty(0,T;\calP_2(\T^d\times\R^d)),
& \sigma=0,\\[1mm]
L^\infty(0,T;L^1_2\cap L\log L(\T^d\times\R^d)),
& \sigma>0,
\end{cases}
\quad \rho\in L^\infty(0,T;L^\gamma_+(\T^d)),\quad \sqrt{\rho}\,u\in L^\infty(0,T;L^2(\T^d)),
\]  
\[
\nabla\rho^\alpha\in L^2(0,T;L^2(\T^d)),\quad \rho^{\frac{\alpha}{2}}\nabla u\in L^2(0,T;L^2(\T^d)), \quad \text{and} \quad \sqrt{\rho}\lt(u+\nabla\phi(\rho)\rt)\in L^\infty(0,T;L^2(\T^d)).
\]
 
\medskip
\noindent
\textup{(ii) Kinetic equation.}
For every $\varphi\in C_c^\infty([0,T)\times\T^d\times\R^d)$, we have
\begin{align*}
&\int_0^T\inttr f\lt(\pa_t\varphi+v\cdot\nabla_x\varphi\rt)\,\dx\dv\dt +\inttr f_0\varphi(0)\,\dx\dv   +\sigma\int_0^T\inttr
\rho_\e f\Delta_v\varphi\,\dx\dv\dt \\ 
&\quad = \int_0^T\inttr \rho_\e(v-u_\e^\e)f\cdot\nabla_v\varphi\,\dx\dv\dt.
\end{align*}
When $\sigma=0$, all phase-space integrals involving $f$ and $f_0$ are understood in the measure sense.

\medskip
\noindent
\textup{(iii) Continuity equation.}
For every $\psi\in C_c^\infty([0,T)\times\T^d)$, we have
\[
\int_0^T\intt \lt(\rho\pa_t\psi+\rho u\cdot\nabla\psi\rt)\,\dx\dt +\intt \rho_0\psi(0)\,\dx=0.
\]

\medskip

\noindent
\textup{(iv) Momentum equation.}
For every $\Psi\in C_c^\infty([0,T)\times\T^d;\R^d)$, we have
\begin{align*}
&\int_0^T\intt \lt(\rho u\cdot\pa_t\Psi+\rho u\otimes u:\nabla\Psi+\rho^\gamma\nabla\cdot\Psi\rt)\,\dx\dt -\int_0^T\intt \mathbb S(\rho,u):\nabla\Psi\,\dx\dt \\
&\quad +\int_0^T\intt \rho F_\e\cdot\Psi\,\dx\dt +\intt \rho_0u_0\cdot\Psi(0)\,\dx=0.
\end{align*}

\medskip

\noindent
\textup{(v) Free energy inequality.}
For almost every $s\in[0,T]$ and every $t\in[s,T]$,
\[
\calE_{\rm free}(t) +\int_s^t\lt(\calD_{\rm vis}(\tau)+\calD_{\rm drag}(\tau)\rt)\,\dd\tau \le \calE_{\rm free}(s).
\]

\medskip

\noindent
\textup{(vi) Combined BD entropy inequality.}
There exists $\eta>0$ such that, for almost every $t\in[0,T]$,
\[
\calE_\eta(t) + \int_{0}^t \lt[ \calD_{\rm vis}(\tau) + \frac12\calD_{\rm drag}(\tau) + \frac{\eta}{2}\calD_{\rm BD}(\tau) + \eta\calD_{\rm rot}(\tau) \rt]\dd\tau \le \calE_\eta( 0 ).
\]
\end{definition}
 
\begin{theorem}\label{thm:weak-exist}
Fix $\e>0$  and $\sigma\ge0$. Suppose that one of the following conditions holds:
\bq\label{eq:admi-case}
\begin{cases}
d=1, \quad 0<\alpha\le\frac12, \quad \gamma=\alpha+1, \\[1mm]
d=2, \quad \alpha=\frac12, \quad \gamma=\frac32.
\end{cases}
\eq
Let the initial data satisfy
\[
f_0\in
\begin{cases}
\calP_2(\T^d\times\R^d),
& \sigma=0,\\[1mm]
L^1_2\cap L\log L(\T^d\times\R^d),
& \sigma>0,
\end{cases}
\quad \rho_0>0\quad\text{a.e. in }\T^d, \quad \rho_0\in L^\gamma(\T^d), \quad \sqrt{\rho_0}u_0\in L^2(\T^d),
\]
and assume the initial BD-compatibility condition 
\[ 
\sqrt{\rho_0}\nabla\phi(\rho_0) \in L^2(\T^d). 
\]  
Then there exists a global finite-energy BD entropy weak solution $(f,\rho,u)$ to \eqref{main_sys}  in the sense of Definition \ref{def:weak-sol}.  Moreover, any such solution  satisfies the density bounds \eqref{eq:pos-nm-c} and \eqref{eq:A2-dc} required in the relaxation argument. Consequently, in the diffusionless case $\sigma=0$ any such solution flocks and synchronizes in the sense of \eqref{eq:cond-rel-zero}. In the diffusive case $\sigma>0$,  any such solution flocks and relaxes in the sense of \eqref{eq:cond-rel-pos} provided the initial condition has finite higher moment:
\[
f_0\in L^2_q(\T^d\times\R^d) \quad\text{for some }q\geq d+4,
\]
in which case the kinetic component regularizes instantaneously: for every $0<t_0<T<\infty$ and every $m\in\N$,
\[
f\in C([t_0,T];H_q^m(\T^d\times\R^d)), \quad \pa_t f\in C([t_0,T];H_{q-3}^{m-1}(\T^d\times\R^d)).
\]
\end{theorem}

\begin{remark}
The assumptions in Theorem \ref{thm:weak-exist} imply 
\[
\calE_{\rm free}(0)+\calE_{\rm BD}(0)<\infty. 
\]
The singular part of the BD energy requires
\[
\nabla\rho_0^{\alpha-\frac12}\in L^2(\T^d) \quad\text{if }\alpha\neq\frac12, \quad \nabla\log\rho_0\in L^2(\T^d) \quad\text{if }\alpha=\frac12.
\]
This condition also rules out vacuum in the appropriate sense. In the one-dimensional cases $0<\alpha\le\frac12$, the singular BD control, together with the mass constraint, implies that the initial density is bounded away from both zero and infinity. In the two-dimensional borderline case $d=2$, $\alpha=\frac12$, the condition on $\nabla\log\rho_0$ excludes vacuum in the above sense, although no pointwise positive lower bound is asserted.
\end{remark}

\begin{remark}[Diffusive case]\label{rem:diffusive-case}
Although Theorem \ref{thm:weak-exist} is stated for every $\sigma\geq0$, we give the detailed proof in the more difficult diffusionless case $\sigma=0$, where the kinetic component is naturally allowed to be measure-valued. In the diffusive case $\sigma>0$, the same construction extends to the finite-entropy class associated with $H$. More precisely, if
\[
f_0\in L^1_2\cap L\log L(\T^d\times\R^d),
\]
then the kinetic entropy estimate gives
\[
f\in L^\infty_{\rm loc} \lt(0,\infty;L^1_2\cap L\log L(\T^d\times\R^d)\rt).
\]
The fluid compactness argument and the BD entropy estimates are unchanged, while the velocity diffusion contributes the nonnegative kinetic entropy dissipation contained in $\calD_{\rm drag}$.
The hypoellipticity of the kinetic equation is addressed  Section \ref{ssec:diff-ext}.
\end{remark}

%
%
%
%
%

\subsection{Related works}

Kinetic--fluid systems with Brinkman-type drag forces have been studied in a variety of settings. We begin by reviewing results concerning existence and well-posedness. In the incompressible case, global weak solutions for Vlasov--Stokes, Vlasov--Navier--Stokes, and Vlasov--Fokker--Planck--Navier--Stokes systems were constructed in \cite{BDGM09,CKL11,Ham98,MV07}. Further developments include weak solutions in time-dependent domains \cite{BGM17}, inhomogeneous fluid models \cite{WY15}, and uniqueness issues in two dimensions \cite{HMM20uni}. These works provide a substantial theory for incompressible kinetic--fluid systems, whereas the compressible case presents additional difficulties due to the coupling with the fluid density and the possible occurrence of vacuum.

For compressible kinetic--fluid systems, several results are available under different structural assumptions. The compressible Vlasov--Fokker--Planck--Navier--Stokes system with constant viscosities was studied in \cite{MV08}. Global classical or strong solutions near equilibrium were obtained in \cite{CKL13,LMW17}. For the diffusionless Vlasov--compressible Navier--Stokes system, finite-time blow-up phenomena and singularity formation for classical solutions were studied in \cite{Cho17,CJ22}. There are also results for compressible kinetic--fluid systems with density-dependent drag or viscosity coefficients. In one space dimension, global well-posedness for the compressible Navier--Stokes--Vlasov system with the drag force $\rho(u-v)$ and viscosity $\mu(\rho)=\mu_0+\mu_1\rho^\beta$ was established in \cite{LS21}. For the corresponding Vlasov--Fokker--Planck coupling with density-dependent viscosity and density-dependent drag coefficient, global weak solutions, uniqueness, and regularity were obtained in \cite{LS22}. The present work differs from these results in that we construct finite-energy weak solutions for a multidimensional compressible fluid with degenerate density-dependent viscosities and an averaged Brinkman force.

We next turn to large-time behavior and relaxation. In the diffusionless case, the relevant relaxation mechanism is velocity alignment rather than convergence to a Maxwellian. For Vlasov--Navier--Stokes type systems, conditional alignment estimates and large-time behavior were obtained in \cite{CCK16,CK15}. Stability and relaxation for incompressible Vlasov--Navier--Stokes systems in bounded or periodic domains were studied in \cite{GHKM18,HMM20,HK22,EHKM21}. In small-data or strong-solution frameworks, exponential relaxation toward mono-kinetic states was established on the torus and in the whole space \cite{HMM20,HK22,Dan26}. More recently, large-time asymptotics of finite-energy weak solutions on the two-dimensional torus were obtained in \cite{DS26}, with algebraic decay for general finite-energy data and exponential decay under a smallness assumption. For compressible Vlasov--Navier--Stokes systems, exponential alignment estimates were proved in an a priori framework in \cite{Cho16}; related one-dimensional asymptotic and large-time results can be found in \cite{LS21,LS22,LSZ25}.

When velocity diffusion is present, the kinetic equation has a Fokker--Planck structure and relaxation toward Maxwellian equilibria can be studied through the hypocoercivity analysis \cite{TV2000,Villani}, adopted to Fokker-Planck models with drift and non-homogeneous diffusion as in \cite{Shv-CPAM,Shv-EA}. Classical solutions near equilibrium and decay estimates for Navier--Stokes--Vlasov--Fokker--Planck type systems were studied in \cite{GHMZ10,CDM11,LMW17}. In one-dimensional compressible settings, large-time convergence for Vlasov--Fokker--Planck couplings was obtained in \cite{LS22}. In contrast, in the diffusionless regime the kinetic distribution is not expected to converge strongly to a Maxwellian; instead, one seeks velocity alignment and, in general, mono-kinetic behavior. The present paper treats the diffusionless and diffusive regimes within the same averaged kinetic--fluid framework and proves quantitative relaxation at the level of finite-energy entropy solutions.

%
%
%
%
%

\subsection{Difficulties and strategy of the proof}

We explain the main difficulties and the strategy of the proof. The analysis has two closely related parts. The first part concerns the quantitative relaxation of sufficiently regular solutions. The second part concerns the construction of finite-energy entropy weak solutions and the verification that they satisfy the density conditions required in the relaxation argument.

%
%
%
%
%

\subsubsection{Relaxation and density coercivity}
The starting point of the relaxation analysis is the energy-dissipation structure of the coupled system. The drag dissipation controls the relaxation of the particle velocity toward the averaged fluid velocity, while the viscous dissipation controls the fluid velocity gradient. However, these physical dissipation terms do not directly give a coercive control of the density fluctuation $\rho-1$. This is a standard difficulty in the study of large-time behavior for compressible fluid models. Following the classical strategy for compressible Navier--Stokes type systems (see, for instance \cite{CJ21AML, CKKT25}), we introduce a compensating
functional associated with the elliptic problem
\[
-\Delta\Phi=\rho-1,\quad \intt \Phi\,\dx=0.
\]
The time derivative of this functional produces the missing density coercivity, and this coercive contribution is then combined with the physical energy dissipation.

This compensating argument itself is classical, and the main issue is not the introduction of the elliptic corrector. The essential point is to make the
argument work in a kinetic--fluid system with degenerate density-dependent viscosities and only finite-energy information on the fluid velocity. In many
large-time analyses, one assumes uniform pointwise bounds on the density, namely upper and lower bounds of the form
\[
0<c\le \rho(t,x)\le C.
\]
Such assumptions are very strong and are not naturally available at the level of entropy weak solutions considered in this paper. We thus formulate the
conditional relaxation estimate under weaker quantitative density assumptions: uniform positive and negative moment bounds \eqref{eq:pos-nm-c}
\[
\sup_{t\ge0}\intt \lt(\rho(t,x)^p+\rho(t,x)^{-s}\rt)\dx\le M
\]
and the weighted Muckenhoupt condition \eqref{eq:A2-dc}
\[
\sup_{t\ge0}[\rho(t)^\alpha]_{\calA_2}\le M.
\]
These assumptions encode enough information on large-density regions and near-vacuum regions to close the relaxation argument, while avoiding uniform
pointwise bounds on $\rho$.

The introduction of the $\calA_2$ framework is particularly important in two space dimensions. In one dimension, the BD entropy gives strong enough control
on the density to recover suitable upper and lower bounds. In two dimensions, however, such pointwise control is in general unavailable from the BD entropy.
Instead, the weaker $\calA_2$ condition allows us to use weighted Calder\'on--Zygmund estimates to the elliptic corrector and thereby control the viscous contribution in the compensating-functional estimate.  This is one of the main points of the long-time analysis in the paper.

The final step of the relaxation argument depends essentially on whether kinetic diffusion is present. In the diffusionless case $\sigma=0$, the kinetic part of the modulated energy is directly controlled by the drag dissipation, and the density assumptions above are sufficient to recover the full coercivity of the energy law. When $\sigma>0$, however, the kinetic free energy is a relative entropy, whereas the physical dissipation provides only Fisher information in the velocity variable. Since there is no direct diffusion in the spatial variable, this partial Fisher information does not by itself control the global relative entropy. The missing spatial coercivity must therefore be recovered through the hypocoercive interaction between transport and velocity diffusion.

To achieve this, we adapt Villani's $(A^*A+B)$-hypocoercivity framework \cite{Villani}. After shifting the velocity variable by the kinetic momentum, the kinetic equation takes the variable-coefficient Fokker--Planck form
\begin{equation}\label{e:fpgenericintro}
 h_t  + \st(t,x) A^* A h + Bh + A^*( \rmb(t,x) h) = 0,
\end{equation}
where both the thermalization coefficient $\st = \st(t,x)>0$ and the drift $\rmb = \rmb(t,x)\in \R^d$ are determined by the evolving fluid variables. This is not a direct application of the classical hypocoercivity theory: the coefficient $\st$ is spatially nonhomogeneous, while the drift contains the locally averaged kinetic--fluid interaction. At the Fisher-information level, these features generate additional commutator and drift contributions that must be controlled in terms of the physical dissipation.

We overcome this difficulty by deriving coupled differential estimates for the velocity, mixed, and spatial Fisher informations
\[
\calI_{vv},\quad \calI_{xv},\quad \calI_{xx},
\]
and combining them with the modulated energy and the density compensator. A key ingredient is a cancellation in the drift contribution, manifested in \eqref{e:J234=0}, which prevents the appearance of uncontrolled higher-order terms. Although a related cancellation was used for particular alignment models in \cite{Shv-EA}, we show that it extends to the much broader class of nonhomogeneous Fokker--Planck equations \eqref{e:fpgenericintro}. This allows us to close the full Fisher-information estimate and, together with the weighted density coercivity developed above, yields the exponential relaxation of the coupled diffusive system. The detailed hypocoercivity estimates are presented in Section \ref{s:hypo}.

%
%
%
%
%

\subsubsection{Bresch--Desjardins entropy and weighted density estimates}
The next challenge is to show that the density assumptions used in the conditional relaxation theorem are not artificial. The free energy estimate controls $\sqrt\rho u$, the pressure potential, and the physical dissipation, but it does not provide enough compactness or quantitative control of the density near vacuum. Here, the special algebraic relation between the shear and bulk viscosities becomes crucial. It yields the Bresch--Desjardins entropy estimate, which provides additional control on suitable powers of the density; see, for instance, \cite{BD07,BDZ15,BVY22,VY16}.

In the present kinetic--fluid system, this entropy estimate has to be combined with the kinetic energy or entropy inequality and with the averaged drag dissipation. This coupling creates extra terms that are absent in the pure compressible Navier--Stokes theory. The conservative form of the averaged Brinkman force is essential here: it is compatible with the total momentum balance and with the drag dissipation appearing in the free energy inequality. After combining the free energy and BD entropy estimates, we obtain uniform bounds that imply the density conditions required in the conditional relaxation theorem in the admissible low-dimensional regimes. Thus the BD entropy plays a dual role: it provides compactness for the construction of weak solutions and, at the same time, verifies the quantitative density framework needed for large-time relaxation.

%
%
%
%
%

\subsubsection{Entropy weak solutions and relaxation}
The weak-solution theory requires a careful approximation scheme that is compatible with the degenerate viscosity structure. Since the viscosity coefficients are density-dependent and may vanish near vacuum, the standard compactness theory for compressible Navier--Stokes equations with constant viscosity is not directly applicable. Moreover, the approximation scheme must preserve enough of the BD structure to obtain estimates that are uniform in the approximation parameters. For this reason, we use a BD-compatible regularization in the spirit of the approximation schemes developed for degenerate compressible Navier--Stokes equations, in particular \cite{BDZ15,BVY22}.

There is an additional difficulty caused by the kinetic coupling. In the regularized system, only the Brinkman force is averaged. The convective term in the fluid equation and the density-dependent viscosity terms are not simply replaced by fully mollified expressions. Thus one has to keep track of several nonlinear terms simultaneously: the compressible transport structure, the degenerate viscous stress, the kinetic drag force, and the approximate BD entropy. Thus, the compactness argument requires a careful combination of the effective velocity mechanism, compactness lemmas for the density, and stability properties of the averaged force. At the approximate level, we derive both the free energy inequality and a combined BD entropy inequality. These estimates yield uniform bounds for the fluid density, the effective velocity, and the kinetic second moment. The averaged form of the drag force is crucial in passing to the limit in the kinetic--fluid coupling. In particular, the lower bound on $\rho_\e$ ensures that the averaged velocity $u_\e^\e$ remains well defined even in the presence of vacuum, while the conservative structure of the force allows the energy inequality to pass to the limit.

Finally, the relaxation estimate must be justified at the level of entropy weak solutions. For smooth solutions, the compensating functional can be differentiated directly in time, but this pointwise computation is no longer available for weak solutions. We hence use an integrated version of the argument and pass to the limit by relying on the free energy inequality, the BD entropy inequality, and the density bounds obtained from the compactness argument. This yields exponential relaxation for entropy weak solutions in the diffusionless case. When $\sigma>0$, the weak construction extends to the finite-entropy
setting, and the additional hypoelliptic regularity is used to complete
the corresponding hypocoercive relaxation argument in Section
\ref{ssec:diff-ext}.

%
%
%
%
%

\subsection{Organization of the paper}

The rest of the paper is organized as follows. Section \ref{sec:energy} proves the conditional relaxation theorem, Theorem \ref{thm:main1}. Section
\ref{sec:bd-estimate} derives the BD entropy estimates and verifies the density assumptions \eqref{eq:pos-nm-c}--\eqref{eq:A2-dc} in the admissible low-dimensional regimes. Section \ref{sec:weak-sol} proves the existence assertion of Theorem \ref{thm:weak-exist} in the diffusionless case $\sigma=0$. Section \ref{sec:weak-relax-diff} proves the corresponding relaxation estimate and completes the existence and relaxation assertions in the diffusive case $\sigma>0$.  The appendices contain the construction of approximate initial data and the auxiliary augmented entropy estimate.

%
%
%
%
%

\section{Conditional relaxation}\label{sec:energy}

In this section we prove the conditional relaxation estimate stated in Theorem \ref{thm:main1}. We work at the level of sufficiently regular solutions and use a modulated energy method. The physical dissipation provides direct control of the kinetic and fluid velocity fluctuations and of the relative momentum between the two phases. To control the density fluctuation, we introduce a compensating functional, which recovers the coercivity not directly supplied by the physical dissipation. The density assumptions in Theorem \ref{thm:main1} enter exactly at this stage.

For a center $b\in\R^d$, we define the centered kinetic functional by
\[
\calH^b(t)
:=
\begin{cases}
\displaystyle \sigma\inttr f\log\frac{f}{\mu_{1,b,\sigma}}\,\dx\dv, & \sigma>0,\\[3mm]
\displaystyle \frac12\inttr |v-b|^2 f\,\dx\dv, & \sigma=0.
\end{cases}
\]
In particular, $\calH^{\bar v}$ is the relative entropy with respect to the Maxwellian centered at the kinetic momentum $\bar v$ when $\sigma>0$, while it is the kinetic fluctuation energy when $\sigma=0$. The modulated energy is defined by
\bq\label{e:moden}
\calE(t) := \calH^{\bar v}(t) + \frac12\intt \rho |u-\bar{m}|^2\,\dx + \frac14|\bar{v}-\bar{m}|^2 + \intt \pi(\rho)\,\dx .
\eq
We will prove the following energy identity:
\bq\label{e:slaw}
\frac{\dd}{\dt}\calE+\calD=0,
\eq
where
\[
\calD:=\calD_{\rm drag}+\calD_{\rm vis}.
\]

%
%
%
%
%
\subsection{Conservation laws}\label{ssec:cons}

Let us start by observing a few conservation laws that will be essential to the energy analysis. First, the traditional mass conservation, 
\[
\frac{\dd}{\dt}\inttr f\,\dx\dv = 0 , \quad \frac{\dd}{\dt} \intt \rho\,\dx = 0.
\]
Next, let us denote the contaminant and fluid momenta by
\[
\bar{v}(t) =\inttr vf\,\dx\dv, \quad \bar{m}(t) =\intt \rho u\,\dx .
\]
Directly from the system \eqref{main_sys} we obtain
\begin{equation*}\label{}
\begin{split}
\frac{\dd}{\dt}\bar{v}(t) & = -\inttr \rho_\e (v-u_\e^\e)f\,\dx\dv, \\
\frac{\dd}{\dt}\bar{m}(t) &  = \inttr \rho_\e(v-u_\e^\e)f\,\dx\dv.
\end{split}
\end{equation*}
Hence, the total momentum is conserved:
\[
\frac{\dd}{\dt}\lt(\bar{v}(t)+\bar{m}(t)\rt)=0.
\]

Next, we have conservation of the residual momentum energy
\begin{equation}\label{e:resid}
\frac{\dd}{\dt}\lt(\frac14|\bar{v}-\bar{m}|^2  -\frac12\lt(|\bar{v}|^2+|\bar{m}|^2\rt) \rt)  =\frac12(\bar{v}-\bar{m})\cdot(\bar{v}'-\bar{m}')  -\lt(\bar{v}\cdot \bar{v}'+\bar{m}\cdot \bar{m}'\rt) =0.
\end{equation}

%
%
%
%
%

 \subsection{Energy identity}

We start with an observation that the energy $\calE$ differs from the free energy 
$\calE_{\rm free}=H+E$ by the residual momentum contribution. Indeed, by the definitions of $H$ and $\calH^{\bar v}$, and by the mass normalization, we have
\[
\calH^{\bar v}=H-\frac12|\bar v|^2.
\]
Similarly,
\[
\frac12\intt \rho|u-\bar m|^2\,\dx = \frac12\intt \rho|u|^2\,\dx - \frac12|\bar m|^2.
\]
Thus,
\[
\calE = \calE_{\rm free} + \frac14|\bar v-\bar m|^2 - \frac12\lt(|\bar v|^2+|\bar m|^2\rt).
\]
The last two momentum terms form a conserved quantity by \eqref{e:resid}. Thus, it is enough to derive the free energy identity.

Recall that
\[
\calE_{\rm free}=H+E,
\]
where $H$ denotes the kinetic entropy when $\sigma>0$ and the kinetic energy when $\sigma=0$, and
\[
E = \intt \lt[ \frac12\rho|u|^2+\pi(\rho) \rt]\dx .
\]

Let us start with the entropy.   When $\sigma>0$, using
\[
\nabla_v\log\frac{f}{\mu_{1,0,\sigma}} = \frac{\nabla_v f}{f}+\frac{v}{\sigma},
\]
the kinetic equation in \eqref{main_sys} gives
\begin{equation*}\label{}
\begin{split}
\frac{\dd}{\dt}H &= -\inttr \frac{\rho_\e}{f} \lt( \sigma \nabla_v f + v f \rt) \cdot \lt(\sigma\nabla_v f+(v-u_\e^\e)f\rt) \,\dx\dv \\
& = -\inttr \frac{\rho_\e}{f} | \sigma\nabla_v f+(v-u_\e^\e)f |^2 \,\dx\dv - \inttr \rho_\e u_\e^\e \cdot \lt(\sigma\nabla_v f+(v-u_\e^\e)f\rt) \,\dx\dv\\
& = -\inttr \frac{\rho_\e}{f} | \sigma\nabla_v f+(v-u_\e^\e)f |^2 \,\dx\dv - \inttr (\rho u)_\e \cdot (v-u_\e^\e)f \,\dx\dv.
\end{split}
\end{equation*}
When $\sigma=0$, the same identity follows by testing the kinetic equation
with $\frac{|v|^2}2$. Hence, in both cases,
\bq\label{eq:kin-free-sig}
\frac{\dd}{\dt}H  + \calD_{\rm drag}=  - \inttr (\rho u)_\e \cdot (v-u_\e^\e)f \,\dx\dv.
\eq

We next derive the fluid energy balance. First, with regard to the pressure potential we observe the following properties:
\[
\pi(1)=\pi'(1)=0, \quad \pi''(r)=\frac{p'(r)}{r}=\gamma r^{\gamma-2}.
\]
In particular, $\pi$ is nonnegative and strictly convex on $(0,\infty)$. Moreover,
\[
r\pi'(r)-\pi(r)=p(r)-p(1)=r^\gamma-1.
\]
Thus, along smooth solutions of the continuity equation, we have
\bq\label{eq:h-ide}
\frac{\dd}{\dt}\intt \pi(\rho)\,\dx=-\intt p(\rho)\nabla\cdot u\,\dx.
\eq

Multiplying the fluid momentum equation in \eqref{main_sys} by $u$ and using \eqref{eq:h-ide}, we get
\[
\frac{\dd}{\dt} E + \intt\lt(2\nu(\rho)|D u|^2+\lambda(\rho)|\nabla \cdot u|^2\rt)\dx  = \intt \rho u\cdot \lt(\intr (v-u_\e^\e)f\,\dv\rt)_\e \dx .
\]
By the symmetry of the mollifier and the identity $(\rho u)_\e=\rho_\e u_\e^\e$, the right-hand side becomes
\[
\intt \rho u\cdot \lt(\intr (v-u_\e^\e)f\,\dv\rt)_\e \,\dx = \intt (\rho u)_\e\cdot \lt(\intr (v-u_\e^\e)f\,\dv\rt)\dx  .
\]
Adding this identity to \eqref{eq:kin-free-sig}, we obtain 
\[
\ddt \calE_\free + \calD_{\rm drag} + \calD_{\rm vis} = 0.
\]
Since $\calD=\calD_{\rm drag}+\calD_{\rm vis}$ and the residual momentum contribution is conserved by \eqref{e:resid}, this proves \eqref{e:slaw}.

%
%
%
%
%

 \subsection[Partial coercivity of the dissipation]{Partial coercivity of the dissipation
}\label{ssec:coer-ex-den}

We first show that the dissipation controls all components of the modulated energy except for the density fluctuation. Throughout this subsection, we use the exponents $r,p,s$ introduced in Theorem \ref{thm:main1}. In particular,
\[
r>\frac{2d}{d+2} \quad \text{if } d\ge2, \quad \frac2p+\frac2{r^\star}\le1, \quad s\ge\frac{\alpha r}{2-r}.
\]

Let us note that since $c_\e < \rho_\e <C_\e$ pointwise, we can remove $ \rho_\e$ altogether from $\calD$. As a result the total dissipation dominates two distinct components -- the viscous component coming from the fluid and the Fisher information (which degenerates into the modulated kinetic energy when $\sigma = 0$):
\begin{equation*}\label{}
\begin{split}
\calD &\gtrsim \calI^{u_\e^\e}_{vv} + \calD_\vis, \\
\calI^{u_\e^\e}_{vv}  & = \inttr \frac{\lt|\sigma\nabla_v f+(v-u_\e^\e)f\rt|^2}{f}
\,\dx\dv, \\
\calD_\vis & = \intt \lt(2\nu(\rho)|D u|^2+\lambda(\rho)|\nabla\cdot u|^2\rt)\dx.
\end{split}
\end{equation*}

We begin with a weighted viscous coercivity estimate. When $d\ge2$, we have
\[
2\nu(\rho)|D u|^2+\lambda(\rho)|\nabla\cdot u|^2=2\rho^\alpha|D^0u|^2+2\lt(\alpha-1+\frac1d\rt)\rho^\alpha|\nabla\cdot u|^2,
\]
where
\[
D^0u := D u-\frac1d(\nabla\cdot u)\mathbb I.
\]
Since $\alpha\ge1 - \frac1d$, it follows that
\[
\calD_{\rm vis}\ge2\intt \rho^\alpha|D^0u|^2\,\dx.
\]
Using the trace-free Korn inequality on $\T^d$ and H\"older's inequality, we obtain
\[
\|\nabla u\|_{L^r(\T^d)}^2 \le C\|D^0u\|_{L^r(\T^d)}^2  \le C\lt(\intt \rho^\alpha|D^0u|^2\,\dx\rt)\lt(\intt \rho^{-\frac{\alpha r}{2-r}}\,\dx\rt)^{\frac{2-r}{r}}  \le C_M\calD_{\rm vis}.
\]
Here, the last inequality follows from \eqref{eq:pos-nm-c} and the condition \eqref{eq:mom-exp}.

When $d=1$, we instead use
\[
\calD_{\rm vis}=2\alpha\int_{\T}\rho^\alpha|\pa_xu|^2\,\dx.
\]
The same H\"older estimate gives
\[
\|\pa_xu\|_{L^r(\T)}^2\le C\lt(\int_{\T}\rho^\alpha|\pa_xu|^2\,\dx\rt)\lt(\int_{\T}\rho^{-\frac{\alpha r}{2-r}}\,\dx\rt)^{\frac{2-r}{r}}\le C_M\calD_{\rm vis}(t).
\]
Thus, in every dimension $d\ge1$, we have
\bq
\|\nabla u\|_{L^r(\T^d)}^2\le C_M\calD_{\rm vis}.
\label{eq:wv-coe-all}
\eq

\begin{lemma} 
Under the assumptions of Theorem \ref{thm:main1}, we have
\bq\label{eq:coer-ex-den}
\calI_{vv}^{\bar{v}} + \intt \rho|u-\bar{m}|^2\,\dx + |\bar{v}-\bar{m}|^2 \le C_{\e,M}\calD.
\eq
\end{lemma}

\begin{proof} Define
\[
\bar{u}(t):=\intt u(t,x)\,\dx.
\]

When $d\ge2$, the Sobolev--Poincar\'e inequality and \eqref{eq:wv-coe-all} yield
\[
\|u-\bar{u}\|_{L^{r^\star}(\T^d)}^2\le C\|\nabla u\|_{L^r(\T^d)}^2\le C_M\calD_{\rm vis}.
\]
When $d=1$, we use the one-dimensional Sobolev--Poincar\'e inequality
\[
\|u-\bar{u}\|_{L^\infty(\T)}^2\le C\|\pa_xu\|_{L^r(\T)}^2\le C_M\calD_{\rm vis}.
\]
Thus, using the convention $r^\star=\infty$ when $d=1$, we obtain
\bq\label{eq:w-sob-velo}
\|u-\bar{u}\|_{L^{r^\star}(\T^d)}^2\le C_M\calD_{\rm vis}.
\eq

The condition $\frac2p+\frac2{r^\star}\le1$ implies that
\[
\frac1p+\frac2{r^\star}\le1, \text{ or equivalently } 2p'\le r^\star.
\]
Since $|\T^d|=1$, it follows from \eqref{eq:w-sob-velo} that
\[
\intt \rho|u-\bar{u}|^2\,\dx \le\|\rho\|_{L^p(\T^d)}\|u-\bar{u}\|_{L^{2p'}(\T^d)}^2  \le C_M\calD_{\rm vis}.
\]
Consequently,
\[
\intt \rho|u-\bar{m}|^2\,\dx+ |\bar{m}-\bar{u}|^2=\intt \rho|u-\bar{u}|^2\,\dx \le C_M\calD_{\rm vis}.
\]

We now estimate the difference between the averaged kinetic and fluid momenta.  We have
\[
\bar{v}-\bar{m}=\inttr (\sigma \nabla_v f + (v-u_\e^\e)f )\,\dx\dv+\intt \rho_f(u_\e^\e-\bar{m})\,\dx.
\]
Hence,
\begin{equation*}\label{}
\begin{split}
|\bar{v}-\bar{m}|^2 &  \leq \lt| \inttr \frac{\sigma \nabla_v f + (v-u_\e^\e)f }{\sqrt{f}} \sqrt{f}\, \dx\dv \rt|^2 + \lt|\intt \rho_f(u_\e^\e-\bar{m})\,\dx \rt|^2  \leq \calI_{vv}^{u_\e^\e} + \intt \rho_f |u_\e^\e-\bar{m}|^2\, \dx.
\end{split}
\end{equation*}
Let us estimate the last term,
\[
u_\e^\e-\bar{m}=\frac{(\rho u)_\e- \rho_\e \bar{m}}{\rho_\e}=\frac{\lt(\rho(u- \bar{m})\rt)_\e}{\rho_\e}.
\]
Jensen's inequality gives
\[
|u_\e^\e-\bar{m}|^2\le\frac1{\rho_\e}\lt(\rho|u-\bar{m}|^2\rt)_\e.
\]
Since $\rho_\e$ is bounded from below, 
\[
|u_\e^\e-\bar{m}|^2\lesssim (\rho|u-\bar{m}|^2)_\e.
\]
Thus,
\begin{equation} \label{eq:w-ue-uc}
\intt \rho_f|u_\e^\e-\bar{m}|^2\,\dx \lesssim \intt \rho_f\lt(\rho|u-\bar{m}|^2\rt)_\e\,\dx  =\intt (\rho_f)_\e\rho|u-\bar{m}|^2\,\dx \le C_{\e,M}\calD_{\rm vis}.
\end{equation}

Finally, from all of the above,
\[
\calI_{vv}^{\bar{v}}  \leq \calI_{vv}^{u_\e^\e} + \intt \rho_f|u_\e^\e-\bar{m}|^2\,\dx+C|\bar{m}-\bar{v}|^2 \le C_{\e,M}\calD.
\]
This completes the proof.
\end{proof}

The estimate \eqref{eq:coer-ex-den} shows that the dissipation $\calD$ controls two terms of the modulated energy directly, namely $\intt \rho|u-\bar{m}|^2\,\dx + |\bar{v}-\bar{m}|^2$, and at least for diffusionless case $\sigma = 0$, it also controls the kinetic part of the energy $\inttr |v-\bar{v}|^2f\,\dx\dv$. It does not yet control the potential part $\intt \pi(\rho)\, \dx$, and it lacks Fisher information in $x$-variable to control the relative entropy $H$ if $\sigma>0$.  Both deficiencies will be filled with the use of hypocoercivity technique by introducing cross-product terms that involve kinetic and potential components.  First we discuss a few properties of the pressure potential and an auxiliary potential that will be used in the hypocoercivity analysis.

%
%
%

\subsection{Pressure potential}
Let us define an auxiliary potential
\[
\pi_\aux(r):=(r^\gamma-1)(r-1), \ r\ge 0 .
\]
The following elementary comparison estimates will be used repeatedly.

\begin{lemma}\label{lem:pres-pot-est}
Let $\gamma>1$. Then, for every $r\ge0$,
\bq\label{eq:pressure-functional-comparison}
\pi_\aux(r)\geq   |r-1|^2, \quad \pi_\aux(r) \geq \pi(r).
\eq
With regard to the potential $\pi$, we have for all $\gamma>1$,
\bq\label{eq:pressure-potential-L1}
\|\rho-1\|_{L^1(\T^d)}^2\le C_\gamma\intt \pi(\rho)\,\dx.
\eq
Moreover, if $1<\gamma\le2$, then
\begin{equation}\label{e:pressure-potential-Lg}
\|\rho-1\|_{L^\gamma(\T^d)}^2\le C_{\gamma}\lt( \intt \pi(\rho)\,\dx + \lt[\intt \pi(\rho)\,\dx\rt]^{2/\gamma} \rt).
\end{equation}
If $\gamma\ge2$, then
\bq\label{eq:pressure-potential-L2}
\|\rho-1\|_{L^2(\T^d)}^2\le\intt \pi(\rho)\,\dx.
\eq
\end{lemma}

\begin{proof}
For $r\neq1$, we have
\[
\frac{\pi_\aux(r)}{|r-1|^2}=\frac{r^\gamma-1}{r-1}\geq 1,
\]
which proves the first inequality in \eqref{eq:pressure-functional-comparison}. To prove the second one, observe that
\[
(\gamma-1)\lt(\pi_\aux(r)-\pi(r)\rt)=r\lt[(\gamma-1)r^\gamma-\gamma r^{\gamma-1}+1\rt].
\]
Let
\[
G(r):=(\gamma-1)r^\gamma-\gamma r^{\gamma-1}+1.
\]
Then
\[
G'(r)=\gamma(\gamma-1)r^{\gamma-2}(r-1), \quad G(1)=0.
\]
Thus, $G$ is decreasing on $(0,1)$ and increasing on $(1,\infty)$, and hence $G(r)\ge G(1)=0$ for every $r\ge0$. This proves the second inequality in \eqref{eq:pressure-functional-comparison}. 

We next prove \eqref{eq:pressure-potential-L1}. Since $\frac{\pi(r)}{|\sqrt r-1|^2}$ extends to a strictly positive continuous function on $[0,\infty)$ and tends to infinity as $r\to\infty$, there exists a constant $c_\gamma>0$ such that $\pi(r)\ge c_\gamma|\sqrt r-1|^2$ for every $r\ge0$. Hence, using the mass constraint and the Cauchy--Schwarz inequality, we obtain
\[
\|\rho-1\|_{L^1(\T^d)}^2 \le\lt(\intt |\sqrt\rho-1|^2\,\dx\rt)\lt(\intt |\sqrt\rho+1|^2\,\dx\rt) \le C\intt |\sqrt\rho-1|^2\,\dx \le C_\gamma\intt \pi(\rho)\,\dx.
\]

Suppose now that $1<\gamma\le2$. The local quadratic behavior of $h$ near $r=1$ and its growth of order $r^\gamma$ as $r\to\infty$ imply that
\[
|r-1|^\gamma\le C_\gamma\lt(\pi(r)^{\frac{\gamma}{2}}+\pi(r)\rt)
\]
for every $r\ge0$. Since $|\T^d|=1$ and $\gamma/2\le1$, Jensen's inequality gives
\[
\|\rho-1\|_{L^\gamma(\T^d)}^\gamma\le C_\gamma\lt[\lt(\intt \pi(\rho)\,\dx\rt)^{\frac{\gamma}{2}}+\intt \pi(\rho)\,\dx\rt],
\]
which implies \eqref{e:pressure-potential-Lg}.

Finally, if $\gamma\ge2$, then $\pi(r)\ge|r-1|^2$ for every $r\ge0$, and \eqref{eq:pressure-potential-L2} follows by integration.
\end{proof}

\begin{remark}
If $0\le r\le\overline\rho$, the pressure potential is quantitatively comparable to the quadratic fluctuation. More precisely, if $1<\gamma\le2$, then
\[
\frac{\gamma}{2}\max\{1,\overline\rho\}^{\gamma-2}|r-1|^2\le \pi(r)\le|r-1|^2.
\]
If $\gamma\ge2$, then
\[
|r-1|^2\le \pi(r)\le\frac{\gamma}{2}\max\{1,\overline\rho\}^{\gamma-2}|r-1|^2.
\]
\end{remark}

%
%
%

\subsection{A compensating functional}\label{ssec:comp-func}

Let $\Phi=\Phi(t,x)$ be the unique zero-average solution of
\[
-\Delta\Phi=\rho-1, \quad \intt \Phi\,\dx=0,
\]
and define
\[
\calX:=\intt \rho(u-\bar{m})\cdot\nabla\Phi\,\dx.
\]
We first show that $\calX$ is controlled by the modulated energy:
\bq\label{eq:w-comp-equi}
|\calX|\le C_M\calE.
\eq
Indeed, by the Cauchy--Schwarz inequality,
\[
|\calX|\le\lt(\intt \rho|u-\bar{m}|^2\,\dx\rt)^{\frac12}\lt(\intt \rho|\nabla\Phi|^2\,\dx\rt)^{\frac12}.
\]
Moreover,
\[
\intt \rho|\nabla\Phi|^2\,\dx\le\|\rho\|_{L^p(\T^d)}\|\nabla\Phi\|_{L^{2p'}(\T^d)}^2.
\]
By \eqref{eq:pos-nm-c}, $\|\rho\|_{L^p(\T^d)} \leq M$ uniformly in time. It remains to estimate $\|\nabla\Phi\|_{L^{2p'}(\T^d)}$. 

Suppose first that $1<\gamma\le2$. By elliptic regularity,
\[
\|\nabla\Phi\|_{W^{1,\gamma}(\T^d)} \le C\|\rho-1\|_{L^\gamma(\T^d)}.
\]
If $\gamma<d$, the Sobolev embedding $W^{1,\gamma}(\T^d) \hookrightarrow L^{\frac{d\gamma}{d-\gamma}}(\T^d)$ and the condition \eqref{eq:pot-exp-c} yield
\[
\|\nabla\Phi\|_{L^{2p'}(\T^d)}^2 \le C\|\rho-1\|_{L^\gamma(\T^d)}^2 .
\]
When $\gamma\ge d$, the same conclusion follows from the corresponding
Sobolev embedding without any additional restriction on $p$.
Since
\[
\intt \pi(\rho(t))\,\dx\le\calE(t)\le\calE(0),
\]
we may apply Lemma \ref{lem:pres-pot-est} to conclude that 
\[
\|\rho-1\|_{L^\gamma(\T^d)}^2 \leq C \intt \pi(\rho)\,\dx \lesssim \calE.
\]

When $\gamma\ge2$, we instead use
\[
\|\nabla\Phi\|_{W^{1,2}(\T^d)} \le C\|\rho-1\|_{L^2(\T^d)}.
\]
If $d=1,2$, the embedding into $L^{2p'}(\T^d)$ holds for every finite exponent. If $d\ge3$, we use $W^{1,2}(\T^d) \hookrightarrow L^{\frac{2d}{d-2}}(\T^d)$. Since \eqref{eq:mom-exp} implies $2p'\le\frac{2d}{d-2}$, we conclude that
\[
\|\nabla\Phi\|_{L^{2p'}(\T^d)}^2 \le C\|\rho-1\|_{L^2(\T^d)}^2 \le C\intt \pi(\rho)\,\dx.
\]
Consequently,
\[
\intt \rho|\nabla\Phi|^2\,\dx\le C_M\intt \pi(\rho)\,\dx.
\]

We now compute the time derivative of the compensating functional. Set $z:=u-\bar{m}$. Then
\[
-\frac{\dd}{\dt}\calX(t)=I+II,
\]
where
\[
I:=-\intt \pa_t(\rho z)\cdot\nabla\Phi\,\dx, \quad II:=-\intt \rho z\cdot\pa_t\nabla\Phi\,\dx.
\]
Using the continuity and momentum equations in \eqref{main_sys}, we obtain
\[
I=\intt \lt[\nabla\cdot(\rho u\otimes z)+\nabla p(\rho)-\nabla\cdot\mathbb S(\rho,u)-\rho F_\e+\rho \bar{m}'\rt]\cdot\nabla\Phi\,\dx=:\sum_{i=1}^5 I_i.
\]

The pressure term yields the desired density coercivity. Since $\intt (\rho-1)\,\dx=0$, we have
\[
I_2 =\intt \nabla p(\rho)\cdot\nabla\Phi\,\dx =\intt \lt(\rho^\gamma-1\rt)(\rho-1)\,\dx =\intt \pi_\aux(\rho) \,\dx.
\]

We next estimate the convective term. Since $u=z+\bar{m}$, integration by parts gives
\[
I_1=-\intt \rho z\otimes z:\nabla^2\Phi\,\dx-\intt \rho \bar{m}\otimes z:\nabla^2\Phi\,\dx.
\]
By \eqref{eq:w-sob-velo} and the estimate for $|\bar{m}-\bar{u}|$, we have
\bq\label{eq:weighted-sobolev-centered-velocity}
\|z\|_{L^{r^\star}(\T^d)}^2\le C_M\calD_{\rm vis}.
\eq
For the first term, the condition \eqref{eq:mom-exp} and the elliptic estimate imply
\[
\lt|\intt \rho z\otimes z:\nabla^2\Phi\,\dx\rt| \le\|\rho\|_{L^p(\T^d)}\|z\|_{L^{r^\star}(\T^d)}^2\|\nabla^2\Phi\|_{L^p(\T^d)} \le C_M\calD_{\rm vis}.
\]
Here, we used
\[
\|\nabla^2\Phi\|_{L^p(\T^d)}\le C\|\rho-1\|_{L^p(\T^d)}\le C_M.
\]
For the second term, the same exponent condition gives
\[
\frac1p+\frac1{r^\star}+\frac12\le1.
\]
Hence,
\begin{align*}
\lt|\intt \rho \bar{m}\otimes z:\nabla^2\Phi\,\dx\rt|
&\le |\bar{m}|\|\rho\|_{L^p(\T^d)}\|z\|_{L^{r^\star}(\T^d)}\|\nabla^2\Phi\|_{L^2(\T^d)}\\
&\le C_M\calD_{\rm vis}^{\frac12}\|\rho-1\|_{L^2(\T^d)}\\
&\le C_M\calD_{\rm vis}+\frac18\intt \pi_\aux(\rho) \,\dx,
\end{align*}
where we used the boundedness of $\bar{m}$, the elliptic estimate, and \eqref{eq:pressure-functional-comparison}. Consequently,
\[
|I_1|\le C_M\calD_{\rm vis}+\frac18\intt \pi_\aux(\rho) \,\dx.
\]

We now turn to the viscous term. Integration by parts gives
\[
I_3=\intt \mathbb S(\rho,u):\nabla^2\Phi\,\dx.
\]
Since $\rho^\alpha\in\calA_2(\T^d)$, we apply the weighted Calder\'on--Zygmund estimate for periodic second-order Riesz transforms; see for instance \cite[Theorem 1.1]{AV12}. We also refer to \cite{Hyt12} for the sharp Euclidean estimate. Indeed, since $-\Delta\Phi=\rho-1$, we have $\partial_i\partial_j\Phi = -\partial_i\partial_j(-\Delta)^{-1}(\rho-1)$. Thus,
\[
\intt \rho^\alpha \lt|\partial_i\partial_j\Phi\rt|^2 \dx \le C_{[\rho^\alpha]_{\calA_2}} \intt \rho^\alpha|\rho-1|^2\,\dx.
\]
Summing over $i,j=1,\dots,d$ and using \eqref{eq:A2-dc}, we obtain
\[
\intt \rho^\alpha|\nabla^2\Phi|^2\,\dx \le C_M\intt \rho^\alpha|\rho-1|^2\,\dx.
\]
Since $\alpha\le\gamma-1$, we have
\[
\rho^\alpha|\rho-1|^2 \le C(\rho^\gamma-1)(\rho-1)
\]
for every $\rho\ge0$.  Hence,
\bq\label{eq:w-cz-g}
\intt \rho^\alpha|\nabla^2\Phi|^2\,\dx\le C_M\intt \pi_\aux(\rho) \,\dx.
\eq
Using the decomposition
\[
\mathbb S(\rho,u) = 2\rho^\alpha D^0u + 2\lt(\alpha-1+\frac1d\rt) \rho^\alpha(\nabla\cdot u)\mathbb I
\]
when $d\ge2$, we obtain
\begin{align*}
|I_3|
&\le 2\lt(\intt \rho^\alpha|D^0u|^2\,\dx\rt)^{\frac12} \lt(\intt \rho^\alpha|\nabla^2\Phi|^2\,\dx\rt)^{\frac12}\\
&\quad + 2\lt(\alpha-1+\frac1d\rt) \lt(\intt \rho^\alpha|\nabla\cdot u|^2\,\dx\rt)^{\frac12} \lt(\intt \rho^\alpha|\Delta\Phi|^2\,\dx\rt)^{\frac12}.
\end{align*}
Since $|\Delta\Phi|^2\le d|\nabla^2\Phi|^2$ and
\[
\calD_{\rm vis}  = 2\intt \rho^\alpha|D^0u|^2\,\dx + 2\lt(\alpha-1+\frac1d\rt) \intt \rho^\alpha|\nabla\cdot u|^2\,\dx,
\]
we find
\[
|I_3| \le C_{\alpha,d}\calD_{\rm vis}^{\frac12} \lt( \intt \rho^\alpha|\nabla^2\Phi|^2\,\dx \rt)^{\frac12}.
\]
The same estimate follows directly in one dimension from
\[
\mathbb S(\rho,u)=2\alpha\rho^\alpha\partial_xu,
\quad
\calD_{\rm vis}  = 2\alpha\int_{\T}\rho^\alpha|\partial_xu|^2\,\dx.
\]
Using \eqref{eq:w-cz-g}, we conclude that
\[
|I_3| \le C_M\calD_{\rm vis}^{\frac12}\lt(\intt \pi_\aux(\rho)\, \dx\rt)^{\frac12} \le C_M\calD_{\rm vis} + \frac18\intt \pi_\aux(\rho) \,\dx.
\]
 
For the coupling term,  
\[
I_4=-\intt \rho F_\e\cdot\nabla\Phi\,\dx=-\intt F\cdot(\rho\nabla\Phi)_\e\,\dx \leq \|F\|_1 \|(\rho\nabla\Phi)_\e\|_\infty.
\]
As to the drag force we can write 
\begin{equation}\label{e:Fsigma}
F = \int_{\R^d} (\sigma \nabla_v f + (v-u_\e^\e)f )\,\dv.
\end{equation}
So,
\begin{equation}\label{e:FL1}
 \|F\|_1\leq \inttr \frac{|\sigma \nabla_v f + (v-u_\e^\e)f |}{\sqrt{f}} \sqrt{f}\, \dx\dv  \leq (\calI_{vv}^{u_\e^\e})^{\frac12}.
\end{equation}
As to the compensator, since $p\ge2$,
\[
\|(\rho\nabla\Phi)_\e\|_{L^\infty(\T^d)}\le\|\theta_\e\|_{L^\infty(\T^d)}\|\rho\nabla\Phi\|_{L^1(\T^d)}\le C_{\e,M}\|\rho-1\|_{L^2(\T^d)}.
\]
Thus,
\[
|I_4|\leq C_{\e,M}\calI_{vv}^{u_\e^\e}+ \frac18\intt \pi_\aux(\rho) \,\dx.
\]
 
The remaining term $I_5$ vanishes identically. Indeed,
\[
I_5=\bar{m}'\cdot\intt \rho\nabla\Phi\,\dx.
\]
Since $\rho-1=-\Delta\Phi$, periodicity gives
\[
\intt \rho\nabla\Phi\,\dx=\intt (\rho-1)\nabla\Phi\,\dx=-\intt \Delta\Phi\nabla\Phi\,\dx=0.
\]
Hence,
\[
I_5=0.
\]

It remains to estimate $II$. Differentiating
\[
-\Delta\Phi=\rho-1
\]
in time and using the continuity equation, we obtain
\[
-\Delta\pa_t\Phi=-\nabla\cdot(\rho u)=-\nabla\cdot(\rho z)-\nabla\cdot\lt((\rho-1)\bar{m}\rt).
\]
The exponent condition in \eqref{eq:mom-exp} implies
\[
\frac1p+\frac1{r^\star}\le\frac12.
\]
Hence, by \eqref{eq:weighted-sobolev-centered-velocity},
\[
\|\rho z\|_{L^2(\T^d)}\le\|\rho\|_{L^p(\T^d)}\|z\|_{L^{r^\star}(\T^d)}\le C_M\calD_{\rm vis}^{\frac12}.
\]
Using the elliptic estimate and the boundedness of $\bar{m}$, we find
\[
\|\nabla\pa_t\Phi\|_{L^2(\T^d)}\le C_M\calD_{\rm vis}^{\frac12}+C_M\|\rho-1\|_{L^2(\T^d)}.
\]
This together with applying Young's inequality yields
\[
|II|\le\|\rho z\|_{L^2(\T^d)}\|\nabla\pa_t\Phi\|_{L^2(\T^d)}\le C_M\calD_{\rm vis} + C_M\calD_{\rm vis}^{\frac12}\lt(\intt \pi_\aux(\rho) \,\dx\rt)^{\frac12} \le C_M\calD_{\rm vis} +\frac18\intt \pi_\aux(\rho) \,\dx.
\]

Combining the above estimates, we obtain
\bq\label{eq:w-comp-est}
\frac{\dd}{\dt}\calX\le - c \intt \pi_\aux(\rho)\, \dx + C\calD.
\eq

Incorporating a small proportion  of $\calX$ into the modulated energy, we define for $\delta>0$, 
\[
\calE^\delta:=\calE+\delta\calX.
\]
By \eqref{eq:w-comp-equi}, we get $\calE^\delta \sim\calE $ provided that $\delta>0$ is sufficiently small. Yet, the energy law reads
\begin{equation}\label{e:enX}
\ddt \calE^\delta \lesssim - \intt \pi(\rho)\, \dx - \calD.
\end{equation}

%
%
%
%
%


 \subsection[Completion of the proof for $\sigma = 0$]{
    Completion of the proof for $\sigma = 0$}

By virtue of the coercivity bound \eqref{eq:coer-ex-den}
\[
\intt \pi(\rho)\, \dx + \calD \gtrsim \calE.
\]
Thus, \eqref{e:enX} implies 
\[
\calE^\delta(t)\le\calE^\delta(0)e^{-\lambda t}.
\]
Since $\calE^\delta\sim\calE$, this proves the exponential decay of the kinetic fluctuation, the fluid velocity fluctuation, the relative momentum, and the pressure potential.

Finally, Lemma \ref{lem:pres-pot-est} implies
\[
\|\rho(t)-1\|_{L^1(\T^d)}^2\le C\intt \pi(\rho(t))\,\dx\le Ce^{-\lambda t}.
\]
This completes the proof of Theorem \ref{thm:main1}  in the case $\sigma=0$.

%
%
%
%
%


 \subsection{Hypocoercivity and completion of the proof for $\sigma > 0$}\label{s:hypo}

Let us set $\sigma = 1$  for notational simplicity. The same argument applies to every fixed $\sigma>0$, with constants depending additionally on $\sigma$.
In this case the partial Fisher information $\calI_{vv}^{\bar{v}}$ is insufficient to control the global relative entropy 
\[
\calH^{\bar{v}} = \inttr f\log\frac{f}{\mu_{1,\bar{v},1}}\,\dx\dv.
\]
The full Fisher information must include the $x$-component
\begin{equation}\label{e:Fisherxx}
\calI_{xx}^{\bar{v}} = \inttr \frac{| \nabla_x f|^2}{f} \,\dv \dx.
\end{equation}
So, defining
\[
\calI^{\bar{v}} = \calI_{vv}^{\bar{v}}+\calI_{xx}^{\bar{v}},
\]
by the log-Sobolev inequality we have
\[
\calH^{\bar{v}} \leq \lambda \calI^{\bar{v}}.
\]

By the classical hypocoercivity methodology, see \cite{Villani}, we restore $\calI_{xx}^{\bar{v}}$ with the use of the cross-product information
\[
\calI_{xv}^{\bar{v}} = \inttr \frac{(\nabla_v f + (v - \bar{v})f)\cdot \nabla_x f}{f} \,\dv \dx,
\]
and upgrading all the energy estimates to the Fisher level by writing out evolution equations for all three informations, $\calI_{vv}^{\bar{v}}$, $\calI_{xv}^{\bar{v}}$, and $\calI_{xx}^{\bar{v}}$.

It is most convenient to work it out in the $(A^*A+B)$-framework described in \cite{Villani} (we note however that the presence of the drag force precludes direct application of \cite{Villani}). 
To introduce it we first shift the kinetic velocity variable by the  kinetic momentum
\begin{equation}\label{e:tildas}
\tilde{f}(x,v,t) = f(x,v+\bar{v},t), \quad \tilde{u}_f = u_f - \bar{v}, \quad \tilde{\rho} = \rho.
\end{equation}
Dropping tildes, the kinetic equation in the new variables reads
\[
 \partial_t f + (v + \bar{v}) \cdot \nabla_x f  = \bar{v}_t \cdot \nabla_v f +\rho_\e \nabla_v \cdot ( \nabla_v f +  (v + \bar{v} -u_\e^\e)f) .
\]

Note that in the new variables the total momentum of $f$ vanishes, and the Fisher information becomes centered at $0$. We will therefore drop superscripts from all $\calI$'s. Let us recall from  \eqref{e:enX} that we obtain the following energy inequality and the partial coercivity \eqref{eq:coer-ex-den}
\begin{equation}\label{e:eehypo}
	\ddt \calE^\delta  
		\lesssim - \calD -  \calI_{vv}   - \intt \rho|u-\bar{m}|^2\,\dx  -  |\bar{m} - \bar{v}|^2 - \intt \pi(\rho)\, \dx .
\end{equation}

Let us write the equation for the new distribution $h = \frac{f}{\mu}$, where $\mu = \mu_{1,0,1}$. For this purpose it is simpler to view the equation as a Fokker-Planck equation with a drift,
\begin{equation}\label{ }
\begin{split}
\partial_t f + (v + \bar{v}) \cdot \nabla_x f  &= \bar{v}_t \cdot \nabla_v f + \rmb^\e \cdot \nabla_v  f  + \rho_\e \nabla_v \cdot (  \nabla_v f + vf ),\\
\rmb^\e &=  \rho_\e(\bar{v} -  u^\e_\e).
\end{split}
\end{equation}
Then the $h$-equation reads
\begin{equation}\label{e:FPAh}
h_t + (v + \bar{v}) \cdot \nabla_x h = \bar{v}_t \cdot \nabla_v h - v \cdot \bar{v}_t h+  \rho_\e (  \Delta_v h - v \cdot \nabla_v h) + \rmb^\e \cdot \nabla_v h - v \cdot \rmb^\e  h.
\end{equation}
 Denoting
\[
B = (v + \bar{v}) \cdot \nabla_x, \quad A =  \nabla_v, \quad A^* = (v  -  \nabla_v) \cdot,
\]
where $A^*$ is understood relative to the inner product of the weighted space $L^2(\mu)$,
we arrive at 
\begin{equation}\label{e:FPALN}
h_t =  - \rho_\e A^* A h- Bh - A^*( \rmb^\e h) - A^*( \bar{v}_t h).
\end{equation}

We will now adopt the hypocoercivity estimates presented in \cite{Shv-EA} to the model \eqref{e:FPALN}. Similar adaptations to alignment models have been claimed  in various other contexts, see \cite{ST25,Vinh2023}. To make this suitable for future reference let us state and provide details for the most inclusive result.

Consider a Fokker-Planck model of the following type:
\begin{equation}\label{e:fpgeneric}
h_t  + \st(t,x) A^* A h + Bh + A^*( \rmb(t,x) h) = 0,
\end{equation}
where $\rmb$ is a vector field that depends only on $(t,x)$,  and $\st$ is a scalar function of  $(t,x)$ satisfying
\begin{equation}\label{e:sbounds}
c_0 \leq \st \leq c_1, \text{ and } \|\n_x \st\|_\infty \leq c_2,
\end{equation}
for some $c_0,c_1,c_2>0$. The drift $\rmb$ will also be assumed to be at least in $H^1(\rho_f)$, where $\rho_f = \intr h\,\rd\mu$. However, we will not label its norm by another generic constant. Instead, we make it explicit in the estimates below.

Let us define the full set of Fisher informations:
\[
\calI_{vv} =  \inttr \frac{|\nabla_v h|^2}{h} \,\dmu,
\quad  \calI_{xv} =  \inttr \frac{\nabla_x h \cdot \nabla_v h}{h} \,\dmu, \quad  \calI_{xx}=  \inttr \frac{|\nabla_x h|^2}{h} \,\dmu,
 \]
 and the corresponding dissipation functionals:
 \[
\cD_{vv} =  \inttr \st h |\n_v^2 \log h|^2 \,\dmu, \quad 
\cD_{xv} =  \inttr \st h |\n_v \n_x \log h|^2 \,\dmu.
\]
The remaining $\cD_{xx}$ will not appear in the estimates just as $\cI_{xx}$ does not appear in the entropy equation.

\begin{proposition}\label{p:hypo} For any mass-one classical solution to \eqref{e:fpgeneric} with zero kinetic momentum, we have the following estimate
\begin{equation}\label{e:I1}
\ddt (\calI_{vv}+ \frac12 \calI_{xv} +\calI_{xx} ) \leq c_3 \calI_{vv}  - \frac18 \calI_{xx}  + c_4  \| \nabla_x \rmb \|^2_{L^2(\rho_f)}  + c_5 \| \rmb  - \bar{\rmb} \|^2_{L^2(\rho_f)} + c_6 \| u_f \|^2_{L^2(\rho_f)} ,
\end{equation}
where $c_3,c_4,c_5,c_6>0$ depend only on $c_0,c_1,c_2$ from \eqref{e:sbounds}, and $\bar{\rmb} = \int_{\T^d} \rmb \rho_f \,\dx$.
\end{proposition}
\begin{proof}
The proof of the proposition will be split into several lemmas where we develop differential inequalities for all the Fisher information functionals.

\begin{lemma}\label{}
 The velocity Fisher information satisfies
\[ 
\ddt \cI_{vv}  \leq  -2  \cD_{vv}-  2 c_0 \cI_{vv} -2  \cI_{xv} +  \| \rmb  - \bar{\rmb} \|^2_{L^2(\rho_f)} + \| u_f \|^2_{L^2(\rho_f)} .
\]
\end{lemma}
\begin{proof} 
Let us write $\cI_{vv} = ( \n_v h \cdot \n_v \log h )_\mu$, and compute the derivative
\begin{equation*}\label{}
 \ddt \cI_{vv} = 2 ( \n_v h_t \cdot \n_v \log h )_\mu - ( |\n_v \log h |^2 h_t )_\mu   =  J_A + J_B + J_\rmb,
 \end{equation*}
 where
 \begin{equation*}\label{}
\begin{split}
J_A & =   -2  ( \st \n_v A^*A h \cdot  \n_v \log h )_\mu +  (\st  |\n_v \log h |^2 A^*A h )_\mu, \\
J_B &= -2 ( \n_v B h\cdot  \n_v \log h )_\mu + ( |\n_v \log h |^2 B h )_\mu, \\
J_\rmb & = {\color{blue} - }  2( \n_v A^*(\rmb h)\cdot \n_v \log h )_\mu +  (  |\n_v \log h |^2 A^*(\rmb h) )_\mu.
\end{split}
\end{equation*}

Let us start with the dissipation term $J_A$.  Observe the identity 
\[
 (A^*A h)_{v_i} = A^*Ah_{v_i} + h_{v_i}.
\]
We have
\begin{equation*}\label{}
J_A  = -2  (\st  A^*Ah_{v_i} (\log h)_{v_i} )_\mu - 2 (\st \n_v h \cdot \n_v \log h )_\mu +  (\st |\n_v \log h |^2 A^*A h )_\mu.
\end{equation*}
Note that the term in the middle is bounded from above by $- 2 c_0\cI_{vv} $ in view of the lower bound on $\st$ stated in \eqref{e:sbounds}. In the remaining two terms we move $A^*$ to obtain
\begin{equation*}\label{}
\begin{split}
J_A  & \leq  -2  (\st  Ah_{v_i} \cdot A (\log h)_{v_i} )_\mu - 2 c_0 \cI_{vv} +  (\st  A |\n_v \log h |^2\cdot  A h  )_\mu \\
&= -2 (\st  h A(\log h)_{v_i}\cdot  A (\log h)_{v_i} )_\mu -2 (\st (\log h)_{v_i} Ah\cdot  A (\log h)_{v_i}  )_\mu - 2 c_0 \cI_{vv} + 2 (\st  (\log h)_{v_i}A (\log h)_{v_i} \cdot  A h  )_\mu.
\end{split}
\end{equation*}
The second and fourth terms cancel, while the first term is exactly equal to $-2  \cD_{vv}$:
\[
J_A  \leq -2 \cD_{vv}  -  2 c_0 \cI_{vv}.
\]

Next, we treat the transport term $J_B$. We have
\begin{align*}
J_B &= -2 ( \n_x h \cdot  \n_v \log h )_\mu - 2 ( ((v+\bar{v}) \cdot \n_x h_{v_i}) (\log h)_{v_i} )_\mu  + ( |\n_v \log h |^2 (v+\bar{v}) \cdot \n_x h )_\mu  \\
&=: J_B^1 + J_B^2 + J_B^3.
\end{align*}
Note  that 
\begin{equation*}\label{}
\begin{split}
J_B^1 & =  -2  \cI_{xv}, \\
J_B^2 &= - 2 ((v+\bar{v}) \cdot \n_x h_{v_i} {h}_{v_i}{h}^{-1} )_\mu = - ( (v+\bar{v}) \cdot \n_x |h_{v_i}|^2 h^{-1} )_\mu \\
& = -( |h_{v_i}|^2 (v+\bar{v}) \cdot \n_x h h^{-2} )_\mu =- ( |(\log h)_{v_i}|^2 (v+\bar{v}) \cdot \n_x h )_\mu \\
&= - J_B^3.
\end{split}
\end{equation*}
Thus, $J_B^2  + J_B^3 = 0$. We obtain
\[
J_B = -2 \cI_{xv}.
\]

Next, for the drift term we prove the following the exact identity 
\begin{equation}\label{e:Juexact}
J_\rmb = -  2 ( \rmb \cdot u_f )_{\rho_f}.
\end{equation}
Indeed,
\begin{equation*}\label{}
\begin{split}
J_\rmb & = -  2(\n_v A^*(\rmb h)\cdot  \n_v \log h )_\mu + ( |\n_v \log h |^2 A^*(\rmb h))_\mu \\
& = {\color{blue} - } 2(  \n_v ( v \cdot \rmb h - \rmb \cdot \n_v h)\cdot  \n_v \log h )_\mu + (  \n_v |\n_v \log h |^2 \cdot  \rmb h )_\mu \\
& = {\color{blue} - } 2 ( \rmb h  \cdot  \n_v \log h )_\mu - 2 ( (v \cdot \rmb) \n_v h \cdot  \n_v \log h )_\mu + 2( \n^2_v h \rmb  \cdot  \n_v \log h )_\mu + 2(  \n_v^2 \log h(\n_v \log h) \cdot \rmb h )_\mu\\
& = : J_\rmb^1 + J_\rmb^2+J_\rmb^3+J_\rmb^4.
\end{split}
\end{equation*} 
The first term is exactly the inner product that appears  in formula \eqref{e:Juexact}:
 \[
J_\rmb^1 =  {\color{blue} - } 2 (\rmb \cdot  \n_v h )_\mu = {\color{blue} - } 2 (\rmb \cdot v h )_\mu = {\color{blue} - } 2 ( \rmb \cdot u_f )_{\rho_f}.
 \] 
The next crucial observation is that the remaining terms cancel:
\begin{equation}\label{e:J234=0}
J_\rmb^2+J_\rmb^3+J_\rmb^4=0.
\end{equation}
 Indeed, using the identity 
 \begin{equation}\label{e:hbarh}
  h_{v_i v_j} = h (\log h)_{v_i v_j} + \frac{1}{h} h_{v_i}  h_{v_j},
\end{equation}
we obtain
\begin{align*}
J_\rmb^3 &=   2  (  h_{v_i v_j} \rmb_j (\log h)_{v_i} )_\mu \\
& =  2 (  h (\log h)_{v_i v_j}  \rmb_j (\log h)_{v_i} )_\mu + 2 ( \frac{1}{h} h_{v_i}  h_{v_j} \rmb_j (\log h)_{v_i} )_\mu \\
&= J_\rmb^4 + 2 ( \n_v h \cdot \rmb | \n_v \log h |^2   )_\mu.
\end{align*}
Moreover,
\[
J_\rmb^2 = {\color{blue} - } 2 ( v \cdot \rmb h |\n_v \log h|^2 )_\mu.
\]
Then we have
\[
J_\rmb^2 + J_\rmb^3 = J_\rmb^4 - 2( v \cdot \rmb h  |\n_v \log h|^2 )_\mu + 2 ( \n_v h \cdot \rmb | \n_v \log h |^2   )_\mu = J_\rmb^4 - 2 (  A^*( \rmb h)  | \n_v \log h |^2 )_\mu .
\]
Switching $A^*$ in the last term we obtain
\[
{\color{blue} - } 2 (  A^*( \rmb h)  | \n_v \log h |^2 )_\mu  = {\color{blue} - } 2 (  h  \rmb \cdot \n_v  | \n_v \log h |^2 )_\mu = -2 J_\rmb^4.
\]
Thus, $J_\rmb^2 + J_\rmb^3 = - J_\rmb^4$, which proves \eqref{e:J234=0}.

Now, recalling that $u_f$ has zero $\rho_f$-weighted average, we obtain
\[
J_\rmb \leq  \| \rmb  - \bar{\rmb} \|^2_{L^2(\rho_f)} + \| u_f \|^2_{L^2(\rho_f)}.
\]
This completes the proof.
\end{proof}

\begin{lemma}\label{}
 The mixed Fisher information satisfies
\[
\ddt \cI_{xv}  \leq  - \frac14  \cI_{xx}  + a_1 \cI_{vv}   +  2 \cD_{vv} +\cD_{xv} + a_2 \| \n_x \rmb \|_{L^2(\rho_f)}^2 + a_3  \| \rmb -  \bar{\rmb}\|_{L^2(\rho_f)}^2,
\]
where $a_1,a_2,a_3>0$ depend only on $c_0,c_1,c_2$.
\end{lemma}
\begin{proof}
Let us write
\[
 \ddt \cI_{xv} = (\n_x h_t \cdot \n_v \log h )_\mu + ( \n_x \log h \cdot \n_v {h}_t )_\mu - ( h_t \n_v \log h\cdot \n_x \log h )_\mu  = J_A + J_B + J_\rmb,
\]
where as before $J_A,J_B,J_\rmb$ collect contributions from $A$, $B$, and the drift, respectively:
\begin{equation*}\label{}
\begin{split}
J_A & =    - (\n_x (\st A^*Ah) \cdot \n_v \log h )_\mu - (\st  \n_x \log h \cdot \n_v  A^*Ah )_\mu + (\st  A^*Ah \n_v \log h\cdot \n_x \log h )_\mu \\
&= J_A^1 + J_A^2 + J_A^3,\\
J_B & = - ( \n_x ((v + \bar{v}) \cdot \n_x h) \cdot \n_v \log h )_\mu - ( \n_x \log h \cdot \n_v ((v + \bar{v}) \cdot \n_x h)  )_\mu  \\
&\quad + ( ((v + \bar{v}) \cdot \n_x h) \n_v \log h\cdot \n_x \log h  )_\mu \\
&= J_B^1 + J_B^2 + J_B^3,\\
J_\rmb & = {\color{blue} - }  ( \n_x ( A^*(\rmb h)) \cdot \n_v \log h)_\mu -( \n_x \log h \cdot \n_v (  A^*(\rmb h))  )_\mu \\
&\quad  +(  A^*(\rmb h) \n_v \log h\cdot \n_x \log h )_\mu,
\end{split}
\end{equation*}

We start with the $J_A$-term. For $J_A^1$ we obtain
\[
\begin{split}
J_A^1 & = - ( \st A^* A h_{x_i} (\log h)_{v_i} )_\mu - ( \st_{x_i} A^* A h (\log h)_{v_i} )_\mu \\
& = - ( \st \n_v h_{x_i} \cdot \n_v  (\log h)_{v_i} )_\mu - ( \st_{x_i} \n_v h \n_v (\log h)_{v_i} )_\mu \\
& = -  ( \st h \n_v (\log h)_{x_i} \cdot \n_v  (\log h)_{v_i} )_\mu- ( \st  (\log h)_{x_i} \n_v h \cdot \n_v  (\log h)_{v_i} )_\mu \\
&\quad - \left( \frac{\st_{x_i}}{\st^{1/2}} \frac{\n_v h}{h^{1/2}} \cdot  \st^{1/2} h^{1/2} \n_v (\log h)_{v_i} \right)_\mu .
\end{split}
\]
In view of the assumption \eqref{e:sbounds}, 
\[
J_A^1 \leq  -  (\st  h \n_v (\log h)_{x_i} \cdot \n_v  (\log h)_{v_i} )_\mu- (\st  (\log h)_{x_i} \n_v h \cdot \n_v  (\log h)_{v_i} )_\mu + c_3 \sqrt{\cI_{vv} \cD_{vv}},
\]
for some $c_3=c_3(c_0,c_1,c_2)>0$. For $J_A^2$ we obtain
\begin{equation*}\label{}
\begin{split}
J_A^2 & = - (\st \n_x \log h \cdot \n_v h )_\mu - ( \st (\log h)_{x_i} A^*A {h}_{v_i} )_\mu\\
&    = - (\st \n_x \log h \cdot \n_v h )_\mu -  ( \st h \n_v (\log h)_{x_i} \cdot \n_v  (\log h)_{v_i} )_\mu  -  ( \st  (\log h)_{v_i} \n_v (\log h)_{x_i} \cdot \n_v  h )_\mu  .
\end{split}
\end{equation*}
The two add up to 
\[
\begin{split}
J_A^1+J_A^2 & \leq  - \left(\st \frac{\n_x {h}}{h^{1/2}} \cdot \frac{\n_v h}{h^{1/2}} \right)_\mu  - 2  (\st h \n_v (\log h)_{x_i} \cdot \n_v  (\log h)_{v_i} )_\mu - (\st  Ah \cdot A(\n_v \log h\cdot \n_x \log h ) )_\mu \\
&\quad  + c_3 \sqrt{\cI_{vv} \cD_{vv}} \\
& \leq  c_1 \sqrt{\cI_{xx} \cI_{vv}} + 2 \sqrt{ \cD_{xv} \cD_{vv}} - J_A^3+c_3 \sqrt{\cI_{vv} \cD_{vv}} .
\end{split}
\]
Thus,
\[
J_A \leq c_1 \sqrt{\cI_{xx} \cI_{vv}}  + 2 \sqrt{ \cD_{xv} \cD_{vv}} +   c_3 \sqrt{\cI_{vv} \cD_{vv}}.
\]

Moving on to the $B$-term we expand $J_B^2$ as follows
\[
J_B^2 =  - ( \n_x \log h \cdot \n_x h )_\mu - ( (\log h)_{x_i} (v_j +\bar{v}_j) h_{x_j v_i})_\mu.
 \]
The first term is exactly $- \cI_{xx} $, while in the second integrating by parts in $x_j$, we obtain
 \[
 = - \cI_{xx} + ( (\log h)_{x_i x_j} (v_j +\bar{v}_j) h_{v_i} )_\mu.
\]
Using that $(\log h)_{x_i x_j}= h^{-1} {h}_{x_i x_j} - (\log h)_{x_i} (\log h)_{x_j}$,
\[
= -  \cI_{xx} +( {h}_{x_i x_j} (v_j +\bar{v}_j) (\log h)_{v_i} )_\mu - ( (\log h)_{x_i} (\log h)_{x_j} (v_j +\bar{v}_j) h_{v_i} )_\mu = -  \cI_{xx} - J_B^1 - J_B^3.
\]
Hence,
\begin{equation}\label{e:JBxv}
J_B = -  \cI_{xx} .
\end{equation}

As to the drift term,
\begin{equation*}\label{}
\begin{split}
J_\rmb & = {\color{blue} -} ( \n_x (  A^*(\rmb h)) \cdot \n_v \log h)_\mu - ( \n_x \log h \cdot \n_v ( A^*(\rmb h))  )_\mu + (   A^*(\rmb h) \n_v \log h\cdot \n_x \log h )_\mu \\
& =  {\color{blue} -} (  A^*(\rmb_{x_i} h) (\log h)_{v_i})_\mu - ( A^*(\rmb h_{x_i})  (\log h)_{v_i})_\mu \\
&\quad -  (  (\log h)_{x_i}   A^*(\rmb h_{v_i} ) )_\mu -  ( h \n_x \log h \cdot \rmb)_\mu  + (  h \rmb \cdot \n_v( \n_v \log h\cdot \n_x \log h) )_\mu \\
& = - ( h \rmb_{x_i}  \cdot \n_v (\log h)_{v_i})_\mu - (  h \rmb (\log h)_{x_i} \cdot \n_v (\log h)_{v_i})_\mu \\
&\quad - ( h \n_v (\log h)_{x_i} \cdot   \rmb (\log h)_{v_i}  )_\mu -  (  h \n_x \log h \cdot \rmb)_\mu  + (  h \rmb  \cdot \n_v( \n_v \log h\cdot \n_x \log h) )_\mu
\end{split} \end{equation*}
The second, third, and fifth terms cancel. Hence,
\[
J_\rmb = {\color{blue} -} ( h \rmb_{x_i}  \cdot \n_v (\log h)_{v_i})_\mu -( h \n_x \log h \cdot \rmb)_\mu =: J_\rmb^1+J_\rmb^2.
\]
The first term is estimated by the H\"older inequality,
\[
J_\rmb^1  \leq c_4 \| \n_x \rmb \|_{L^2(\rho_f)} \sqrt{\cD_{vv}}.
\]
In the second we notice that one can subtract any constant vector from $\rmb$ without changing the result. So, by subtracting $\bar{\rmb}$ we obtain
\[
J_\rmb^2  \leq \| \rmb -  \bar{\rmb}\|_{L^2(\rho_f)} \sqrt{\cI_{xx}} .
\]

Collecting the obtained estimates and using the generalized Young inequality, we obtain
\begin{equation*}\label{}
\begin{split}
\ddt \cI_{xv} & \leq c_1 \sqrt{\cI_{xx} \cI_{vv}}  + 2 \sqrt{ \cD_{xv} \cD_{vv}} +   c_3 \sqrt{\cI_{vv} \cD_{vv}}-  \cI_{xx} +c_4 \| \n_x \rmb \|_{L^2(\rho_f)} \sqrt{\cD_{vv}}+  \| \rmb -  \bar{\rmb}\|_{L^2(\rho_f)} \sqrt{\cI_{xx}}  \\
&\leq - \frac14  \cI_{xx}  + c_5\cI_{vv}   +  2 \cD_{vv} +\cD_{xv} + c_6\| \n_x \rmb \|_{L^2(\rho_f)}^2 + c_7  \| \rmb -  \bar{\rmb}\|_{L^2(\rho_f)}^2.
\end{split}
\end{equation*}

\end{proof}

\begin{lemma}\label{}
The spatial Fisher information satisfies
\[
\ddt \cI_{xx}  \leq  a_4 \cI_{vv} - \frac12 \cD_{xv} + a_5 \| \n_x \rmb \|^2_{L^2(\rho_f)},
\]
where $a_4,a_5>0$ depend only on $c_0,c_1,c_2$.
\end{lemma}
 \begin{proof}
 We have
 \[
 \ddt \cI_{xx} =  2( \n_x h_t \cdot \n_x \log h )_\mu - ( |\n_x \log h |^2 h_t )_\mu = J_A+ J_B + J_\rmb.
\]

The $A$-term is given by
\[
\begin{split}
 J_A & = -2 ( \n_x (\st A^*A h) \cdot \n_x \log h )_\mu + (\st  |\n_x \log h |^2 A^*Ah )_\mu\\
& =-2 ( \st_{x_i} A h \cdot A (\log h)_{x_i} )_\mu -2 (\st  A h_{x_i} \cdot A (\log h)_{x_i} )_\mu+ (\st A |\n_x \log h |^2 \cdot Ah)_\mu \\
& \leq c_3 \sqrt{\cD_{xv} \cI_{vv}} -2 (\st \n_v (h (\log h)_{x_i}) \cdot \n_v (\log h)_{x_i} )_\mu + (\st \n_v |\n_x \log h |^2 \cdot \n_vh )_\mu\\
& = c_3 \sqrt{\cD_{xv} \cI_{vv}} -2 (\st h \n_v (\log h)_{x_i} \cdot \n_v (\log h)_{x_i} )_\mu -2  (\st  (\log h)_{x_i} \n_v h \cdot \n_v (\log h)_{x_i} )_\mu \\
&\quad + (\st \n_v |\n_x \log h |^2 \cdot \n_vh )_\mu\\
& =c_3 \sqrt{\cD_{xv} \cI_{vv}} -2 \cD_{xv} -  (\st \n_v |\n_x \log h |^2 \cdot \n_vh )_\mu+ (\st \n_v |\n_x \log h |^2 \cdot \n_vh )_\mu \\
& = c_3  \sqrt{\cD_{xv} \cI_{vv}}-2 \cD_{xv}.
\end{split}
\]
Thus,
\[
J_A \leq c_4 \cI_{vv} - \cD_{xv}.
\]

The $B$-term cancels altogether, 
\[
\begin{split}
J_B & = -2 ( \n_x((v+\bar{v}) \cdot \n_x h) \cdot \n_x \log h )_\mu + ( |\n_x \log h |^2 (v+\bar{v}) \cdot \n_x h )_\mu \\
& = -2 ( ((v+\bar{v}) \cdot \n_x h_{x_i}) h_{x_i} h^{-1} )_\mu + ( |\n_x \log h |^2 (v+\bar{v}) \cdot \n_x h )_\mu \\
& = - ( ((v+\bar{v}) \cdot \n_x |\n_x h|^2 h^{-1} )_\mu + ( |\n_x \log h |^2 (v+\bar{v}) \cdot \n_x h )_\mu \\
& =  - ( (v+\bar{v}) \cdot \n_x h  |\n_x h|^2 h^{-2} )_\mu + ( |\n_x \log h |^2 (v+\bar{v}) \cdot \n_x h )_\mu \\
&= 0.
\end{split}
\]

Finally, the drift term is given by
\[
J_\rmb = {\color{blue} -} 2( \n_x  A^*(\rmb h)  \cdot \n_x \log h )_\mu + (  |\n_x \log h |^2 A^*(\rmb h) )_\mu .
\]
In the second term we switch the operator $A^*$: 
\begin{equation}\label{e:auxxx1}
 ( |\n_x \log h |^2 A^*(\rmb h))_\mu =  ( \n_v |\n_x \log h |^2 \rmb h)_\mu.
\end{equation}
For the first term we obtain
\begin{equation*}\label{}
\begin{split}
 {\color{blue} -} 2( \n_x  A^*(\rmb h)  \cdot \n_x \log h )_\mu & = {\color{blue} -} 2(h \rmb_{x_i}   \cdot \n_v (\log h)_{x_i} )_\mu  - 2(   h  (\log h)_{x_i} \rmb \cdot \n_v (\log h)_{x_i} )_\mu.
  \end{split}
\end{equation*}
Since
\[
\n_v|\n_x\log h|^2 = 2(\log h)_{x_i}\n_v(\log h)_{x_i},
\]
the second term on the right-hand side cancels with \eqref{e:auxxx1}. Thus,
\[
J_\rmb =   {\color{blue} -}  2(h \rmb_{x_i}   \cdot \n_v (\log h)_{x_i} )_\mu  \leq c_5 \| \n_x \rmb \|^2_{L^2(\rho_f)} + \frac12 \cD_{xv},
\]
and subsequently, 
\[
\ddt \cI_{xx} \leq  c_4 \cI_{vv} - \frac12 \cD_{xv} + c_5 \| \n_x \rmb \|^2_{L^2(\rho_f)} .
\]
\end{proof}

 Let us summarize the three obtained inequalities
 \begin{equation*}\label{}
\begin{split}
\ddt \cI_{vv} & \leq  -2  \cD_{vv}-  2 c_0 \cI_{vv} -2  \cI_{xv} +  \| \rmb  - \bar{\rmb} \|^2_{L^2(\rho_f)} + \| u_f \|^2_{L^2(\rho_f)}, \\
\ddt \cI_{xv} &  \leq  - \frac14  \cI_{xx}  + a_1 \cI_{vv}   +  2 \cD_{vv} +\cD_{xv} + a_2 \| \n_x \rmb \|_{L^2(\rho_f)}^2 + a_3  \| \rmb -  \bar{\rmb}\|_{L^2(\rho_f)}^2,\\
\ddt \cI_{xx}  &\leq  a_4 \cI_{vv} - \frac12 \cD_{xv} + a_5 \| \n_x \rmb \|^2_{L^2(\rho_f)}.
\end{split}
\end{equation*}
Furthermore, in the first equation we can use the generalized Young inequality, $|\calI_{xv}| \leq \frac{1}{32} \calI_{xx} + 32 \calI_{vv}$. 
Thus, adding up the obtained equations we absorb all the dissipation terms $\calD$'s and obtain
\eqref{e:I1}. 
\end{proof}

We now apply Proposition~\ref{p:hypo} in our situation, where 
\[
\rmb = \rmb^\e + \bar{v}_t.
\]
Since all the terms involving $\rmb$ involve either the gradient in $x$ or the mean-zero part of $\rmb$, we can see that $\bar{v}_t$ plays no role in any of them.

Let us now estimate the last three terms that appear in \eqref{e:I1}. Note that 
\[
\calI_{vv} = \calI_{vv}^{u_f} + \| u_{f} \|^2_{L^2(\rho_f)} \geq  \| u_{f} \|^2_{L^2(\rho_f)}.
\]
So, the last energy term $\| u_{f} \|^2_{L^2(\rho_f)}$ just estimates by another Fisher information.  Next, 
\[
\| \rmb^\e \|^2_{L^2(\rho_f)} =\intt | \rho_\e(\bar{v} -  u^\e_\e)|^2 \rho_f \,\dx \lesssim \intt |\bar{m} -  u^\e_\e|^2 \rho_f \, \dx + | \bar{m} - \bar{v}|^2,
\]
which by \eqref{eq:w-ue-uc} and \eqref{eq:coer-ex-den} are bounded by $C_{\e,M} \calD$,
\[
\|  \rmb^\e  \|^2_{L^2(\rho_f)} \leq C_{\e,M} \calD.
\]
Finally,
\[
 \nabla_x  \rmb^\e  = \nabla_x \rho_\e (\bar{v} -  u^\e_\e) - \rho_\e \nabla_x (\bar{v} -  u)^\e_\e.
\] 
Since $\nabla_x \rho_\e\in L^\infty_{t,x}$,  the first term is estimated exactly as $b$ itself. Lastly, 
\[
\nabla_x (\bar{v} -  u)^\e_\e =-  ((\bar{v} -  u)\rho)_\e \frac{ \nabla_x  \rho_\e}{ \rho_\e^2} + \frac{1}{\rho_\e}  ( (\bar{v} -  u)\rho)_{\nabla_x \theta_\e}.
\]
Again, using the boundedness of $\rho_\e$ from above and below, and its smoothness, we obtain
\[
\| \nabla_x  \rmb^\e  \|^2_{L^2(\rho_f)} \lesssim \intt |\bar{v} - u|^2 \rho \,\dx \leq \intt |\bar{m} - u|^2 \rho \,\dx + |\bar{v} - \bar{m}|^2 \lesssim \calD.
\]

Coming back to \eqref{e:I1} we obtain
\begin{equation}\label{e:I2}
\ddt (\calI_{vv}+ \frac12 \calI_{xv} +\calI_{xx} ) \leq c_3 \calI_{vv}  - \frac12 \calI_{xx}  + c_7 \calD.
\end{equation}

Adding this to a large constant multiple of \eqref{e:eehypo} we obtain
\[
\ddt \lt( \calI_{vv}+ \frac12 \calI_{xv} +\calI_{xx} + C \calE^\delta\rt) \lesssim -  (\calI_{vv}  + \calI_{xx})   - \intt \rho|u-\bar{m}|^2\,\dx  -  |\bar{m} - \bar{v}|^2 - \intt \pi(\rho) \,\dx .
\]
Since $\calI_{vv}  + \calI_{xx}$ dominates the relative entropy component of $\calE$, we obtain the Gr\"onwall's inequality for  $Y = \calI_{vv}+\frac12 \calI_{xv} +\calI_{xx} + C \calE^\delta$. 
So, from the above we obtain 
\[
\ddt Y \lesssim - Y.
\]
It remains to observe that 
\begin{align*}
Y &\sim \calI_{vv}+\calI_{xx} + \intt \rho|u-\bar{m}|^2\,\dx  +  |\bar{m} - \bar{v}|^2 + \intt \pi(\rho) \,\dx \\
&\gtrsim \|f - \mu\|_1^2 +\intt \rho|u-\bar{m}|^2\,\dx  +  |\bar{m} - \bar{v}|^2+\|\rho - 1\|_{L^1}^2 .
\end{align*}
This proves the relaxation statement.

%
%
%
%
%
%
%
%
%
%
%

\section{Unconditional relaxation in low dimensions}\label{sec:bd-estimate}
 
The conditional relaxation theorem proved in Section \ref{sec:energy} requires uniform bounds on positive and negative powers of the fluid density, together with a uniform $\calA_2$ bound on $\rho^\alpha$. These assumptions are weaker than uniform pointwise upper and lower bounds on $\rho$, but they do not follow from the free-energy inequality alone. The purpose of this section is to show that, in the low-dimensional regimes relevant to the weak solution theory developed in Section \ref{sec:weak-sol}, these weighted density assumptions are consequences of the Bresch--Desjardins entropy structure.

More precisely, we first derive a BD entropy estimate for sufficiently regular positive solutions. We then use the resulting density-gradient bounds to verify the weighted density conditions in Theorem \ref{thm:main1}. This gives an unconditional relaxation result for smooth solutions in the admissible low-dimensional cases, and the same estimates will later be used,
in their time-integrated weak form, to prove the long-time relaxation of the finite-energy BD entropy weak solutions constructed in Section \ref{sec:weak-sol}.

\begin{proposition}\label{prop:main}
Fix $\e>0$, and let $(f,\rho,u)$ be a sufficiently regular positive solution to \eqref{main_sys} on $[0,\infty)$  satisfying
\[
\calE_{\rm free}(0)+\calE_{\rm BD}(0)<\infty .
\]
Then the conditions of Theorem \ref{thm:main1} are verified in the following two cases:
\[
\begin{cases}
d=1, \quad 0<\alpha\le\frac12, \quad \gamma=\alpha+1,\\[1mm]
d=2, \quad \alpha=\frac12, \quad \gamma=\frac32.
\end{cases}
\]
\end{proposition}

\begin{remark}
All the bounds in Proposition \ref{prop:main} may depend on $\e$ and on the initial energies. In particular, the estimates are not asserted to be uniform with respect to the regularization parameter $\e$.
\end{remark}

Under either set of assumptions in Proposition \ref{prop:main}, we have
\[
1-\frac1d\le\alpha=\gamma-1.
\]
In particular, the viscous dissipation is nonnegative. Moreover,
\[
q:=\alpha-\gamma+1=0,
\]
so the drag remainder can be controlled without imposing any additional
positive or negative power bound on the density. For completeness, we
derive below a slightly more general BD estimate under the condition
\[
\gamma-1\le\alpha\le2\gamma-1.
\]

%
%
%
%
%
%
%
%
%
%
%

\subsection{Effective velocity and Bresch--Desjardins entropy}
We recall the BD potential $\phi$ introduced in the Introduction. Equivalently, up to an irrelevant additive constant, it is given by
\[
\phi(\rho)=
\begin{cases}
\dfrac{2\alpha}{\alpha-1}\rho^{\alpha-1}, & \alpha\neq1,\\[2mm]
2\log\rho, & \alpha=1.
\end{cases}
\]
 In particular,
\bq\label{eq:bd-pote}
\phi'(\rho)=\frac{2\nu'(\rho)}{\rho}=2\alpha\rho^{\alpha-2}.
\eq
We introduce the effective velocity
\[
w:=u+\nabla\phi(\rho).
\]
With the notation fixed in the Introduction, the BD entropy reads
\[
\calE_{\rm BD}= H+ \intt \lt[\frac12\rho|w|^2+\pi(\rho)\rt]\dx,
\]
and the combined energy is
\[
\calE_\eta=\calE_\free+\eta\calE_\BD.
\]
We will also use the dissipations $\calD_{\rm vis}$, $\calD_{\rm BD}$, $\calD_{\rm rot}$, and $\calD_{\rm drag}$ introduced above. In particular,
\[
\calD_{\rm drag} \sim \calI_{vv}^{u_\e^\e}, \quad \calI_{vv}^{u_\e^\e} := \inttr \frac{\lt|\sigma\nabla_v f+(v-u_\e^\e)f\rt|^2}{f} \,\dx\dv,
\]
where the equivalence constants depend on $\e$, since $\rho_\e$ is bounded above and below by positive constants.

We first derive the equation satisfied by the effective velocity. Since the continuity equation gives
\[
(\pa_t+u\cdot\nabla)\rho=-\rho\nabla\cdot u,
\]
we obtain from \eqref{eq:bd-pote} that
\[
(\pa_t+u\cdot\nabla)\phi(\rho)=-2\nu'(\rho)\nabla\cdot u.
\]
Taking the gradient, we find
\bq\label{eq:grad-bd-pot}
(\pa_t+u\cdot\nabla)\nabla\phi(\rho)=-\nabla\lt(2\nu'(\rho)\nabla\cdot u\rt)-(\nabla u)^{\mathsf T}\nabla\phi(\rho).
\eq
Moreover, the definition of $\phi$ implies $\rho\nabla\phi(\rho)=2\nabla\nu(\rho)$. Multiplying \eqref{eq:grad-bd-pot} by $\rho$ and using the continuity equation again, we obtain
\bq\label{eq:BD-grad-cons}
\pa_t\lt(\rho\nabla\phi(\rho)\rt)+\nabla\cdot\lt(\rho u\otimes\nabla\phi(\rho)\rt) =-\rho\nabla\lt(2\nu'(\rho)\nabla\cdot u\rt)-2(\nabla u)^{\mathsf T}\nabla\nu(\rho).
\eq

The BD relation \eqref{eq:bd-rel} yields the algebraic identity
\bq\label{eq:BD-visc-canc}
\nabla\cdot\lt(2\nu(\rho)D u\rt)+\nabla\lt(\lambda(\rho)\nabla\cdot u\rt)-\rho\nabla\lt(2\nu'(\rho)\nabla\cdot u\rt)-2(\nabla u)^{\mathsf T}\nabla\nu(\rho)  =\nabla\cdot\lt(2\nu(\rho)A u\rt).
\eq
Adding \eqref{eq:BD-grad-cons} to the fluid momentum equation and using \eqref{eq:BD-visc-canc}, we obtain the effective-velocity equation
\bq\label{eq:effe-vel-eq}
\pa_t(\rho w)+\nabla\cdot(\rho u\otimes w)+\nabla p(\rho)-\nabla\cdot\lt(2\nu(\rho)A u\rt)=\rho F_\e.
\eq

We next derive the corresponding entropy identity. Recall that the pressure potential satisfies
\[
\rho \pi'(\rho)-\pi(\rho)=p(\rho)-p(1).
\]
Multiplying \eqref{eq:effe-vel-eq} by $w$, integrating over $\T^d$, and using the continuity equation, we obtain
\[
\intt \lt[\pa_t(\rho w)+\nabla\cdot(\rho u\otimes w)\rt]\cdot w\,\dx=\frac{\dd}{\dt}\frac12\intt \rho|w|^2\,\dx.
\]
For the pressure term, we use $w=u+\nabla\phi(\rho)$ and \eqref{eq:h-ide} to write
\[
\intt \nabla p(\rho)\cdot w\,\dx=-\intt p(\rho)\nabla\cdot u\,\dx+\intt \nabla p(\rho)\cdot\nabla\phi(\rho)\,\dx=\frac{\dd}{\dt}\intt \pi(\rho)\,\dx+\intt \nabla p(\rho)\cdot\nabla\phi(\rho)\,\dx.
\]
Moreover, since $A u$ is antisymmetric and $\nabla^2\phi(\rho)$ is symmetric, we have
\[
\intt 2\nu(\rho)A u:\nabla w\,\dx=2\intt \nu(\rho)|A u|^2\,\dx.
\]
Consequently,
\begin{align*}
&\frac{\dd}{\dt}\intt \lt[\frac12\rho|w|^2+\pi(\rho)\rt]\dx+\intt \nabla p(\rho)\cdot\nabla\phi(\rho)\,\dx+2\intt \nu(\rho)|A u|^2\,\dx =\intt \rho w\cdot F_\e\,\dx.
\end{align*}
Recall from \eqref{eq:kin-free-sig} that the $H$-equation reads
\bq\label{eq:kin-free-sig2}
\frac{\dd}{\dt}H + \calD_\drag  =  - \intt \rho u \cdot F_\e\,\dx.
\eq
Adding together we obtain
\begin{equation}\label{eq:BD-e-iden}
\frac{\dd}{\dt}\calE_\BD +\calD_\BD+ \calD_{\rm rot}+ \calD_\drag =\intt \rho\nabla\phi(\rho)\cdot F_\e\,\dx.
\end{equation}
The right-hand side is the only term that does not have a definite sign.

%
%
%
%
%
%
%
%
%
%
%

\subsection{Density dissipation and estimate of the drag remainder}

Under our choice  $\nu(\rho)=\rho^\alpha$, $p(\rho)=\rho^\gamma$, we have $\lambda(\rho)=2(\alpha-1)\rho^\alpha$, $\phi'(\rho)=2\alpha\rho^{\alpha-2}$. The density dissipation appearing in \eqref{eq:BD-e-iden} can be written explicitly as
\bq\label{eq:BD-den-diss}
\nabla p(\rho)\cdot\nabla\phi(\rho)=2\alpha\gamma\rho^{\alpha+\gamma-3}|\nabla\rho|^2 =\frac{8\alpha\gamma}{(\alpha+\gamma-1)^2}\lt|\nabla\rho^{\frac{\alpha+\gamma-1}{2}}\rt|^2.
\eq
Once the BD estimate is closed below, this identity yields
\[
\nabla\rho^{\frac{\alpha+\gamma-1}{2}}\in L^2\lt(0,T;L^2(\T^d)\rt).
\]

The combined energy also controls a density-gradient quantity uniformly in time. Indeed, $|\nabla\phi(\rho)|^2\le2|u+\nabla\phi(\rho)|^2+2|u|^2$. Consequently, for every fixed $\eta>0$,
\[
\intt \rho|\nabla\phi(\rho)|^2\,\dx \le C_\eta\calE_\eta(t).
\]
If $\alpha\neq\frac12$, then
\[
\rho|\nabla\phi(\rho)|^2 = 4\alpha^2\rho^{2\alpha-3}|\nabla\rho|^2 = \frac{4\alpha^2}{\lt(\alpha-\frac12\rt)^2} \lt|\nabla\rho^{\alpha-\frac12}\rt|^2.
\]
Hence, the combined estimate derived below yields
\bq\label{eq:BD-den-grad}
\nabla\rho^{\alpha-\frac12}\in L^\infty\lt(0,T;L^2(\T^d)\rt).
\eq
In the borderline case $\alpha=\frac12$, we instead have $\rho|\nabla\phi(\rho)|^2=|\nabla\log\rho|^2$, and hence
\bq\label{eq:BD-log-den-grad}
\nabla\log\rho\in L^\infty\lt(0,T;L^2(\T^d)\rt).
\eq

It remains to control the right-hand side of \eqref{eq:BD-e-iden}. Recall the free energy is non-increasing
\[
\calE_\free(t)\le\calE_\free(0).
\]
So it remains uniformly bounded.  Since $|\T^d|=1$ and $\intt \rho\,\dx=1$, the definition of $\pi$ also yields
\[
\intt \rho^\gamma\,\dx=1+(\gamma-1)\intt \pi(\rho)\,\dx\le1+(\gamma-1)\calE_{\rm free}(0).
\]
So, $\rho \in L^\infty_t L^\gamma_x$.
Using $\rho\nabla\phi(\rho)=2\alpha\rho^{\alpha-1}\nabla\rho$ and \eqref{eq:BD-den-diss}, we obtain from Young's inequality that 
\begin{equation}\label{e:Fe}
\lt|\intt \rho\nabla\phi(\rho)\cdot F_\e\,\dx\rt| \le \frac12 \calD_\BD +C \intt \rho^q|F_\e|^2\,\dx.
\end{equation}

We next estimate the last term using the Fisher information estimate \eqref{e:FL1}, and the fact that 
$\calD_\drag \sim \calI_{vv}^{u_\e^\e}$,
\[
\intt \rho^q|F_\e|^2\,\dx \leq C_\e  \|F\|_1^2 \intt \rho^q dx \lesssim \calD_\drag\intt \rho^q dx.
\]
 Assuming that
\bq\label{eq:BD-exp-cond}
0\le q\le\gamma, \text{ or equivalently, } \gamma-1\le\alpha\le2\gamma-1,
\eq
we bound 
\[
\intt \rho^q\,\dx\le C\lt(1+\intt \rho^\gamma\,\dx\rt)\le C\lt(1+\calE_{\rm free}(0)\rt).
\]
Incorporating this into \eqref{e:Fe} we obtain
\bq\label{eq:BD-d-rema}
\lt|\intt \rho\nabla\phi(\rho)\cdot F_\e\,\dx\rt| \leq \frac12 \calD_\BD +C_\e \calD_{\drag}.
\eq

Consequently, we can see that the $BD$-dissipation is absorbed by cost of leaking the drag dissipation out:
\begin{equation}\label{eq:BD-e-iden2}
\frac{\dd}{\dt}\calE_\BD +\frac12 \calD_\BD+ \calD_{\rm rot}\leq C \calD_{\drag}.
\end{equation}

Combining this estimate with the free energy identity and choosing $\eta>0$ sufficiently small, we recover the drag and viscous dissipations and obtain the following combined estimate.

\begin{proposition} \label{prop:BD-est}
Fix $\e>0$, and let $(f,\rho,u)$ be a sufficiently smooth positive solution to \eqref{main_sys}  with $\sigma\ge0$ on $[0,\infty)$. Assume that
\[
\alpha\ge1-\frac1d, \quad \gamma-1\le\alpha\le2\gamma-1.
\]
Then the following energy inequality holds
 \begin{equation}\label{eq:BD-e-iden3}
\frac{\dd}{\dt}\calE_\eta + c_1 \calD_\BD+ c_2 \calD_{\rm rot} + c_3\calD_{\rm vis} +c_4  \calD_{\drag} \leq 0.
\end{equation}
\end{proposition}

\begin{remark}
In the low-dimensional regimes considered in Proposition \ref{prop:main}, we get $\alpha=\gamma-1$. Thus, $q=\alpha-\gamma+1=0$. In particular, the estimate of the drag remainder does not use any additional positive or negative power bound on the density.
\end{remark}

%
%
%
%
%
%
%
%
%
%
%

 Since $\calE_\eta$ controls both $\calE_{\rm free}$ and $\calE_{\rm BD}$, the preceding pointwise inequality yields   $\nabla\rho^{\frac{\alpha+\gamma-1}{2}}\in L^2\lt(0,T;L^2(\T^d)\rt)$. It also gives
\begin{equation}\label{e:unif1}
\nabla\rho^{\alpha-\frac12}\in L^\infty\lt(0,T;L^2(\T^d)\rt) \quad \text{when }\alpha\neq\frac12
\end{equation}
and
\begin{equation}\label{e:unif2}
\nabla\log\rho\in L^\infty\lt(0,T;L^2(\T^d)\rt)  \quad \text{when }\alpha = \frac12.
\end{equation}

%
%
%
%
%
%
%
%
%
%
%

\subsection{Verification of the weighted density conditions}
It remains to verify the weighted density assumptions required in Theorem \ref{thm:main1}. This is the step in which the restrictions on the dimension and the exponents in Proposition \ref{prop:main} enter. In one dimension, the BD energy yields uniform pointwise upper and lower bounds on the density. In the two-dimensional borderline case, it provides uniform exponential integrability of $\log\rho$ and the required $\calA_2$ estimate.

\begin{proposition} 
\label{prop:veri-w-den-c}
Under the assumptions of Proposition \ref{prop:main}, Proposition \ref{prop:BD-est} applies. Moreover, for every finite $p,s>0$,
\bq\label{eq:all-pn-mom}
\sup_{t\ge0}\intt \lt(\rho(t,x)^p+\rho(t,x)^{-s}\rt)\dx<\infty.
\eq
In addition,
\[
\sup_{t\ge0}[\rho(t)^\alpha]_{\calA_2}<\infty.
\]
\end{proposition}
 
 \begin{proof}
We treat the one- and two-dimensional cases separately.
 
\medskip
\noindent\textit{Case 1: $d=1$, $0<\alpha\le\frac12$, and $\gamma=\alpha+1$.} Suppose first that $0<\alpha<\frac12$, and set
\[
\beta:=\frac12-\alpha>0.
\]
By \eqref{eq:BD-e-iden3} and \eqref{eq:BD-den-grad}, we have
\bq\label{eq:1d-np-grad}
\sup_{t\ge0}\|\pa_x\rho^{-\beta}(t)\|_{L^2(\T)}\le C.
\eq
We claim that this estimate, together with the mass constraint $\int_{\T}\rho(t,x)\,\dx=1$, implies uniform pointwise upper and lower bounds for $\rho$.

To obtain the lower bound, observe that, for each $t\ge0$, there exists a point $x_t^{+}\in\T$ such that $\rho(t,x_t^{+})\ge1$. Thus, $\rho(t,x_t^{+})^{-\beta}\le1$. The one-dimensional Morrey estimate and \eqref{eq:1d-np-grad} give
\[
\rho(t,x)^{-\beta}\le\rho(t,x_t^{+})^{-\beta}+C|x-x_t^{+}|^{\frac12}\le C
\]
for every $x\in\T$. Hence,
\[
\rho(t,x)\ge\underline\rho>0
\]
uniformly in $(t,x)$.

We next prove the upper bound. Let $x_t^{-}\in\T$ be a minimum point of $\rho(t,\cdot)^{-\beta}$ and set
\[
a_t:=\rho(t,x_t^{-})^{-\beta}.
\]
By the Morrey estimate,
\[
\rho(t,x)^{-\beta}\le a_t+C|x-x_t^{-}|^{\frac12}.
\]
 Consequently,
\[
1=\int_{\T}\rho(t,x)\,\dx\ge\int_{|x-x_t^{-}|\le\frac12}\lt(a_t+C|x-x_t^{-}|^{\frac12}\rt)^{-\frac1\beta}\dx.
\]
Since
\[
\lt( a+Cy^{\frac12} \rt)^{-\frac1\beta} \to C^{-\frac1\beta}y^{-\frac1{2\beta}} \quad \text{as } a\to0^+ 
\]
and $\frac1{2\beta}>1$, the monotone convergence theorem gives
\[
\lim_{a\to0^+} \int_0^{\frac12} \lt( a+Cy^{\frac12} \rt)^{-\frac1\beta}\,\dy = \infty.
\]
It follows that $a_t$ is bounded away from zero uniformly in time. Since
\[
a_t = \min_{x\in\T}\rho(t,x)^{-\beta} = \lt( \max_{x\in\T}\rho(t,x) \rt)^{-\beta},
\]
we conclude that
\[
\rho(t,x)\le\overline\rho<\infty \quad \text{uniformly in $(t,x)$.}
\]

Suppose now that $\alpha=\frac12$. By \eqref{e:unif2} we have
\bq\label{eq:1d-log-grad}
\sup_{t\ge0}\|\pa_x\log\rho(t)\|_{L^2(\T)}\le C.
\eq
For each $t\ge0$, the mass constraint implies the existence of points $x_t^{-},x_t^{+}\in\T$ such that
\[
\rho(t,x_t^{-})\le1\le\rho(t,x_t^{+}).
\]
Using \eqref{eq:1d-log-grad} and the one-dimensional Morrey estimate, we obtain $|\log\rho(t,x)|\le C$ for every $(t,x)\in[0,\infty)\times\T$. Thus, the same uniform pointwise bounds hold:
\bq\label{eq:1d-pt-den-b}
0<\underline\rho\le\rho(t,x)\le\overline\rho<\infty.
\eq

In either case, \eqref{eq:1d-pt-den-b} immediately gives
\[
\sup_{t\ge0}\int_{\T}\lt(\rho(t,x)^p+\rho(t,x)^{-s}\rt)\dx<\infty
\]
for every finite $p,s>0$. Moreover,
\[
[\rho(t)^\alpha]_{\calA_2}\le\lt(\frac{\overline\rho}{\underline\rho}\rt)^\alpha,
\]
and hence
\[
\sup_{t\ge0}[\rho(t)^\alpha]_{\calA_2}<\infty.
\]

\medskip
\noindent\textit{Case 2: $d=2$, $\alpha=\frac12$, and $\gamma=\frac32$.} By \eqref{e:unif2} we have
\bq\label{eq:log-den-L2-grad}
\sup_{t\ge0}\|\nabla\log\rho(t)\|_{L^2(\T^2)}\le C.
\eq
We first prove uniform positive-power and inverse-density bounds.  Set 
\[
\zeta:=\log\rho, \quad g:=\zeta-\zeta_{\T^2}, \quad \text{with } \zeta_{\T^2}:=\int_{\T^2}\zeta\,\dx.
\]
Then
\[
\int_{\T^2}g\,\dx=0, \quad \nabla g=\nabla\zeta.
\]
By the global Moser--Trudinger inequality,
\[
\int_{\T^2} \exp\lt( c_0\frac{|g|^2}{\|\nabla g\|_{L^2(\T^2)}^2} \rt)\dx \le C.
\]
Since \eqref{eq:log-den-L2-grad} gives $\|\nabla g\|_{L^2(\T^2)}\le K$ uniformly in time, we obtain
\[
\sup_{t\ge0} \int_{\T^2} \exp\lt( \frac{c_0}{K^2}|g|^2 \rt)\dx \le C.
\]
For every $a>0$, Young's inequality gives
\[
a|g| \le \frac{c_0}{2K^2}|g|^2 + \frac{a^2K^2}{2c_0}.
\]
Consequently,
\bq\label{eq:c-log-d-exp-m}
\sup_{t\ge0}\int_{\T^2}e^{a|\zeta-\zeta_{\T^2}|}\,\dx\le C_a.
\eq
Since
\[
\int_{\T^2}e^\zeta\,\dx=\int_{\T^2}\rho\,\dx=1,
\]
Jensen's inequality gives
\[
\zeta_{\T^2}\le\log\int_{\T^2}e^\zeta\,\dx=0.
\]
On the other hand, using \eqref{eq:c-log-d-exp-m} with $a=1$, we obtain
\[
1=e^{\zeta_{\T^2}}\int_{\T^2}e^{\zeta-\zeta_{\T^2}}\,\dx\le C e^{\zeta_{\T^2}}.
\]
This gives $-C\le\zeta_{\T^2}\le0$. Combining this estimate with \eqref{eq:c-log-d-exp-m}, we conclude that, for every finite $p,s>0$,
\[
\int_{\T^2}\rho^p\,\dx+\int_{\T^2}\rho^{-s}\,\dx =e^{p\zeta_{\T^2}}\int_{\T^2}e^{p(\zeta-\zeta_{\T^2})}\,\dx+e^{-s\zeta_{\T^2}}\int_{\T^2}e^{-s(\zeta-\zeta_{\T^2})}\,\dx  \le C_{p,s}.
\]
This proves \eqref{eq:all-pn-mom} in the two-dimensional case.

It remains to verify the $\calA_2$ condition. Let $Q\subset\T^2$ be a cube, and define
\[
\zeta_Q:=\frac1{|Q|}\int_Q\zeta\,\dx.
\]
The scaled local Moser--Trudinger inequality gives
\[
\frac1{|Q|}\int_Q\exp\lt(c\frac{|\zeta-\zeta_Q|^2}{\|\nabla\zeta\|_{L^2(Q)}^2}\rt)\dx\le C.
\]
If $\|\nabla\zeta\|_{L^2(Q)}=0$, the desired estimate is immediate. Otherwise, Young's inequality gives
\[
\frac12|\zeta-\zeta_Q| \le \frac{c}{2} \frac{|\zeta-\zeta_Q|^2}{\|\nabla\zeta\|_{L^2(Q)}^2} + \frac{\|\nabla\zeta\|_{L^2(Q)}^2}{8c}.
\]
Since $\|\nabla\zeta\|_{L^2(Q)}\le\|\nabla\zeta\|_{L^2(\T^2)}\le C$, the scaled local Moser--Trudinger inequality yields
\[
\frac1{|Q|} \int_Q e^{\frac12|\zeta-\zeta_Q|} \,\dx \le C.
\]
Hence,
\[
\lt(\frac1{|Q|}\int_Q\sqrt\rho\,\dx\rt)\lt(\frac1{|Q|}\int_Q\rho^{-\frac12}\,\dx\rt) =\lt(\frac1{|Q|}\int_Qe^{\frac12(\zeta-\zeta_Q)}\,\dx\rt)\lt(\frac1{|Q|}\int_Qe^{-\frac12(\zeta-\zeta_Q)}\,\dx\rt) \le C.
\]
Taking the supremum over all cubes $Q\subset\T^2$, we obtain
\[
\sup_{t\ge0}[\sqrt{\rho(t)}]_{\calA_2}<\infty.
\]
This completes the proof.
\end{proof}

%
%
%
%
%
%
%
%
%
%
%
 

 \subsection[Completion of the proof of Proposition \ref{prop:main}]{
    Completion of the proof of Proposition \ref{prop:main}
}

Under either set of assumptions in Proposition \ref{prop:main}, we have
\[
1-\frac1d\le\alpha=\gamma-1\le2\gamma-1.
\]
Therefore, Proposition \ref{prop:BD-est} applies and yields
\[
\sup_{t\ge0}\calE_\eta(t)+\int_0^\infty\lt[\calD_{\rm vis}(t)+\calD_{\rm BD}(t)+\calD_{\rm rot}(t)+\calD_{\rm drag}(t)\rt]\dt<\infty.
\]

By Proposition \ref{prop:veri-w-den-c}, the positive-power and inverse-density bounds and the $\calA_2$ condition required in Theorem \ref{thm:main1} are satisfied. When $d=1$, we may choose
\[
r=\frac32, \quad p=3, \quad s=3\alpha.
\]
When $d=2$, we choose
\[
r=\frac32, \quad r^\star=6, \quad p=3, \quad s=\frac32.
\]
Thus, all the assumptions of Theorem \ref{thm:main1} are satisfied.  This completes the proof of Proposition \ref{prop:main}.

%
%
%
%
%
%
%
%
%
%
%

\section{Construction of finite-energy BD entropy weak solutions}\label{sec:weak-sol}
In this section, we construct global finite-energy BD entropy weak solutions to \eqref{main_sys} in the diffusionless case $\sigma=0$ and in the low-dimensional regimes stated in Theorem \ref{thm:weak-exist}. The proof is based on a BD-compatible approximation scheme with artificial pressure, regularized viscosity, capillarity, and damping terms. The locally averaged Brinkman force is treated as a lower-order perturbation, since the spatial regularization parameter $\e>0$ is fixed throughout the construction.

The purpose of this section is to prove the existence assertion of Theorem \ref{thm:weak-exist} in the diffusionless case $\sigma=0$. The exponential relaxation of the weak solutions constructed here will be proved in Section \ref{sec:weak-relax-diff}.

%
%
%
%
%
%
%
%
%
%
%

\subsection{Approximation scheme and a BD-compatible regularized system}

To prove the diffusionless existence assertion of Theorem \ref{thm:weak-exist}, we introduce the approximation parameters
\[
\mathbf a:=(\zeta,\tau),\quad \zeta,\tau>0.
\]
The parameter $\zeta$ regularizes the viscosity coefficient, while the parameter $\tau$ is used for the artificial pressure, the capillarity term, and the linear damping term. We also use an auxiliary augmented Galerkin approximation indexed by the Galerkin dimension $N$. The limiting procedure is performed in the order
\[
N\to\infty,\quad \tau\to0,\quad \zeta\to0.
\]
The approximation scheme is based on the augmented BD-compatible construction developed in \cite{BDZ15,BVY22}. The additional kinetic equation is incorporated through a decoupling and fixed-point argument.

For $\zeta>0$, define
\[
\nu_\zeta(\rho):=\rho^\alpha+\zeta\rho, \quad \lambda_\zeta(\rho):=2(\rho\nu_\zeta'(\rho)-\nu_\zeta(\rho))=2(\alpha-1)\rho^\alpha.
\]
Thus, the regularized viscosity coefficients preserve the BD relation. Moreover,
\[
\nu_\zeta'(\rho)=\alpha\rho^{\alpha-1}+\zeta\ge\zeta.
\]
We denote the regularized stress tensor by
\[
\mathbb S_\zeta(\rho,u):=2\nu_\zeta(\rho)D u+\lambda_\zeta(\rho)(\nabla\cdot u)\mathbb I.
\]

We also define the regularized BD potential by
\[
s_\zeta'(\rho):=\frac{\nu_\zeta'(\rho)}{\rho},\quad \phi_\zeta:=2s_\zeta,\quad \phi_\zeta'(\rho)=\frac{2\nu_\zeta'(\rho)}{\rho}.
\]
Then
\[
\rho\nabla\phi_\zeta(\rho)=2\nabla\nu_\zeta(\rho).
\]

The artificial pressure is
\[
p_\tau(\rho):=p(\rho)+\tau\rho^{10}=\rho^\gamma+\tau\rho^{10}.
\]
We denote by $\pi_{10}$ the corresponding pressure potential:
\[
\pi_{10}(\rho):=\rho\int_1^\rho\frac{s^{10}-1}{s^2}\, \ds =\frac{\rho^{10}-10\rho+9}{9}.
\]

To define the capillarity regularization, set
\[
K_\zeta(\rho):=\frac{4|\nu_\zeta'(\rho)|^2}{\rho},\quad 
Y_\zeta(\rho):=\int_1^\rho\sqrt{K_\zeta(s)}\, \ds,
\quad \text{and} \quad 
\mathfrak K_\zeta(\rho):=\rho\nabla\lt(\sqrt{K_\zeta(\rho)}\,\Delta Y_\zeta(\rho)\rt).
\]
This choice is compatible with the BD structure; see \cite{BVY22}. We also introduce
\[
j_\zeta(\rho):=\int_1^\rho\frac{\nu_\zeta'(z)}{z^2}\,\dz,\quad J_\zeta(\rho):=\int_1^\rho j_\zeta(s)\, \ds.
\]
Since $s_\zeta'(\rho)=\nu_\zeta'(\rho)/\rho$, we get $\rho\nabla j_\zeta(\rho)=\nabla s_\zeta(\rho)$. This identity will be used to absorb the damping contribution in the BD entropy balance.

For the regularized kinetic--fluid coupling, we define
\[
u_{\mathbf a,\e}^\e:=\frac{(\rho_{\mathbf a}u_{\mathbf a})_\e}{(\rho_{\mathbf a})_\e},
\quad
F_{\mathbf a}:=\intr (v-u_{\mathbf a,\e}^\e)f_{\mathbf a}\,\dv.
\]
Since $\intt\rho_{\mathbf a}(t,x)\,\dx=1$, the strict positivity of $\theta_\e$ gives
\[
(\rho_{\mathbf a})_\e(t,x)\ge c_\e:=\min_{\T^d}\theta_\e>0.
\]
Hence $u_{\mathbf a,\e}^\e$ is well defined independently of the possible presence of vacuum.

We consider the regularized system
\begin{align}\label{eq:reg-sys}
\begin{aligned}
&\pa_t f_{\mathbf a}+v\cdot\nabla_x f_{\mathbf a} =(\rho_{\mathbf a})_\e\nabla_v\cdot\lt[(v-u_{\mathbf a,\e}^\e)f_{\mathbf a}\rt], \\
&\pa_t\rho_{\mathbf a}+\nabla\cdot(\rho_{\mathbf a}u_{\mathbf a})=0, \\
&\pa_t(\rho_{\mathbf a}u_{\mathbf a}) +\nabla\cdot(\rho_{\mathbf a}u_{\mathbf a}\otimes u_{\mathbf a}) +\nabla p_\tau(\rho_{\mathbf a})  
 = \nabla\cdot\mathbb S_\zeta(\rho_{\mathbf a},u_{\mathbf a}) + \rho_{\mathbf a}(F_{\mathbf a})_\e + \tau\mathfrak K_\zeta(\rho_{\mathbf a}) - \tau u_{\mathbf a}.
\end{aligned}
\end{align}
The three terms multiplied by $\tau$ play different roles. The artificial pressure controls large values of the density, the capillarity term yields higher-order compactness while preserving the BD structure, and the linear damping term is incorporated into the approximate BD entropy through an exact time derivative. Since these terms will be removed simultaneously, we use the same parameter $\tau$.

We now provide the approximation of the initial data.  
\begin{lemma}\label{lem:init-app}
There exist smooth approximate initial data $f_{0,n}\ge0$, $\rho_{0,n}>0$, and $u_{0,n}$ such that
\[
f_{0,n},\ \rho_{0,n},\ u_{0,n}\in C^\infty,\quad
\inttr f_{0,n}\,\dx\dv=1,\quad
\intt\rho_{0,n}\,\dx=1.
\]
Moreover,
\[
f_{0,n}\to f_0 \quad\text{in }W_2(\T^d\times\R^d), \quad \rho_{0,n}\to\rho_0 \quad\text{in }L^\gamma(\T^d),\quad
\sqrt{\rho_{0,n}}u_{0,n}\to\sqrt{\rho_0}u_0
\quad\text{in }L^2(\T^d),
\]
and
\[
\sqrt{\rho_{0,n}}\lt(u_{0,n}+\nabla\phi(\rho_{0,n})\rt)
\to
\sqrt{\rho_0}\lt(u_0+\nabla\phi(\rho_0)\rt)
\quad\text{in }L^2(\T^d).
\]
\end{lemma}

The proof of Lemma \ref{lem:init-app} is postponed to Appendix \ref{app:init}. There, we also establish the diagonal compatibility of the smooth initial data with the $\zeta$- and $\tau$-dependent regularized energies; see Lemma \ref{lem:init-energy-diag}. This compatibility will be used in the final approximation to choose sequences $\zeta_n\to0$ and $\tau_n\to0$ along the smooth initial data. For each fixed pair $(\zeta,\tau)$ and each smooth initial datum constructed in Lemma \ref{lem:init-app}, the regularized system \eqref{eq:reg-sys} is first solved by an auxiliary augmented Galerkin approximation with a vanishing
hyperregularization term, described in Appendix \ref{app:aug}.

%
%
%
%
%
%
%
%
%
%
%

\subsection{Approximate energy and BD entropy estimates} 

We derive the free energy estimate and the BD entropy estimate for the regularized system \eqref{eq:reg-sys}. The estimates are first obtained at the level of smooth positive solutions. After the auxiliary construction in Section \ref{ssec:aux-aug-app}, the same estimates pass to the regularized weak solutions by lower semicontinuity.

We recall the artificial pressure potential
\[
\pi_{10}(\rho)=\frac{\rho^{10}-10\rho+9}{9}.
\]
Then, $\rho \pi_{10}'(\rho)-\pi_{10}(\rho)=\rho^{10}-1$. Since the total fluid mass is conserved, the affine terms in $\pi_{10}$ do not affect the energy identity.

For the regularized system, we use the following approximate analogues of the free energy and the dissipations introduced in the Introduction:
\[
\calE_{{\rm free},\mathbf a}(t) := \frac12\inttr |v|^2 f_{\mathbf a}\,\dx\dv +\intt\lt[ \frac12\rho_{\mathbf a}|u_{\mathbf a}|^2+\pi(\rho_{\mathbf a}) +\tau \pi_{10}(\rho_{\mathbf a}) +\frac{\tau}{2}|\nabla Y_\zeta(\rho_{\mathbf a})|^2 \rt]\dx .
\]
We also define
\[
\calD_{{\rm vis},\mathbf a}(t) := \intt\lt[ 2\nu_\zeta(\rho_{\mathbf a})|D u_{\mathbf a}|^2 +\lambda_\zeta(\rho_{\mathbf a})|\nabla\cdot u_{\mathbf a}|^2 \rt]\dx,
\]
\[
\calD_{{\rm drag},\mathbf a}(t) := \inttr(\rho_{\mathbf a})_\e |v-u_{\mathbf a,\e}^\e|^2 f_{\mathbf a}\,\dx\dv,
\quad
\calD_{{\rm damp},\mathbf a}(t):=\tau\intt |u_{\mathbf a}|^2\,\dx.
\]

\begin{proposition}\label{prop:reg-energy}
Let $(f_{\mathbf a},\rho_{\mathbf a},u_{\mathbf a})$ be a sufficiently smooth positive solution to \eqref{eq:reg-sys}. Then
\bq\label{eq:reg-energy}
\frac{\dd}{\dt}\calE_{{\rm free},\mathbf a}(t) +\calD_{{\rm vis},\mathbf a}(t) +\calD_{{\rm drag},\mathbf a}(t) +\calD_{{\rm damp},\mathbf a}(t)=0.
\eq
Consequently,
\[
\sup_{0\le t\le T}\calE_{{\rm free},\mathbf a}(t) +\int_0^T\lt[ \calD_{{\rm vis},\mathbf a}(t) +\calD_{{\rm drag},\mathbf a}(t) +\calD_{{\rm damp},\mathbf a}(t) \rt]\dt \le \calE_{{\rm free},\mathbf a}(0).
\]
\end{proposition}

\begin{proof}
Multiplying the kinetic equation in \eqref{eq:reg-sys} by $\frac{|v|^2}2$ and integrating over $\T^d\times\R^d$, we obtain
\[
\frac{\dd}{\dt}\frac12\inttr |v|^2 f_{\mathbf a}\,\dx\dv = -\inttr(\rho_{\mathbf a})_\e v\cdot(v-u_{\mathbf a,\e}^\e)f_{\mathbf a}\,\dx\dv.
\]
We next multiply the momentum equation by $u_{\mathbf a}$. Using the continuity equation, we find
\[
\frac{\dd}{\dt}\intt\lt[ \frac12\rho_{\mathbf a}|u_{\mathbf a}|^2+\pi(\rho_{\mathbf a}) +\tau \pi_{10}(\rho_{\mathbf a}) \rt]\dx
+\calD_{{\rm vis},\mathbf a} +\calD_{{\rm damp},\mathbf a}  
= \intt\rho_{\mathbf a}u_{\mathbf a}\cdot(F_{\mathbf a})_\e\,\dx + \tau\intt\mathfrak K_\zeta(\rho_{\mathbf a})\cdot u_{\mathbf a}\,\dx.
\]
By the symmetry of the mollifier,
\[
\intt\rho_{\mathbf a}u_{\mathbf a}\cdot(F_{\mathbf a})_\e\,\dx = \intt(\rho_{\mathbf a}u_{\mathbf a})_\e\cdot F_{\mathbf a}\,\dx
= \inttr(\rho_{\mathbf a})_\e u_{\mathbf a,\e}^\e\cdot(v-u_{\mathbf a,\e}^\e)f_{\mathbf a}\,\dx\dv.
\]
Moreover, using the continuity equation and $Y_\zeta'(\rho)=\sqrt{K_\zeta(\rho)}$, we obtain
\begin{align*}
\intt\mathfrak K_\zeta(\rho_{\mathbf a})\cdot u_{\mathbf a}\,\dx &= \intt\rho_{\mathbf a}u_{\mathbf a}\cdot\nabla\lt[ \sqrt{K_\zeta(\rho_{\mathbf a})}\Delta Y_\zeta(\rho_{\mathbf a}) \rt]\dx \\
&= - \intt\nabla\cdot(\rho_{\mathbf a}u_{\mathbf a}) \sqrt{K_\zeta(\rho_{\mathbf a})}\Delta Y_\zeta(\rho_{\mathbf a})\,\dx \\
& = \intt\pa_tY_\zeta(\rho_{\mathbf a})\Delta Y_\zeta(\rho_{\mathbf a})\,\dx \\
&= -\frac{\dd}{\dt}\frac12\intt|\nabla Y_\zeta(\rho_{\mathbf a})|^2\,\dx.
\end{align*}
Adding the kinetic and fluid identities, the coupling terms combine as
\begin{align*}
&-\inttr(\rho_{\mathbf a})_\e v\cdot(v-u_{\mathbf a,\e}^\e)f_{\mathbf a}\,\dx\dv +\inttr(\rho_{\mathbf a})_\e u_{\mathbf a,\e}^\e\cdot(v-u_{\mathbf a,\e}^\e)f_{\mathbf a}\,\dx\dv \\
&\quad = -\inttr(\rho_{\mathbf a})_\e |v-u_{\mathbf a,\e}^\e|^2 f_{\mathbf a}\,\dx\dv.
\end{align*}
This proves \eqref{eq:reg-energy}.
\end{proof}

In the admissible regimes \eqref{eq:admi-case}, the viscous dissipation is nonnegative; indeed, in one dimension $2\nu_\zeta+\lambda_\zeta=2\rho\nu_\zeta'\ge0$, while in two dimensions with $\alpha=\frac12$ one has
\[
2\nu_\zeta(\rho)|D u|^2+\lambda_\zeta(\rho)|\nabla\cdot u|^2 = 2\nu_\zeta(\rho)|D^0u|^2 +\zeta\rho|\nabla\cdot u|^2.
\]

We next derive the BD entropy estimate. Define the regularized effective velocity
\[
w_{\mathbf a}:=u_{\mathbf a}+\nabla\phi_\zeta(\rho_{\mathbf a}) =u_{\mathbf a}+2\nabla s_\zeta(\rho_{\mathbf a}).
\]
The pressure dissipation is
\[
\calD_{{\rm BD},\mathbf a}(t) := \intt\nabla p_\tau(\rho_{\mathbf a})\cdot\nabla\phi_\zeta(\rho_{\mathbf a})\,\dx.
\]
Since
\[
p_\tau'(\rho)=\gamma\rho^{\gamma-1}+10\tau\rho^9, \quad \phi_\zeta'(\rho)=2\alpha\rho^{\alpha-2}+\frac{2\zeta}{\rho},
\]
we have
\[
\calD_{{\rm BD},\mathbf a}(t) = 2\intt \lt(\gamma\rho_{\mathbf a}^{\gamma-1}+10\tau\rho_{\mathbf a}^9\rt) \lt(\alpha\rho_{\mathbf a}^{\alpha-2}+\frac{\zeta}{\rho_{\mathbf a}}\rt) |\nabla\rho_{\mathbf a}|^2\,\dx.
\]
In particular,
\[
\calD_{{\rm BD},\mathbf a}(t) \ge 2\alpha\gamma\intt\rho_{\mathbf a}^{\alpha+\gamma-3}|\nabla\rho_{\mathbf a}|^2\,\dx = \frac{8\alpha\gamma}{(\alpha+\gamma-1)^2} \intt\lt|\nabla\rho_{\mathbf a}^{\frac{\alpha+\gamma-1}{2}}\rt|^2\,\dx.
\]
Under the relation $\gamma=\alpha+1$, this becomes
\[
\calD_{{\rm BD},\mathbf a}(t) \ge \frac{2(\alpha+1)}{\alpha}\intt|\nabla\rho_{\mathbf a}^\alpha|^2\,\dx.
\]

We recall the generalized Bohm identity \cite{BCNV16,BVY22}:
\bq\label{eq:Bohm}
\mathfrak K_\zeta(\rho) = 2\nabla\cdot\lt[\nu_\zeta(\rho)\nabla^2(2s_\zeta(\rho))\rt] +\nabla\lt[\lambda_\zeta(\rho)\Delta(2s_\zeta(\rho))\rt].
\eq
Consequently,
\[
-2\tau\intt\mathfrak K_\zeta(\rho_{\mathbf a})\cdot\nabla s_\zeta(\rho_{\mathbf a})\,\dx = \calD_{{\rm cap},\mathbf a}(t),
\]
where
\[
\calD_{{\rm cap},\mathbf a}(t) := 4\tau\intt\lt[ 2\nu_\zeta(\rho_{\mathbf a})|\nabla^2s_\zeta(\rho_{\mathbf a})|^2 +\lambda_\zeta(\rho_{\mathbf a})|\Delta s_\zeta(\rho_{\mathbf a})|^2 \rt]\dx.
\]
The capillarity dissipation is nonnegative in the regimes considered here. Indeed, when $d=1$, this follows from $2\nu_\zeta(\rho)+\lambda_\zeta(\rho)=2\rho\nu_\zeta'(\rho)\ge0$. When $d\ge2$, writing
\[
2\nu_\zeta(\rho)|\nabla^2 g|^2+\lambda_\zeta(\rho)|\Delta g|^2 = 2\nu_\zeta(\rho)|(\nabla^2g)^0|^2 +\lt(\frac2d\nu_\zeta(\rho)+\lambda_\zeta(\rho)\rt) |\Delta g|^2,
\]
the same condition $\alpha\ge1-\frac1d$ gives nonnegativity.

We also define the rotational dissipation
\[
\calD_{{\rm rot},\mathbf a}(t) := 2\intt\nu_\zeta(\rho_{\mathbf a})|A u_{\mathbf a}|^2\,\dx, \quad A u:=\frac12\lt(\nabla u-(\nabla u)^{\mathsf T}\rt).
\]

The contribution of the linear damping term can be incorporated into the BD energy through an exact time derivative. Recall
\[
j_\zeta(\rho)=\int_1^\rho\frac{\nu_\zeta'(z)}{z^2}\,\dz, \quad J_\zeta(\rho)=\int_1^\rho j_\zeta(s)\, \ds.
\]
Then
\[
J_\zeta''(\rho)=\frac{\nu_\zeta'(\rho)}{\rho^2}>0,
\quad
J_\zeta(1)=J_\zeta'(1)=0,
\]
and thus $J_\zeta(\rho)\ge0$ for every $\rho>0$. Moreover, since $\rho\nabla j_\zeta(\rho)=\nabla s_\zeta(\rho)$, the continuity equation gives
\bq
\label{eq:linear-damp}
\intt u_{\mathbf a}\cdot\nabla s_\zeta(\rho_{\mathbf a})\,\dx = \intt\rho_{\mathbf a}u_{\mathbf a}\cdot\nabla j_\zeta(\rho_{\mathbf a})\,\dx  = -\intt\nabla\cdot(\rho_{\mathbf a}u_{\mathbf a})j_\zeta(\rho_{\mathbf a})\,\dx = \frac{\dd}{\dt}\intt J_\zeta(\rho_{\mathbf a})\,\dx.
\eq

We introduce a modified BD energy
\[
\widetilde{\calE}_{{\rm BD},\mathbf a}(t) := \frac12\inttr |v|^2 f_{\mathbf a}\,\dx\dv +\intt\lt[ \frac12\rho_{\mathbf a}|w_{\mathbf a}|^2 +\pi(\rho_{\mathbf a}) +\tau \pi_{10}(\rho_{\mathbf a}) +\frac{\tau}{2}|\nabla Y_\zeta(\rho_{\mathbf a})|^2 +2\tau J_\zeta(\rho_{\mathbf a}) \rt]\dx.
\]
Finally, define the kinetic remainder by
\[
\calR_{{\rm kin},\mathbf a}(t) := \intt\rho_{\mathbf a}\nabla\phi_\zeta(\rho_{\mathbf a})\cdot(F_{\mathbf a})_\e\,\dx.
\]

\begin{proposition}\label{prop:reg-BD}
Let $(f_{\mathbf a},\rho_{\mathbf a},u_{\mathbf a})$ be a sufficiently smooth positive solution to \eqref{eq:reg-sys}. Then, for every $0\le s\le t\le T$,
\begin{align}\label{eq:reg-BD-int}
\begin{aligned}
&\widetilde{\calE}_{{\rm BD},\mathbf a}(t) +\int_s^t\lt[ \calD_{{\rm BD},\mathbf a}(\tau) +\calD_{{\rm rot},\mathbf a}(\tau) +\calD_{{\rm drag},\mathbf a}(\tau) +\calD_{{\rm cap},\mathbf a}(\tau) +\calD_{{\rm damp},\mathbf a}(\tau) \rt]\dd\tau \\
&\quad = \widetilde{\calE}_{{\rm BD},\mathbf a}(s) +\int_s^t\calR_{{\rm kin},\mathbf a}(\tau)\,\dd\tau.
\end{aligned}
\end{align}
\end{proposition}

\begin{proof}
For simplicity, we omit the subscript $\mathbf a$ throughout the proof; all energies and dissipations in this proof refer to the corresponding regularized quantities. Recall that
\[
w=u+\nabla\phi_\zeta(\rho)=u+2\nabla s_\zeta(\rho).
\]
We use two identities. The first one is the free energy identity from Proposition \ref{prop:reg-energy}:
\bq\label{eq:BD-proof-energy}
\frac{\dd}{\dt}\calE_{\rm free} +\calD_{\rm vis} +\calD_{\rm drag} +\calD_{\rm damp}=0.
\eq
The second one is the BD correction identity obtained by testing the momentum equation with $\nabla\phi_\zeta(\rho)$.

Testing the momentum equation by $\nabla\phi_\zeta(\rho)$, using the continuity equation, and differentiating
\[
\intt\lt[
\rho u\cdot\nabla\phi_\zeta(\rho)
+\frac12\rho|\nabla\phi_\zeta(\rho)|^2
\rt]\dx,
\]
the inertial part gives
\[
\intt\lt[ \pa_t(\rho u)+\nabla\cdot(\rho u\otimes u) \rt]\cdot\nabla\phi_\zeta(\rho)\,\dx = \frac{\dd}{\dt}\intt\lt[ \rho u\cdot\nabla\phi_\zeta(\rho) +\frac12\rho|\nabla\phi_\zeta(\rho)|^2 \rt]\dx.
\]
This is the usual BD cancellation: the terms containing the material derivative of $\nabla\phi_\zeta(\rho)$ are compensated by the time derivative of $\frac12\intt\rho|\nabla\phi_\zeta(\rho)|^2\,\dx$.

The pressure term yields
\[
\intt\nabla p_\tau(\rho)\cdot\nabla\phi_\zeta(\rho)\,\dx=\calD_{\rm BD}.
\]
For the viscous term, integration by parts and the BD relation give
\[
\intt\mathbb S_\zeta(\rho,u):\nabla^2\phi_\zeta(\rho)\,\dx = \calD_{\rm rot}-\calD_{\rm vis}.
\]
The capillarity term is treated by the generalized Bohm identity \eqref{eq:Bohm}. Since $\nabla\phi_\zeta(\rho)=2\nabla s_\zeta(\rho)$, we have
\[
\tau\intt\mathfrak K_\zeta(\rho)\cdot\nabla\phi_\zeta(\rho)\,\dx = 2\tau\intt\mathfrak K_\zeta(\rho)\cdot\nabla s_\zeta(\rho)\,\dx = -\calD_{\rm cap}.
\]
The force term gives
\[
\intt\rho(F)_\e\cdot\nabla\phi_\zeta(\rho)\,\dx=\calR_{\rm kin}.
\]
Finally, the linear damping contribution is
\[
-\tau\intt u\cdot\nabla\phi_\zeta(\rho)\,\dx = -2\tau\intt u\cdot\nabla s_\zeta(\rho)\,\dx = -\frac{\dd}{\dt}\lt(2\tau\intt J_\zeta(\rho)\,\dx\rt),
\]
where we used \eqref{eq:linear-damp}. Combining these identities gives the BD correction identity
\bq\label{eq:BD-proof-corr}
\frac{\dd}{\dt}\intt\lt[ \rho u\cdot\nabla\phi_\zeta(\rho) +\frac12\rho|\nabla\phi_\zeta(\rho)|^2 +2\tau J_\zeta(\rho) \rt]\dx   +\calD_{\rm BD} +\calD_{\rm rot} -\calD_{\rm vis} +\calD_{\rm cap} = \calR_{\rm kin}.
\eq
Adding \eqref{eq:BD-proof-corr} to the free energy identity \eqref{eq:BD-proof-energy}, we obtain 
\[
\frac{\dd}{\dt}\widetilde{\calE}_{{\rm BD},\mathbf a}(t) +\calD_{{\rm BD},\mathbf a}(t) +\calD_{{\rm rot},\mathbf a}(t) +\calD_{{\rm drag},\mathbf a}(t) +\calD_{{\rm cap},\mathbf a}(t) +\calD_{{\rm damp},\mathbf a}(t)
= \calR_{{\rm kin},\mathbf a}(t).
\]
 Integrating in time gives \eqref{eq:reg-BD-int}.
\end{proof}

We now estimate the remainder generated by the kinetic--fluid coupling. This is the only additional term compared with the standard BD-compatible approximation scheme.

\begin{lemma}\label{lem:kin-rem}
Suppose that $\gamma=\alpha+1$, $0<\alpha\le\frac12$. Then
\[
\lt|\calR_{{\rm kin},\mathbf a}(t)\rt| \le \frac12\calD_{{\rm BD},\mathbf a}(t) + C_\e\lt(1+\calE_{{\rm free},\mathbf a}(0)\rt) \calD_{{\rm drag},\mathbf a}(t)
\]
for almost every $t\ge0$. The constant is independent of $\zeta$ and $\tau$.
\end{lemma}

\begin{proof}
For notational simplicity, we omit the subscript $\mathbf a$. By Young's inequality,
\bq\label{eq:kin-rem-young}
\lt| \intt\rho\nabla\phi_\zeta(\rho)\cdot F_\e\,\dx \rt|
\le \frac12\intt\nabla p_\tau(\rho)\cdot\nabla\phi_\zeta(\rho)\,\dx +\frac12\intt Q_{\zeta,\tau}(\rho)|F_\e|^2\,\dx,
\eq
where
\[
Q_{\zeta,\tau}(\rho):=\frac{\rho^2\phi_\zeta'(\rho)}{p_\tau'(\rho)}.
\]
Since
\[
\phi_\zeta'(\rho)=2\alpha\rho^{\alpha-2}+\frac{2\zeta}{\rho}, \quad p_\tau'(\rho)=\gamma\rho^{\gamma-1}+10\tau\rho^9,
\]
we obtain
\[
Q_{\zeta,\tau}(\rho) \le \frac{2\alpha}{\gamma}\rho^{\alpha-\gamma+1} + \frac{2\zeta}{\gamma}\rho^{2-\gamma}.
\]
Under the relation $\gamma=\alpha+1$, this becomes $Q_{\zeta,\tau}(\rho)\le C(1+\zeta\rho^{1-\alpha})$. Assuming without loss of generality that $0<\zeta\le1$, we infer that $Q_{\zeta,\tau}(\rho)\le C(1+\rho^{1-\alpha})$. Since $1-\alpha\le\gamma$, the energy estimate yields
\[
\intt Q_{\zeta,\tau}(\rho)\,\dx \le C\lt(1+\intt\rho^\gamma\,\dx\rt) \le C\lt(1+\calE_{{\rm free},\mathbf a}(0)\rt).
\]

We next define
\[
\rho_f:=\intr f\,\dv,\quad E_f:=\intr |v-u_\e^\e|^2f\,\dv.
\]
By the Cauchy--Schwarz inequality,
\[
|F|^2 = \lt|\intr (v-u_\e^\e)f\,\dv\rt|^2 \le \rho_f E_f.
\]
Using the nonnegativity of $\theta_\e$ and applying the Cauchy--Schwarz inequality once more, we obtain $|F_\e|^2\le(\rho_f)_\e (E_f)_\e$. By the symmetry of the mollifier,
\[
\intt Q_{\zeta,\tau}(\rho)|F_\e|^2\,\dx \le  \|(\rho_f)_\e\|_{L^\infty} \intt\lt(Q_{\zeta,\tau}(\rho)\rt)_\e E_f\,\dx \le
\|(\rho_f)_\e\|_{L^\infty} \lt\| \frac{(Q_{\zeta,\tau}(\rho))_\e}{\rho_\e} \rt\|_{L^\infty} \calD_{\rm drag}.
\]
Since $\rho_\e\ge c_\e>0$, we obtain
\[
\lt\| \frac{(Q_{\zeta,\tau}(\rho))_\e}{\rho_\e} \rt\|_{L^\infty} \le \frac{\|\theta_\e\|_{L^\infty}}{c_\e} \intt Q_{\zeta,\tau}(\rho)\,\dx \le C_\e\lt(1+\calE_{{\rm free},\mathbf a}(0)\rt).
\]
Combining the preceding estimates with \eqref{eq:kin-rem-young}, we conclude the desired result.  
\end{proof}

We now combine the free energy estimate with a sufficiently small multiple of the modified BD entropy estimate. Define
\[
\calE_{\eta,\mathbf a}(t):= \calE_{{\rm free},\mathbf a}(t) +\eta\widetilde{\calE}_{{\rm BD},\mathbf a}(t).
\]

\begin{proposition}\label{prop:comb-est}
Let $(f_{\mathbf a},\rho_{\mathbf a},u_{\mathbf a})$ be a sufficiently smooth positive solution to \eqref{eq:reg-sys}. Assume that
\[
N_0:= \calE_{{\rm free},\mathbf a}(0) + \widetilde{\calE}_{{\rm BD},\mathbf a}(0) <\infty.
\]
Then there exists $\eta>0$, depending only on $\e$ and $N_0$, such that, for every $0\le s\le t\le T$,
\begin{align}\label{eq:comb-est}
\begin{aligned}
&\calE_{\eta,\mathbf a}(t) +\int_s^t\lt[ \calD_{{\rm vis},\mathbf a}(\tau) +\frac12\calD_{{\rm drag},\mathbf a}(\tau) +\frac{\eta}{2}\calD_{{\rm BD},\mathbf a}(\tau) +\eta\calD_{{\rm rot},\mathbf a}(\tau) +\eta\calD_{{\rm cap},\mathbf a}(\tau) +\calD_{{\rm damp},\mathbf a}(\tau) \rt]\dd\tau \\
&\quad \le \calE_{\eta,\mathbf a}(s).
\end{aligned}
\end{align}
In particular, if the initial energies are bounded uniformly along a family of regularized solutions, then $\eta>0$ can be chosen uniformly and the corresponding estimates are uniform on every finite time interval.
\end{proposition}

\begin{proof}
Choose $\eta>0$ sufficiently small so that $\eta C_\e(1+N_0)\le\frac12$. Adding the free energy inequality and $\eta$ times \eqref{eq:reg-BD-int}, and using Lemma \ref{lem:kin-rem}, we obtain
\[
\begin{aligned}
&\calE_{\eta,\mathbf a}(t) +\int_s^t\calD_{{\rm vis},\mathbf a}(\tau)\,\dd\tau +\frac12\int_s^t\calD_{{\rm drag},\mathbf a}(\tau)\,\dd\tau +\frac{\eta}{2}\int_s^t\calD_{{\rm BD},\mathbf a}(\tau)\,\dd\tau \\
&\quad +\eta\int_s^t\calD_{{\rm rot},\mathbf a}(\tau)\,\dd\tau +\eta\int_s^t\calD_{{\rm cap},\mathbf a}(\tau)\,\dd\tau +\int_s^t\calD_{{\rm damp},\mathbf a}(\tau)\,\dd\tau \\
&\quad \quad \le \calE_{\eta,\mathbf a}(s).
\end{aligned}
\]
This proves \eqref{eq:comb-est}. The uniform version follows immediately if $N_0$ is bounded uniformly along the approximating family.
\end{proof}

\begin{proposition}\label{prop:app-density-conseq}
Assume that one of the admissible cases in \eqref{eq:admi-case} holds. Let $(f_{\mathbf a},\rho_{\mathbf a},u_{\mathbf a})$ be a regularized solution satisfying the combined estimate \eqref{eq:comb-est}, with
$\mathbf a=(\zeta,\tau)$ and $0<\zeta,\tau\le1$. Then the following density consequences hold uniformly with respect to $\zeta$ and $\tau$ on every finite time interval.

If $d=1$ and $0<\alpha\le\frac12$, then there exist constants $\underline\rho,\overline\rho>0$, depending only on $\e$ and the uniform initial energy bound, such that
\[
0<\underline\rho \le \rho_{\mathbf a}(t,x) \le \overline\rho<\infty
\]
for a.e. $(t,x)\in (0,\infty) \times\T$.

If $d=2$ and $\alpha=\frac12$, then, for every finite $p>0$,
\[
\esssup_{ t\ge0} \int_{\T^2} \lt( \rho_{\mathbf a}(t,x)^p+\rho_{\mathbf a}(t,x)^{-p} \rt) \dx \le C_{\e,p},
\]
and
\[
\esssup_{t\ge 0} [\rho_{\mathbf a}(t)^\alpha]_{\calA_2} \le C_{\e}.
\]
\end{proposition}
\begin{proof}
The proof is identical to the verification of the weighted density conditions in Proposition \ref{prop:veri-w-den-c}. Indeed, the combined estimate \eqref{eq:comb-est} gives the same singular BD controls as in the a priori estimate, namely
\[
\nabla\rho_{\mathbf a}^{\alpha-\frac12} \in L^\infty(0,T;L^2(\T^d)) \quad\text{if }\alpha\neq\frac12,
\]
and
\[
\nabla\log\rho_{\mathbf a} \in L^\infty(0,T;L^2(\T^d)) \quad\text{if }\alpha=\frac12.
\]
The additional $\zeta$- and $\tau$-regularizing terms are nonnegative in the energy and entropy estimates and therefore do not affect the argument. The one-dimensional pointwise bounds and the two-dimensional finite moment and $\calA_2$ bounds then follow exactly as in Proposition \ref{prop:veri-w-den-c}.
\end{proof}

%
%
%
%
%
%
%
%
%
%
%

\subsection{An effective-velocity compactness criterion} 

We provide an effective-velocity compactness criterion which will be used in the auxiliary limit $N\to\infty$, in the limit $\tau\to0$, and in the limit $\zeta\to0$. The statement is a form of the BD compactness mechanism for density-dependent viscosities, in the spirit of \cite{BVY22}.  Its main role is to combine the density compactness with the weak compactness of the velocity gradients. The almost everywhere convergence of the velocities, which is needed to identify the convective term, will be verified separately from the corresponding approximate momentum equations.

\begin{lemma}\label{lem:BD-convec-comp}
Let $d=1,2$. Assume either $d=1$ with $0<\alpha\le\frac12$, or $d=2$ with $\alpha=\frac12$. Let $\zeta_n\ge0$ and $\zeta_n\to\zeta\ge0$, and set
\[
\nu_n(r):=r^\alpha+\zeta_n r, \quad \lambda_n(r):=2(\alpha-1)r^\alpha, \quad \phi_n'(r):=\frac{2\nu_n'(r)}{r}.
\]
Let $(\rho_n,u_n)$ be a sequence satisfying
\[
\rho_n\ge0, \quad \pa_t\rho_n+\nabla\cdot(\rho_nu_n)=0 \quad\text{in }\calD'((0,T)\times\T^d),
\]
and assume that the following bounds hold uniformly in $n$:
\bq\label{eq:BD-compact-hyp-energy}
 \esssup_{0\le t \le T} \intt \lt[ \rho_n^\gamma + \rho_n|u_n|^2 + \rho_n|u_n+\nabla\phi_n(\rho_n)|^2 \rt]\dx +
\int_0^T\intt \lt[ |\nabla\rho_n^\alpha|^2 + \nu_n(\rho_n)|\nabla u_n|^2 \rt]\dx\dt \le C.
\eq
Assume moreover that the density consequences stated in Proposition \ref{prop:app-density-conseq} hold uniformly along the sequence. More precisely, if $d=1$, then $\rho_n$ is bounded away from zero and infinity
uniformly in $n$; if $d=2$ and $\alpha=\frac12$, then, for every finite $p>0$,
\[
\sup_n\esssup_{0\le t\le T} \int_{\T^2}\lt(\rho_n^p+\rho_n^{-p}\rt) \dx<\infty.
\]

Then, up to a subsequence,
\[
\rho_n\to\rho \quad\text{strongly in } L^p((0,T)\times\T^d) \quad\text{for every finite $p\ge1$ and almost everywhere.}
\]
Moreover, there exists a velocity $u$ such that
\[
\rho_nu_n\rightharpoonup\rho u,  \quad\text{in }\calD'((0,T)\times\T^d)
\]
and
\[
u_n\rightharpoonup u \quad\text{weakly in }
\begin{cases}
L^2(0,T;H^1(\T)),&d=1,\\[1mm]
L^2(0,T;W^{1,q}(\T^2)),&d=2,\quad 1<q<2.
\end{cases}
\]
Furthermore,
\[
\rho_n^{\frac{\alpha}{2}}\nabla u_n \rightharpoonup \rho^{\frac{\alpha}{2}}\nabla u \quad\text{weakly in } L^2((0,T)\times\T^d).
\]
Finally,
\[
\mathbb S_n(\rho_n,u_n) \rightharpoonup \mathbb S(\rho,u) \quad\text{in }\calD'((0,T)\times\T^d),
\]
where
\[
\mathbb S_n(\rho_n,u_n) := 2\nu_n(\rho_n)D u_n + \lambda_n(\rho_n)(\nabla\cdot u_n)\mathbb I
\]
and
\[
\mathbb S(\rho,u) := 2(\rho^\alpha+\zeta\rho)D u + 2(\alpha-1)\rho^\alpha(\nabla\cdot u)\mathbb I.
\]
\end{lemma}

\begin{proof}
We first prove the compactness of the density. By the uniform bounds \eqref{eq:BD-compact-hyp-energy}, $\rho_n^\alpha$ is bounded in $L^2(0,T;H^1(\T^d))$. Since $d\le2$, we have the compact embedding $H^1(\T^d)\Subset L^q(\T^d)$ for every finite $q$. Hence $\rho_n^\alpha$ is spatially compact in $L^q(\T^d)$ for every finite $q$. Since the map $z\mapsto z^{1/\alpha}$ is continuous on $[0,\infty)$, this implies compactness of $\rho_n$ in measure. Together with the uniform $L^\gamma$-bound on $\rho_n$, Vitali's theorem gives spatial
compactness of $\rho_n$ in $L^1(\T^d)$.

We next obtain time compactness from the continuity equation. Since $\sqrt{\rho_n}u_n$ is bounded in $L^\infty(0,T;L^2(\T^d))$ and $\sqrt{\rho_n}$ is bounded in $L^\infty(0,T;L^{2\gamma}(\T^d))$, we have
\bq\label{eq:mom-bdd0}
\rho_nu_n=\sqrt{\rho_n}\,\sqrt{\rho_n}u_n \quad\text{bounded in}\quad L^\infty(0,T;L^{\frac{2\gamma}{\gamma+1}}(\T^d)).
\eq
Hence
\[
\pa_t\rho_n=-\nabla\cdot(\rho_nu_n) \quad\text{is bounded in}\quad L^\infty(0,T;W^{-1,\frac{2\gamma}{\gamma+1}}(\T^d)).
\]
Combining the spatial compactness above with the time compactness and using the Aubin--Lions compactness lemma, we obtain, up to a subsequence,
\[
\rho_n\to\rho \quad\text{strongly in }L^1((0,T)\times\T^d) \quad\text{and almost everywhere}.
\]
Moreover, the density consequences assumed in the statement improve the strong convergence to
\[
\rho_n\to\rho \quad\text{strongly in }L^p((0,T)\times\T^d) \quad\text{for every finite }p\ge1.
\]
This follows by dominated convergence in $d=1$ and by Vitali's theorem in the two-dimensional borderline case. In the latter case, the uniform negative moment bounds also imply, by Fatou's lemma, that $\rho>0$ almost everywhere. In one space dimension, the uniform pointwise lower bound passes to the limit, so that $\rho$ is also bounded away from zero.

We now record the convergence of the viscosity coefficients. Since $\nu_n(r)=r^\alpha+\zeta_n r$ and $\zeta_n\to\zeta$, the preceding strong convergence of $\rho_n$ gives, for every finite $p\ge1$
\[
\nu_n(\rho_n)\to\nu(\rho):=\rho^\alpha+\zeta\rho \quad\text{strongly in }L^p((0,T)\times\T^d).
\]

We notice from \eqref{eq:mom-bdd0} that, after extracting a subsequence, there exists $m\in L^\infty(0,T;L^{\frac{2\gamma}{\gamma+1}}(\T^d))$ such that
\[
\rho_nu_n\rightharpoonup m \quad\text{weakly-* in } L^\infty(0,T;L^{\frac{2\gamma}{\gamma+1}}(\T^d)),
\]
and thus also in $\calD'((0,T)\times\T^d)$.

We identify this momentum limit. By the lower semicontinuity of the kinetic energy under the strong convergence $\rho_n\to\rho$, we have
\[
\int_0^T\intt \frac{|m|^2}{\rho}\,\dx\dt \le \liminf_{n\to\infty} \int_0^T\intt \rho_n|u_n|^2\,\dx\dt <\infty,
\]
where $\frac{|m|^2}{\rho}$ is understood in the relaxed sense, namely it is $+\infty$ on the set $\{\rho=0,\ m\neq0\}$. In particular, $m=0$ a.e. on $\{\rho=0\}$. Thus, $m$ is absolutely continuous with respect to the measure $\rho\,\dx\dt$, and by the Radon--Nikodym theorem there exists a velocity $u$ such that
\[
m=\rho u, \quad \sqrt{\rho}\,u\in L^2((0,T)\times\T^d).
\]
Equivalently, we may define $u=\frac{m}\rho$ on $\{\rho>0\}$ and $u=0$ on $\{\rho=0\}$. With this definition,
\[
\rho_nu_n\rightharpoonup\rho u \quad\text{in }\calD'((0,T)\times\T^d).
\]

We next identify the weak limit of the velocity gradients. In one space dimension, the uniform pointwise bounds on $\rho_n$ imply that $\nu_n(\rho_n)$ is bounded from below and above uniformly in $n$. Hence, the kinetic-energy and weighted-gradient estimates give
\[
\{ u_n\}_n \text{ bounded in } L^\infty(0,T;L^2(\T)) \cap  L^2(0,T;H^1(\T)).
\]
Up to a subsequence,
\[
u_n\rightharpoonup \widetilde u \quad\text{weakly in }L^2(0,T;H^1(\T)).
\]
Since $\rho_n\to\rho$ strongly in every finite $L^p$ and $\rho_nu_n\rightharpoonup\rho u$ in distributions, we obtain
\[
\rho\,\widetilde u=\rho u.
\]
The limiting density is bounded away from zero; hence $\widetilde u=u$ a.e. Consequently,
\[
u_n\rightharpoonup u \quad\text{weakly in }L^2(0,T;H^1(\T)).
\]

We now consider the two-dimensional borderline case $d=2$, $\alpha=\frac12$. Let $1<q<2$.  By the kinetic-energy estimate
and the inverse-density bounds,
\[
\|u_n(t)\|_{L^q(\T^2)}^2 \le \lt( \int_{\T^2} \rho_n|u_n|^2 \,\dx \rt) \lt( \int_{\T^2} \rho_n^{-\frac{q}{2-q}} \,\dx \rt)^{\frac{2-q}{q}}.
\]
Moreover, since $\nu_n(\rho_n)\ge\rho_n^{\frac12}$, the weighted gradient bound gives
\[
\|\nabla u_n(t)\|_{L^q(\T^2)}^2 \le  \lt( \int_{\T^2} \nu_n(\rho_n)|\nabla u_n|^2 \,\dx \rt) \lt( \int_{\T^2} \rho_n^{-\frac{q}{2(2-q)}} \,\dx \rt)^{\frac{2-q}{q}}.
\]
It follows that
\[
\{ u_n \}_n \text{ bounded in }L^\infty(0,T;L^q(\T^2))  \cap L^2(0,T;W^{1,q}(\T^2))
\]
for every $1<q<2$. Thus, after extraction, $u_n\rightharpoonup\widetilde u$ weakly in $L^2(0,T;W^{1,q}(\T^2))$. Since $\rho_n\to\rho$ strongly in every finite $L^p((0,T)\times\T^2)$ and $\rho_nu_n\rightharpoonup\rho u$ in distributions, we obtain 
\[
\rho \widetilde u=\rho u. 
\]
As $\rho>0$ a.e., we infer that $\widetilde u=u$ a.e. Hence,
\[
u_n\rightharpoonup u \quad\text{weakly in }L^2(0,T;W^{1,q}(\T^2))
\]
for every $1<q<2$. 

We now identify the viscous stress. In one space dimension,
\[
\nabla u_n \rightharpoonup \nabla u \quad\text{weakly in } L^2((0,T)\times\T),
\]
while
\[
\nu_n(\rho_n) \to \nu(\rho) \quad\text{strongly in }   L^2((0,T)\times\T).
\]
The same strong convergence holds for $\rho_n^\alpha$. Thus,
\[
\nu_n(\rho_n)\nabla u_n \rightharpoonup \nu(\rho)\nabla u \quad \text{and} \quad  \rho_n^\alpha\nabla\cdot u_n
\rightharpoonup \rho^\alpha\nabla\cdot u \quad\text{in }\calD'((0,T)\times\T).
\]

In two space dimensions, we choose $q=\frac32$. Then,
\[
\nabla u_n \rightharpoonup \nabla u \quad\text{weakly in } L^{\frac32}((0,T)\times\T^2).
\]
On the other hand,
\[
\nu_n(\rho_n) \to \nu(\rho) \quad \text{and} \quad \sqrt{\rho_n}
\to
\sqrt{\rho} \quad\text{strongly in } L^3((0,T)\times\T^2).
\]
It follows that
\[
\nu_n(\rho_n)D u_n
\rightharpoonup
\nu(\rho)D u \quad \text{and} \quad \sqrt{\rho_n}\,\nabla\cdot u_n
\rightharpoonup
\sqrt{\rho}\,\nabla\cdot u
\quad\text{in }\calD'((0,T)\times\T^2).
\]
Consequently, in both admissible cases,
\[
\mathbb S_n(\rho_n,u_n) \rightharpoonup \mathbb S(\rho,u) \quad\text{in }\calD'((0,T)\times\T^d),
\]
where
\[
\mathbb S(\rho,u) = 2(\rho^\alpha+\zeta\rho)D u + 2(\alpha-1)\rho^\alpha(\nabla\cdot u)\mathbb I.
\]

Finally, since $\nu_n(\rho_n)\ge\rho_n^\alpha$,  the sequence $Z_n:=\rho_n^{\frac{\alpha}{2}}\nabla u_n$ is bounded in $L^2((0,T)\times\T^d)$. Let $Z$ denote its weak limit.  In one space dimension, the strong convergence of $\rho_n^{\frac\alpha2}$ in $L^2$ and the weak convergence of
$\nabla u_n$ in $L^2$ give
\[
Z=\rho^{\frac{\alpha}{2}}\nabla u \quad\text{in }\calD'((0,T)\times\T).
\]
In two space dimensions, using $q=\frac32$, we have
\[
\nabla u_n \rightharpoonup \nabla u \quad\text{weakly in } L^{\frac32}((0,T)\times\T^2),
\]
while
\[
\rho_n^{\frac14} \to \rho^{\frac14} \quad\text{strongly in } L^3((0,T)\times\T^2).
\]
This yields
\[
Z=\rho^{\frac14}\nabla u \quad\text{in }\calD'((0,T)\times\T^2).
\]
Hence,
\[
\rho_n^{\frac{\alpha}{2}}\nabla u_n \rightharpoonup \rho^{\frac{\alpha}{2}}\nabla u \quad\text{weakly in } L^2((0,T)\times\T^d).
\]
This completes the proof.
\end{proof}

\begin{remark}\label{rem:BD-convective-comp}
Suppose, in addition to the assumptions of Lemma \ref{lem:BD-convec-comp}, that
\[
u_n\to u \quad\text{a.e. in } (0,T)\times\T^d.
\]
Then
\[
\rho_nu_n\otimes u_n \to \rho u\otimes u \quad\text{strongly in } L^1((0,T)\times\T^d).
\]

Indeed, in one space dimension, the uniform upper and lower bounds on $\rho_n$, together with the kinetic-energy and weighted-gradient estimates, give
\[
u_n \quad\text{bounded in } L^\infty(0,T;L^2(\T)) \cap L^2(0,T;H^1(\T)).
\]
Hence
\[
u_n \quad\text{is bounded in }  L^4((0,T)\times\T),
\]
and consequently
\[
\rho_nu_n\otimes u_n \quad\text{is bounded in }  L^2((0,T)\times\T).
\]

In the two-dimensional case, taking $q=\frac32$ in the proof of Lemma \ref{lem:BD-convec-comp}, we obtain
\[
u_n \quad\text{bounded in }  L^\infty(0,T;L^{\frac32}(\T^2)) \cap L^2(0,T;W^{1,\frac32}(\T^2)).
\]
Since
\[
W^{1,\frac32}(\T^2)\hookrightarrow L^6(\T^2),
\]
interpolation yields
\[
u_n \quad\text{bounded in } L^3((0,T)\times\T^2).
\]
Combining this with the uniform $L^4$ bound on $\rho_n$, we obtain
\[
\rho_nu_n\otimes u_n \quad\text{bounded in }  L^{\frac{12}{11}}((0,T)\times\T^2).
\]

Thus, in either admissible case, the convective tensors are uniformly
integrable. Since
\[
\rho_nu_n\otimes u_n
\to
\rho u\otimes u
\quad\text{almost everywhere},
\]
Vitali's theorem gives the asserted strong $L^1$ convergence. In
particular,
\[
\rho_nu_n\otimes u_n
\rightharpoonup
\rho u\otimes u
\quad\text{in }\calD'((0,T)\times\T^d).
\]
\end{remark}

%
%
%
%
%
%
%
%
%
%
%
 
\subsection{Auxiliary augmented approximation}\label{ssec:aux-aug-app}

We now describe the auxiliary approximation used to construct solutions to the regularized system \eqref{eq:reg-sys}. Throughout this subsection, the parameters $\zeta,\tau>0$ are fixed. We also fix once and for all a number $\kappa\in(0,\frac12)$.  It only determines the augmented velocity
\[
w=u+\kappa\nabla\phi_\zeta(\rho).
\]
The only independent auxiliary index is the Galerkin dimension $N\in\N$. We also fix a positive sequence $\epsilon_N\to0$, which multiplies a purely auxiliary
hyperregularization term. The auxiliary limit means $N\to\infty$.

The construction follows the augmented-velocity philosophy of \cite{BVY22}. We introduce the BD drift velocity
\[
\mathfrak v_N:=\nabla\phi_\zeta(\rho_N), \quad \phi_\zeta'(\rho)=\frac{2\nu_\zeta'(\rho)}{\rho},
\]
and define the physical velocity by $u_N:=w_N-\kappa\mathfrak v_N$. Thus, the identity $\rho_N\mathfrak v_N=2\nabla\nu_\zeta(\rho_N)$ is imposed exactly at the approximate level. In particular, no projection commutator is generated in the BD drift.

Let $X_N\subset C^\infty(\T^d;\R^d)$ be a finite-dimensional Galerkin space and let $P_N$ be the corresponding projection. We choose $X_N$ and $P_N$ so that $P_N$ is stable in $H^m(\T^d)$ for every fixed $m\in\N$, namely
\[
\|P_N\psi\|_{H^m(\T^d)} \le C_m\|\psi\|_{H^m(\T^d)}
\]
with a constant $C_m$ independent of $N$. For instance, one may take $X_N$ to be the space of Fourier modes with frequencies $|\xi|\le N$ and $P_N$ the
Fourier projection. We use a mixed Galerkin--parabolic approximation: only the augmented velocity is projected onto $X_N$, and we seek
\[
w_N(t,\cdot)\in X_N.
\]
For each smooth $w_N$, the density equation below is strictly parabolic and determines a smooth positive density $\rho_N$. We then define the BD drift velocity directly from this approximate density by $\mathfrak v_N=\nabla\phi_\zeta(\rho_N)$. Thus $\mathfrak v_N$ is not an independent Galerkin unknown, but a smooth nonlinear function of $\rho_N$. This is sufficient for the present regularized problem, since the Galerkin formulation of the $w_N$-equation is imposed only against test functions in $X_N$. This should be viewed as a simplified implementation of the augmented-velocity method of \cite{BVY22}, adapted to the present regularized system. This choice also keeps the BD structural identity $\rho_N\mathfrak v_N=2\nabla\nu_\zeta(\rho_N)$ exact at the approximate level. Consequently, the $\kappa$-entropy calculation below can be carried out without introducing additional projection remainders associated with the drift variable.

For a given smooth $w_N$, the density $\rho_N$ is determined by
\bq\label{eq:aux-rho}
\pa_t\rho_N+\nabla\cdot(\rho_N w_N)=2\kappa\Delta\nu_\zeta(\rho_N).
\eq
Equivalently, since $\rho_N\mathfrak v_N=2\nabla\nu_\zeta(\rho_N)$ and $u_N=w_N-\kappa\mathfrak v_N$, this equation can be written in conservative form as
\[
\pa_t\rho_N+\nabla\cdot(\rho_Nu_N)=0.
\]
Since $\nu_\zeta'(\rho)=\alpha\rho^{\alpha-1}+\zeta\ge\zeta>0$, the density equation is strictly parabolic. Hence, smooth strictly positive initial densities remain strictly positive, and the total mass is conserved.

The kinetic equation is solved with the physical velocity $u_N$:
\bq\label{eq:aux-kinetic}
\pa_t f_N+v\cdot\nabla_x f_N=(\rho_N)_\e\nabla_v\cdot\lt[(v-u_{N,\e}^\e)f_N\rt],
\quad
u_{N,\e}^\e:=\frac{(\rho_Nu_N)_\e}{(\rho_N)_\e}.
\eq
Since $(\rho_N)_\e\ge c_\e:=\min_{\T^d}\theta_\e>0$, this is a linear kinetic equation with smooth coefficients at the auxiliary level. We set
\[
F_N:=\intr(v-u_{N,\e}^\e)f_N\,\dv .
\]
The Galerkin equation for $w_N$ is written in augmented variables. For every $\psi\in X_N$, we impose
\begin{align}\label{eq:aux-w}
\begin{aligned}
&\frac{\dd}{\dt}\intt\rho_Nw_N\cdot\psi\,\dx -\intt\rho_Nu_N\otimes w_N:\nabla\psi\,\dx -\intt p_\tau(\rho_N)\nabla\cdot\psi\,\dx \\
&\quad +2(1-\kappa)\intt\nu_\zeta(\rho_N)D w_N:\nabla\psi\,\dx +2\kappa\intt\nu_\zeta(\rho_N)A w_N:\nabla\psi\,\dx \\
&\quad +(1-\kappa)\intt\lambda_\zeta(\rho_N)\nabla\cdot w_N\,\nabla\cdot\psi\,\dx -2\kappa(1-\kappa)\intt\nu_\zeta(\rho_N)\nabla\mathfrak v_N:\nabla\psi\,\dx \\
&\quad  -\kappa(1-\kappa)\intt\lambda_\zeta(\rho_N)\nabla\cdot\mathfrak v_N\,\nabla\cdot\psi\,\dx +\epsilon_N\intt\lt[ \Delta^s w_N\cdot\Delta^s\psi + (1+|\nabla w_N|^2)\nabla w_N:\nabla\psi \rt]\dx \\
&\quad \quad =\intt\rho_N(F_N)_\e\cdot\psi\,\dx +\tau\intt\mathfrak K_\zeta(\rho_N)\cdot\psi\,\dx -\tau\intt u_N\cdot\psi\,\dx .
\end{aligned}
\end{align}
Here $s\in\N$ is fixed so large that $H^{2s}(\T^d)\hookrightarrow C^2(\T^d)$. The terms multiplied by $\epsilon_N$ are purely auxiliary and vanish in the Galerkin limit.

The following estimate is the main reason for introducing the augmented formulation. Its derivation, including the treatment of the kinetic coupling remainder and its absorption, is given in Appendix \ref{app:aug}.

\begin{lemma}\label{lem:aux-kappa-entropy}
Let $(f_N,\rho_N,w_N)$ be a smooth solution of \eqref{eq:aux-rho}--\eqref{eq:aux-w}, and let
\[
u_N=w_N-\kappa\mathfrak v_N, \quad \mathfrak v_N=\nabla\phi_\zeta(\rho_N).
\]
Then, for every $T>0$, one has
\begin{align}\label{eq:aux-kappa-energy}
\begin{aligned}
&\sup_{0\le t\le T}\Bigg\{ \frac12\inttr |v|^2f_N\,\dx\dv +\intt\lt[ \frac12\rho_N|w_N|^2 +\frac{\kappa(1-\kappa)}2\rho_N|\mathfrak v_N|^2 +\pi(\rho_N) +\tau \pi_{10}(\rho_N) \rt]\dx \\
&\hspace{25mm} +\frac{\tau}{2}\intt|\nabla Y_\zeta(\rho_N)|^2\,\dx +2\tau\kappa\intt J_\zeta(\rho_N)\,\dx \Bigg\} \\
&\quad +\int_0^T\Bigg[ 2(1-\kappa)\intt\nu_\zeta(\rho_N)|D w_N-\kappa\nabla\mathfrak v_N|^2\,\dx +(1-\kappa)\intt\lambda_\zeta(\rho_N)|\nabla\cdot w_N-\kappa\nabla\cdot\mathfrak v_N|^2\,\dx \\
&\hspace{25mm} +2\kappa\intt\nu_\zeta(\rho_N)|A w_N|^2\,\dx +\kappa\intt\nabla p_\tau(\rho_N)\cdot\mathfrak v_N\,\dx \\
&\hspace{25mm} +\frac12\calD_{{\rm drag},N}(t) +\calD_{{\rm damp},N}(t) +\calD_{{\rm cap},N}^{\kappa}(t) +\epsilon_N\calG_N(t) \Bigg]\dt \\
&\quad \quad \le C_{\e,\zeta,\tau,T}.
\end{aligned}
\end{align}
Here
\[
\calD_{{\rm drag},N}:=\inttr(\rho_N)_\e |v-u_{N,\e}^\e|^2f_N\,\dx\dv, \quad \calD_{{\rm damp},N}:=\tau\intt |u_N|^2\,\dx,
\]
\[
\calD_{{\rm cap},N}^{\kappa} := 4\tau\kappa\intt\lt[ 2\nu_\zeta(\rho_N)|\nabla^2s_\zeta(\rho_N)|^2 +\lambda_\zeta(\rho_N)|\Delta s_\zeta(\rho_N)|^2 \rt]\dx,
\]
and
\[
\calG_N(t):=\intt\lt[ |\Delta^s w_N|^2 + (1+|\nabla w_N|^2)|\nabla w_N|^2 \rt]\dx .
\]
\end{lemma}

We next state the existence of the auxiliary Galerkin solutions.

\begin{proposition} 
For every fixed $\e,\zeta,\tau>0$, every fixed $\kappa\in(0,\frac12)$, and every $N\in\N$, the auxiliary system \eqref{eq:aux-rho}--\eqref{eq:aux-w} admits a smooth solution on $[0,T]$. Moreover, $f_N\ge0$, $\rho_N>0$, the masses are conserved, and the auxiliary estimate \eqref{eq:aux-kappa-energy} holds.
\end{proposition}

\begin{proof}
The proof follows a standard Galerkin--parabolic fixed-point argument. For a given smooth $w_N\in C([0,T];X_N)$, the equation \eqref{eq:aux-rho} is strictly parabolic because $\nu_\zeta'(\rho)\ge\zeta>0$. Hence it admits a smooth positive solution on a short time interval, and the total mass is conserved. Once $\rho_N$ is known, the drift $\mathfrak v_N=\nabla\phi_\zeta(\rho_N)$ and the physical velocity $u_N=w_N-\kappa\mathfrak v_N$ are smooth. The kinetic equation \eqref{eq:aux-kinetic} is then a linear transport equation with smooth coefficients and with $(\rho_N)_\e\ge c_\e>0$. It preserves nonnegativity and the total mass of $f_N$.

For fixed $(\rho_N,f_N)$, the equation \eqref{eq:aux-w} is a finite-dimensional ODE for the coefficients of $w_N\in X_N$. The hyperregularization term is continuous and coercive in the finite-dimensional topology. Hence, a Schauder fixed-point argument gives a local-in-time solution of the coupled auxiliary system.

The continuation criterion is provided by Lemma \ref{lem:aux-kappa-entropy}. The estimate controls the Galerkin kinetic energy, the augmented fluid energy, the relevant dissipations, and the auxiliary hyperregularization through $\epsilon_N\calG_N$. Since the density remains strictly positive and smooth on finite time intervals and the kinetic equation has smooth coefficients, no finite-time breakdown of the Galerkin solution can occur. Therefore the local solution extends to the whole interval $[0,T]$.
\end{proof}

We now pass to the Galerkin limit $N \to \infty$.

\begin{proposition} 
Fix $\e,\zeta,\tau>0$. There exists a triple $(f_{\mathbf a},\rho_{\mathbf a},u_{\mathbf a})$, $\mathbf a=(\zeta,\tau)$, satisfying the regularized system \eqref{eq:reg-sys} in the sense of distributions. Moreover, after extracting a subsequence as $N\to\infty$, we have
\begin{align*}
\rho_N &\to \rho_{\mathbf a} && \text{strongly in }L^1((0,T)\times\T^d)\text{ and almost everywhere},\\
\rho_Nu_N &\rightharpoonup \rho_{\mathbf a}u_{\mathbf a} && \text{in }\calD'((0,T)\times\T^d),\\
\rho_Nu_N\otimes u_N &\rightharpoonup \rho_{\mathbf a}u_{\mathbf a}\otimes u_{\mathbf a} && \text{in }\calD'((0,T)\times\T^d),\\
f_N &\stackrel{*}{\rightharpoonup} f_{\mathbf a} && \text{weakly-* in }L^\infty(0,T; \calM_+(\T^d\times\R^d)),
\end{align*}
where $\calM_+$ denotes the space of nonnegative Radon measures.  The limiting triple satisfies the free energy inequality, the modified BD entropy inequality, and the combined estimate of Proposition \ref{prop:comb-est}.
\end{proposition}

\begin{proof}
We first derive compactness estimates which are uniform with respect to $N$. By Lemma \ref{lem:aux-kappa-entropy}, together with the free energy estimate, we have
\[
\sup_{0\le t\le T}\lt[\inttr |v|^2f_N\,\dx\dv+\intt \rho_N|u_N|^2\,\dx+\intt p_\tau(\rho_N)\,\dx+\tau\intt|\nabla Y_\zeta(\rho_N)|^2\,\dx\rt]\le C_{\e,\zeta,\tau,T}.
\]
Moreover, the continuity equation
\[
\pa_t\rho_N+\nabla\cdot(\rho_Nu_N)=0
\]
gives a uniform bound on $\pa_t\rho_N$ in a negative Sobolev space. Since $\tau>0$ and $\zeta>0$ are fixed, the capillarity bound gives spatial compactness of the density. Thus, by the Aubin--Lions compactness lemma, after extracting a subsequence,
\[
\rho_N\to\rho_{\mathbf a}\quad\text{strongly in }L^1((0,T)\times\T^d)\text{ and almost everywhere}.
\]
The conservation of the kinetic mass and the kinetic energy bound yield
\[
\sup_N\esssup_{0\le t\le T} \inttr (1+|v|^2)f_N\,\dx\dv \le C_{\e,\zeta,\tau,T}.
\]
In particular, $\{f_N\}_N$ is bounded in $L^\infty(0,T;\calM_+(\T^d\times\R^d))$. Hence, by weak-* compactness, there exist a subsequence, not relabeled, and
\[
f_{\mathbf a}\in L^\infty(0,T;\calM_+(\T^d\times\R^d))
\]
such that
\[
f_N\stackrel{*}{\rightharpoonup}f_{\mathbf a} \quad\text{weakly-* in } L^\infty(0,T;\calM_+(\T^d\times\R^d)).
\]
Moreover, the uniform second velocity moment gives tightness in the velocity variable: for every $R>0$,
\[
\sup_N\int_0^T\iint_{\T^d\times\{|v|>R\}} f_N\,\dx\dv\dt \le \frac{C_{\e,\zeta,\tau,T}}{R^2}.
\]
Combining this tightness with the weak-* convergence, we preserve the kinetic
mass in the limit:
\[
\inttr f_{\mathbf a}\,\dx\dv=1 \quad\text{for a.e. }t\in(0,T).
\]
Furthermore, by lower semicontinuity with compact velocity cutoffs,
\[
\esssup_{0\le t\le T} \inttr (1+|v|^2)f_{\mathbf a}\,\dx\dv \le C_{\e,\zeta,\tau,T}.
\]
Indeed, this follows by testing the weak-* convergence with $(1+|v|^2)\chi_R(v)$, where $\chi_R\in C_c^\infty(\R^d)$, $0\le\chi_R\le1$, $\chi_R\uparrow1$, and then letting $R\to\infty$.
Consequently,
\[
f_{\mathbf a}\in L^\infty(0,T;\calP_2(\T^d\times\R^d)).
\]
 
We next observe that the hyperregularization vanishes in the limit. Let
$\psi\in C_c^\infty([0,T)\times\T^d;\R^d)$ and use $P_N\psi$ as a Galerkin test
function. By the $H^m$-stability of $P_N$ and the embedding
$H^{2s}(\T^d)\hookrightarrow W^{1,\infty}(\T^d)$, we have
\[
\|\Delta^s P_N\psi\|_{L^2} + \|\nabla P_N\psi\|_{L^\infty} \le C\|\psi\|_{H^{2s}}.
\]
Since
\[
\epsilon_N\int_0^T\calG_N(t)\,\dt\le C_{\e,\zeta,\tau,T},
\]
we estimate
\begin{align*}
&\lt|\epsilon_N\int_0^T\intt \Delta^s w_N\cdot\Delta^s P_N\psi\,\dx\dt\rt|  +\lt|\epsilon_N\int_0^T\intt
(1+|\nabla w_N|^2)\nabla w_N:\nabla P_N\psi\,\dx\dt\rt| \\
&\quad\le C\epsilon_N^{\frac12} \lt(\epsilon_N\int_0^T\intt|\Delta^s w_N|^2\,\dx\dt\rt)^{\frac12} \|\psi\|_{L^2(0,T;H^{2s})}\cr
&\quad \quad + C\epsilon_N^{\frac12} \lt(\epsilon_N\int_0^T\intt|\nabla w_N|^2\,\dx\dt\rt)^{\frac12} \|\psi\|_{L^2(0,T;H^{2s})} \\
&\quad \quad + C\epsilon_N^{\frac14} \lt(\epsilon_N\int_0^T\intt|\nabla w_N|^4\,\dx\dt\rt)^{3/4} \|\psi\|_{L^4(0,T;H^{2s})} \cr
&\quad \to 0.
\end{align*}
Thus, all hyperregularization terms disappear as $N\to\infty$.

We provide the physical BD compactness needed to identify the convective term. Since $\kappa\in(0,\frac12)$ is fixed, the auxiliary entropy directly controls the physical effective velocity. Indeed, $u_N+\nabla\phi_\zeta(\rho_N)=w_N+(1-\kappa)\mathfrak v_N$, and thus
\[
\intt\rho_N\lt|u_N+\nabla\phi_\zeta(\rho_N)\rt|^2\,\dx \le C_\kappa\lt[\intt\rho_N|w_N|^2\,\dx+\kappa(1-\kappa)\intt\rho_N|\mathfrak v_N|^2\,\dx\rt].
\]
Moreover,
\[
D u_N=D w_N-\kappa\nabla\mathfrak v_N,\quad A u_N=A w_N,\quad \nabla\cdot u_N=\nabla\cdot w_N-\kappa\nabla\cdot\mathfrak v_N.
\]
Hence, Lemma \ref{lem:aux-kappa-entropy} gives
\begin{align}\label{eq:aux-physical-BD-bound}
\begin{aligned}
&\sup_{0\le t\le T}\intt\rho_N\lt|u_N+\nabla\phi_\zeta(\rho_N)\rt|^2\,\dx +\int_0^T\intt\nu_\zeta(\rho_N)|\nabla u_N|^2\,\dx\dt \\
&\quad +\int_0^T\intt\nabla p_\tau(\rho_N)\cdot\nabla\phi_\zeta(\rho_N)\,\dx\dt \le C_{\e,\zeta,\tau,T}.
\end{aligned}
\end{align}

We now apply Lemma \ref{lem:BD-convec-comp} with $\zeta_n=\zeta$ fixed. The uniform bounds required in the lemma follow from the energy estimate and \eqref{eq:aux-physical-BD-bound}. The density consequences required in the lemma follow from the singular BD control in \eqref{eq:aux-physical-BD-bound}, by the same argument as in Proposition \ref{prop:app-density-conseq}. Hence Lemma \ref{lem:BD-convec-comp} yields
\[
\rho_N\to\rho_{\mathbf a} \quad\text{strongly in } L^p((0,T)\times\T^d) \quad\text{for every finite }p\ge1,
\]
and
\[
\rho_Nu_N \rightharpoonup \rho_{\mathbf a}u_{\mathbf a} \quad\text{in }\calD'((0,T)\times\T^d).
\]
The same lemma also identifies the viscous stress:
\[
\mathbb S_\zeta(\rho_N,u_N)\rightharpoonup\mathbb S_\zeta(\rho_{\mathbf a},u_{\mathbf a}) \quad\text{in }\calD'((0,T)\times\T^d).
\]

It remains to identify the convective term. Set $m_N:=\rho_Nu_N$. Since $\zeta>0$ is fixed, we have
\[
\nu_\zeta(\rho_N)\ge\zeta\rho_N, \quad \nu_\zeta'(\rho_N)\ge\zeta.
\]
Using 
\[ 
\nabla m_N = \rho_N\nabla u_N + 2\sqrt{\rho_N}u_N\otimes\nabla\sqrt{\rho_N}, 
\] 
we estimate the two terms separately. First, by the conservation of mass and the weighted-gradient estimate, 
\begin{align*} 
\|\rho_N\nabla u_N\|_{L^2(0,T;L^1(\T^d))} &\le \|\sqrt{\rho_N}\|_{L^\infty(0,T;L^2(\T^d))} \|\sqrt{\rho_N}\nabla u_N\|_{L^2((0,T)\times\T^d)} \\ 
&\le C_\zeta \|\sqrt{\nu_\zeta(\rho_N)}\nabla u_N\|_{L^2((0,T)\times\T^d)} \\ 
&\le C_{\e,\zeta,\tau,T}. 
\end{align*} 
Here we used $\nu_\zeta(\rho_N)\ge\zeta\rho_N$. Next, since $\sqrt{\rho_N}\nabla\phi_\zeta(\rho_N) = 4\nu_\zeta'(\rho_N)\nabla\sqrt{\rho_N}$, the lower bound $\nu_\zeta'(\rho_N)\ge\zeta$, together with the kinetic-energy and physical effective-velocity estimates, gives 
\begin{align*} 
4\zeta \|\nabla\sqrt{\rho_N}\|_{L^\infty(0,T;L^2(\T^d))} &\le \|\sqrt{\rho_N}\nabla\phi_\zeta(\rho_N)\|_{L^\infty(0,T;L^2(\T^d))} \\ 
&\le \|\sqrt{\rho_N} \lt(u_N+\nabla\phi_\zeta(\rho_N)\rt)\|_{L^\infty(0,T;L^2(\T^d))}  + \|\sqrt{\rho_N}u_N\|_{L^\infty(0,T;L^2(\T^d))} \\ &\le C_{\e,\zeta,\tau,T}. 
\end{align*}
Hence,
\[ 
\sqrt{\rho_N}u_N\otimes\nabla\sqrt{\rho_N} \quad\text{is bounded in } L^\infty(0,T;L^1(\T^d)), 
\] 
and thus also in $L^2(0,T;L^1(\T^d))$. Consequently, 
\[ 
\nabla m_N \quad\text{is bounded in }  L^2(0,T;L^1(\T^d)). 
\]
On the other hand, rewriting the smooth auxiliary system in the physical variables and testing the physical momentum equation against $P_N\psi$, we obtain, for some sufficiently large integer $\ell$,
\[
\pa_t P_Nm_N \quad\text{bounded in }  L^1(0,T;W^{-\ell,1}(\T^d)).
\]
Indeed, the convective, pressure, viscous, capillarity, damping, and averaged-force terms are uniformly bounded in a sufficiently weak negative Sobolev space. The hyperregularization terms are admissible in the same estimate and converge to zero, as shown above.

We now remove the projection from the time-compactness estimate. We take $P_N$ to be the Fourier projection onto the frequencies $|\xi|\le N$. Since $m_N$ is bounded in $L^\infty(0,T;L^1(\T^d))$, for every sufficiently large $s>\frac d2$,
\[
\|(I-P_N)m_N(t)\|_{H^{-s}}^2 = \sum_{|\xi|>N} (1+|\xi|^2)^{-s} |\widehat{m_N}(t,\xi)|^2 \le \|m_N(t)\|_{L^1}^2 \sum_{|\xi|>N} (1+|\xi|^2)^{-s}.
\]
Consequently,
\[
\|(I-P_N)m_N\|_{L^\infty(0,T;H^{-s})} \to0.
\]

Choosing $s>\ell+\frac d2$, the Sobolev embedding $H^s(\T^d)\hookrightarrow W^{\ell,\infty}(\T^d)$ yields, by duality, $W^{-\ell,1}(\T^d)\hookrightarrow H^{-s}(\T^d)$. Hence, for every $h>0$,
\[
\int_0^{T-h} \|m_N(t+h)-m_N(t)\|_{H^{-s}} \,\dt \le h \|\pa_tP_Nm_N\|_{L^1(0,T;H^{-s})} + 2T \|(I-P_N)m_N\|_{L^\infty(0,T;H^{-s})}.
\]
It follows that
\[
\lim_{h\to0} \limsup_{N\to\infty} \int_0^{T-h} \|m_N(t+h)-m_N(t)\|_{H^{-s}} \,\dt =0.
\]

On the other hand, the preceding spatial estimate gives
\[
m_N \quad\text{bounded in } L^2(0,T;W^{1,1}(\T^d)).
\]
Since $W^{1,1}(\T^d) \Subset L^1(\T^d) \hookrightarrow H^{-s}(\T^d)$, the Aubin--Lions compactness criterion yields, up to a subsequence,
\[
m_N\to m \quad\text{strongly in } L^1((0,T)\times\T^d).
\]
The distributional convergence of the momentum identifies $m=\rho_{\mathbf a}u_{\mathbf a}$. Since $\rho_N\to\rho_{\mathbf a}$ a.e.  and $\rho_{\mathbf a}>0$ a.e., we conclude that
\[
u_N\to u_{\mathbf a} \quad\text{almost everywhere in } (0,T)\times\T^d.
\]
Thus, by Remark \ref{rem:BD-convective-comp},
\[
\rho_Nu_N\otimes u_N \to \rho_{\mathbf a}u_{\mathbf a}\otimes u_{\mathbf a} \quad\text{strongly in } L^1((0,T)\times\T^d).
\]
In particular,
\[
\rho_Nu_N\otimes u_N \rightharpoonup \rho_{\mathbf a}u_{\mathbf a}\otimes u_{\mathbf a} \quad\text{in }\calD'((0,T)\times\T^d).
\]

It remains to identify the kinetic--fluid coupling. The strong convergence of $\rho_N$ implies, for every finite $q\ge1$ and every $k\ge0$,
\[
(\rho_N)_\e\to(\rho_{\mathbf a})_\e\quad\text{strongly in }L^q(0,T;C^k(\T^d)).
\]
The weak formulation of the momentum equation gives compactness in time for the mollified momenta, and hence
\[
(\rho_Nu_N)_\e\to(\rho_{\mathbf a}u_{\mathbf a})_\e\quad\text{strongly in }L^q(0,T;C^k(\T^d)).
\]
Since $(\rho_N)_\e\ge c_\e>0$, it follows that
\[
u_{N,\e}^\e=\frac{(\rho_Nu_N)_\e}{(\rho_N)_\e}\to\frac{(\rho_{\mathbf a}u_{\mathbf a})_\e}{(\rho_{\mathbf a})_\e}=:u_{\mathbf a,\e}^\e
\quad\text{strongly in }L^q(0,T;C^k(\T^d)).
\]
 Using this convergence, the weak-* convergence of $f_N$ in $L^\infty(0,T;\calM_+(\T^d\times\R^d))$, and the velocity tightness from the kinetic energy, we identify the kinetic--fluid force.  Let $\psi\in C_c^\infty((0,T)\times\T^d;\R^d)$. Then
\[
\int_0^T\intt \psi\cdot F_N\,\dx\dt = \int_0^T\inttr \psi(t,x)\cdot\lt(v-u_{N,\e}^\e(t,x)\rt)f_N\,\dx\dv\dt .
\]
Let $\chi_R\in C_c^\infty(\R^d)$ be such that ${\bf 1}_{B(0,R)}\le\chi_R\le{\bf 1}_{B(0,2R)}$. We decompose
\begin{align*}
\int_0^T\inttr \psi\cdot\lt(v-u_{N,\e}^\e\rt)f_N\,\dx\dv\dt &= \int_0^T\inttr \chi_R(v)\psi\cdot\lt(v-u_{N,\e}^\e\rt)f_N\,\dx\dv\dt \\
&\quad + \int_0^T\inttr (1-\chi_R(v))\psi\cdot\lt(v-u_{N,\e}^\e\rt)f_N\,\dx\dv\dt .
\end{align*}
For fixed $R>0$,  
\[
\chi_R(v)\psi(t,x)\cdot\lt(v-u_{N,\e}^\e(t,x)\rt) \to \chi_R(v)\psi(t,x)\cdot\lt(v-u_{\mathbf a,\e}^\e(t,x)\rt) \quad \text{strongly in } L^1(0,T;C_0(\T^d\times\R^d))
\]
since
\[
u_{N,\e}^\e\to u_{\mathbf a,\e}^\e \quad\text{strongly in }L^q(0,T;C^k(\T^d))
\]
for every finite $q\ge1$ and every $k\ge0$. Thus, by the weak-* convergence
of $f_N$ as nonnegative measures,
\[
\int_0^T\inttr
\chi_R\psi\cdot\lt(v-u_{N,\e}^\e\rt)f_N\,\dx\dv\dt
\to
\int_0^T\inttr
\chi_R\psi\cdot\lt(v-u_{\mathbf a,\e}^\e\rt)f_{\mathbf a}\,\dx\dv\dt .
\]

It remains to control the velocity tail. By the uniform second moment bound,
\[
\int_0^T\iint_{\T^d\times\{|v|>R\}} |v|f_N\,\dx\dv\dt + \lt\| \iint_{\T^d\times\{|v|>R\}} f_N\,\dx\dv \rt\|_{L^2(0,T)} \le \frac{C_{\e,\zeta,\tau,T}}{R}.
\]
Using also the uniform bound of $u_{N,\e}^\e$ in
$L^2(0,T;L^\infty(\T^d))$, we obtain
\begin{align*}
&\lt| \int_0^T\inttr (1-\chi_R)\psi\cdot\lt(v-u_{N,\e}^\e\rt)f_N\,\dx\dv\dt\rt| \\
&\quad\le \|\psi\|_{L^\infty} \int_0^T\iint_{\T^d\times\{|v|>R\}} |v|f_N\,\dx\dv\dt   +\|\psi\|_{L^\infty} \|u_{N,\e}^\e\|_{L^2(0,T;L^\infty)} \lt\| \iint_{\T^d\times\{|v|>R\}} f_N\,\dx\dv \rt\|_{L^2(0,T)} \\
&\quad\le \frac{C_{\e,\zeta,\tau,T,\psi}}{R}.
\end{align*}
Consequently,
\[
\limsup_{N\to\infty} \lt| \int_0^T\inttr (1-\chi_R)\psi\cdot\lt(v-u_{N,\e}^\e\rt)f_N\,\dx\dv\dt \rt| \le \frac{C_{\e,\zeta,\tau,T,\psi}}{R}.
\]
The same tail estimate holds for the limiting integrand, using the second-moment bound for $f_{\mathbf a}$ and the bound $u_{\mathbf a,\e}^\e\in L^2(0,T;L^\infty(\T^d))$. Letting first $N\to\infty$ and then $R\to\infty$, we obtain
\[
F_N\rightharpoonup F_{\mathbf a}:=\intr (v-u_{\mathbf a,\e}^\e)f_{\mathbf a}\,\dv \quad\text{in }\calD'((0,T)\times\T^d).
\]
Here the integral defining $F_{\mathbf a}$ is understood in the measure sense, according to the density-like notation convention introduced above.
 
Moreover, the drag estimate gives
\[
\|(F_N)_\e\|_{L^2(0,T;W^{k,\infty}(\T^d))}\le C_{\e,k},
\]
and thus
\[
(F_N)_\e\rightharpoonup(F_{\mathbf a})_\e\quad\text{weakly in }L^2(0,T;W^{k,\infty}(\T^d)).
\]
Combining this weak convergence with the strong convergence of $\rho_N$, we get
\[
\rho_N(F_N)_\e\rightharpoonup\rho_{\mathbf a}(F_{\mathbf a})_\e
\quad\text{in }\calD'((0,T)\times\T^d).
\]

We can now pass to the limit in the weak formulation of the auxiliary system. The continuity equation passes to the limiting continuity equation. The kinetic equation passes to the regularized kinetic equation with the limiting averaged velocity $u_{\mathbf a,\e}^\e$. In the momentum equation, the convective term, the stress term, the coupling term, and the vanishing hyperregularization term have been identified above. Since the augmented formulation is equivalent to the physical formulation through $w_N=u_N+\kappa\nabla\phi_\zeta(\rho_N)$, the limit satisfies the regularized momentum equation in \eqref{eq:reg-sys}. Thus $(f_{\mathbf a},\rho_{\mathbf a},u_{\mathbf a})$ satisfies the regularized system \eqref{eq:reg-sys} in the sense of distributions.

Finally, the free energy estimate follows by the usual lower semicontinuity argument. The auxiliary $\kappa$-entropy provides the compactness needed to identify the limiting regularized system. Once the limiting triple has been identified as a solution of \eqref{eq:reg-sys}, the physical BD estimate is obtained from the BD calculation of Proposition \ref{prop:reg-BD} by the standard approximation/lower-semicontinuity argument. Consequently, the regularized solution satisfies the free energy inequality, the modified BD entropy inequality, and the combined estimate of Proposition \ref{prop:comb-est}.
\end{proof}

\begin{remark} 
For each fixed $\kappa\in(0,\frac12)$, the augmented formulation is equivalent, at the smooth level, to the physical regularized system written in the variables $(\rho,u)$. Thus, after the Galerkin limit is taken, the limiting triple satisfies the exact $(\zeta,\tau)$-regularized system. The parameters $\zeta$ and $\tau$ remain present at this stage and are removed only later.
\end{remark}

%
%
%
%
%
%
%
%
%
%
%

\subsection{Passage to the limit  $\tau\to0$ at fixed $\zeta$}\label{ssec:tau-limit}

We now pass to the limit in the approximation terms governed by $\tau$, namely the artificial pressure, capillarity, and linear damping terms, while keeping $\zeta>0$ fixed. Let $\tau_n\to0$, and denote the corresponding regularized solutions by $(f_n,\rho_n,u_n)$. For simplicity, we write
\[
u_{n,\e}^\e:=\frac{(\rho_nu_n)_\e}{(\rho_n)_\e}, \quad F_n:=\intr (v-u_{n,\e}^\e)f_n\,\dv.
\]
The regularized system reads
\begin{align}\label{eq:tau-sys}
\begin{aligned}
&\pa_t f_n+v\cdot\nabla_x f_n = (\rho_n)_\e\nabla_v\cdot\lt[(v-u_{n,\e}^\e)f_n\rt], \\
&\pa_t\rho_n+\nabla\cdot(\rho_nu_n)=0, \\
&\pa_t(\rho_nu_n)+\nabla\cdot(\rho_nu_n\otimes u_n) +\nabla(\rho_n^\gamma+\tau_n\rho_n^{10})   = \nabla\cdot\mathbb S_\zeta(\rho_n,u_n) +\rho_n(F_n)_\e +\tau_n\mathfrak K_\zeta(\rho_n) -\tau_nu_n.
\end{aligned}
\end{align}

By Proposition \ref{prop:comb-est}, for every $T>0$,
\begin{align}\label{eq:tau-unif}
\begin{aligned}
& \esssup_{0\le t \le T} \calE_{\eta,n}(t) +\int_0^T\lt[ \calD_{{\rm vis},n}(t) +\calD_{{\rm drag},n}(t) +\calD_{{\rm BD},n}(t) \rt]\dt \\
&\quad +\int_0^T\lt[ \calD_{{\rm rot},n}(t) +\calD_{{\rm cap},n}(t) +\calD_{{\rm damp},n}(t) \rt]\dt \le C_\e,
\end{aligned}
\end{align}
where the constant is independent of $n$, $\tau_n$, and $\zeta$. In estimates where the fixed parameter $\zeta>0$ is used quantitatively, the constants may depend on $\zeta$.

We first show that all terms multiplied by $\tau_n$ vanish in the limit.

\medskip
\noindent
\textit{Linear damping.}
Since
\[
\calD_{{\rm damp},n}(t)=\tau_n\intt |u_n|^2\,\dx,
\]
we have
\[
\|\tau_nu_n\|_{L^2((0,T)\times\T^d)}^2 = \tau_n\lt[\tau_n\int_0^T\intt |u_n|^2\,\dx\dt\rt] \le C_\e\tau_n.
\]
Hence
\bq\label{eq:tau-u}
\tau_nu_n\to0 \quad \text{strongly in }L^2((0,T)\times\T^d).
\eq

\medskip

\noindent
\textit{Artificial pressure.}
The BD pressure dissipation satisfies
\[
\calD_{{\rm BD},n}(t) = 2\intt \lt(\gamma\rho_n^{\gamma-1}+10\tau_n\rho_n^9\rt) \lt(\alpha\rho_n^{\alpha-2}+\frac{\zeta}{\rho_n}\rt) |\nabla\rho_n|^2\,\dx.
\]
In particular, since $\zeta>0$ is fixed,
\[
\tau_n\int_0^T\intt |\nabla\rho_n^5|^2\,\dx\dt \le C_{\e,\zeta}.
\]
The approximate energy bound also gives
\[
\tau_n \esssup_{0\le t \le T}  \intt\rho_n^{10}\,\dx\le C_\e.
\]
Thus
\[
\sqrt{\tau_n}\,\rho_n^5 \quad\text{is bounded in }L^2(0,T;H^1(\T^d)).
\]
Since $d\le2$, the Sobolev embedding $H^1(\T^d)\hookrightarrow L^6(\T^d)$ yields
\bq\label{eq:tau-rho30}
\tau_n\int_0^T\|\rho_n(t)\|_{L^{30}(\T^d)}^{10}\,\dt \le C_{\e,\zeta,T}.
\eq

Let $\vartheta\in(0,1)$ be defined by
\[
\frac1{10}=\frac{\vartheta}{\gamma}+\frac{1-\vartheta}{30}, \quad \text{equivalently}, \quad  \vartheta=\frac{2\gamma}{30-\gamma}.
\]
Interpolating between $L^\gamma(\T^d)$ and $L^{30}(\T^d)$, using the uniform $L^\infty(0,T;L^\gamma)$-bound and \eqref{eq:tau-rho30}, we obtain
\[
\tau_n\int_0^T\|\rho_n(t)\|_{L^{10}(\T^d)}^{10}\,\dt \le C\tau_n\int_0^T\|\rho_n(t)\|_{L^{30}(\T^d)}^{10(1-\vartheta)}\,\dt \le C_{\e,\zeta,T}\tau_n^\vartheta.
\]
Consequently,
\bq\label{eq:tau-pressure}
\tau_n\rho_n^{10}\to0 \quad \text{strongly in }L^1((0,T)\times\T^d).
\eq
In particular,
\[
\tau_n\nabla\rho_n^{10}\to0 \quad \text{in }\calD'((0,T)\times\T^d).
\]

\medskip

\noindent
\textit{Capillarity.}
Define
\[
B_{\zeta,n}:= 2\nu_\zeta(\rho_n)\nabla^2s_\zeta(\rho_n) +\lambda_\zeta(\rho_n)\Delta s_\zeta(\rho_n)\mathbb I.
\]
By the generalized Bohm identity \eqref{eq:Bohm}, we note $\mathfrak K_\zeta(\rho_n)=2\nabla\cdot B_{\zeta,n}$. We claim that
\[
\tau_n B_{\zeta,n}\to0 \quad \text{strongly in }L^1((0,T)\times\T^d).
\]
When $d=1$, this follows directly from $2\nu_\zeta(\rho)+\lambda_\zeta(\rho)=2\rho\nu_\zeta'(\rho)=2\alpha\rho^\alpha+2\zeta\rho\ge0$ and the capillarity dissipation. When $d\ge2$, we write
\[
\nabla^2s_\zeta(\rho_n) = (\nabla^2s_\zeta(\rho_n))^0 +\frac1d\Delta s_\zeta(\rho_n)\mathbb I.
\]
Then
\[
B_{\zeta,n} = 2\nu_\zeta(\rho_n)(\nabla^2s_\zeta(\rho_n))^0 + a_\zeta(\rho_n)\Delta s_\zeta(\rho_n)\mathbb I,
\]
where
\[ 
a_\zeta(\rho):=\frac2d\nu_\zeta(\rho)+\lambda_\zeta(\rho) = 2\lt(\alpha-1+\frac1d\rt)\rho^\alpha+\frac{2\zeta}{d}\rho.
\]
Under the condition $\alpha\ge1-\frac1d$, we have $a_\zeta(\rho)\ge0$. Moreover,
\[
2\nu_\zeta(\rho_n)|\nabla^2s_\zeta(\rho_n)|^2 +\lambda_\zeta(\rho_n)|\Delta s_\zeta(\rho_n)|^2 = 2\nu_\zeta(\rho_n)|(\nabla^2s_\zeta(\rho_n))^0|^2 +a_\zeta(\rho_n)|\Delta s_\zeta(\rho_n)|^2.
\]
Since $\zeta>0$ is fixed and $0<\zeta\le1$ may be assumed, the energy bound gives
\[
\sup_n \esssup_{0\le t \le T} \intt[\nu_\zeta(\rho_n)+a_\zeta(\rho_n)]\,\dx \le C_{\e,\zeta,T}.
\]
The Cauchy--Schwarz inequality and the capillarity estimate in \eqref{eq:tau-unif} yield
\[
\tau_n\int_0^T\intt |B_{\zeta,n}|\,\dx\dt \le C_{\e,\zeta,T}\sqrt{\tau_n}.
\]
Thus
\bq\label{eq:tau-cap}
\tau_n\mathfrak K_\zeta(\rho_n)\to0 \quad \text{strongly in }L^1(0,T;W^{-1,1}(\T^d)).
\eq

We now pass to the limit in \eqref{eq:tau-sys}. We give the details of the compactness and of the passage to the limit in the weak formulations. The argument follows the BD-compatible compactness mechanism for density-dependent viscosities, but here $\zeta>0$ is fixed and the artificial terms have already been shown to vanish.

By the uniform energy and BD estimates,
\[
\sup_n  \esssup_{0\le t \le T} \intt[\rho_n^\gamma+\rho_n|u_n|^2]\,\dx
+\sup_n\int_0^T\intt|\nabla\rho_n^\alpha|^2\,\dx\dt
\le C_{\e,\zeta,T}.
\]
Moreover,
\[
\sup_n \esssup_{0\le t \le T} 
\intt\rho_n|u_n+\nabla\phi_\zeta(\rho_n)|^2\,\dx
+
\sup_n\int_0^T\intt\nu_\zeta(\rho_n)|\nabla u_n|^2\,\dx\dt
\le C_{\e,\zeta,T}.
\]
The density consequences required in Lemma \ref{lem:BD-convec-comp} are provided by Proposition \ref{prop:app-density-conseq}, uniformly with respect to $\tau_n$. Thus Lemma \ref{lem:BD-convec-comp}, applied first to the density compactness part, yields, up to a subsequence,
\bq\label{eq:tau-rho-conv}
\rho_n \to \rho_\zeta \quad  \text{strongly in }L^1((0,T)\times\T^d) \text{ and almost everywhere}.
\eq

We now prove the strong convergence of the pressure. The almost everywhere convergence of $\rho_n$ gives
\[
\rho_n^\gamma\to\rho_\zeta^\gamma \quad \text{almost everywhere in }(0,T)\times\T^d.
\]
It remains to prove uniform integrability. We show that there exists $\delta>0$, independent of $n$, such that
\[
\sup_n\int_0^T\intt\rho_n^{\gamma+\delta}\,\dx\dt \le C_{\e,\zeta,T}.
\]
Set $z_n:=\rho_n^\alpha$. Then $z_n$ is bounded in  $L^\infty(0,T;L^{\gamma/\alpha}(\T^d))$ and  $\nabla z_n$ is bounded in $L^2((0,T)\times\T^d)$. Choose $\delta>0$ sufficiently small and set
\[
p_0:=\frac{\gamma}{\alpha}, \quad p_\delta:=\frac{\gamma+\delta}{\alpha}.
\]
Since $d\le2$, the Gagliardo--Nirenberg inequality gives, for some $\vartheta_\delta\in(0,1)$,
\[
\|z_n\|_{L^{p_\delta}(\T^d)} \le C\|\nabla z_n\|_{L^2(\T^d)}^{\vartheta_\delta} \|z_n\|_{L^{p_0}(\T^d)}^{1-\vartheta_\delta} +C\|z_n\|_{L^{p_0}(\T^d)}.
\]
Taking $\delta>0$ small enough so that $\vartheta_\delta p_\delta\le2$, raising the preceding inequality to the power $p_\delta$, and integrating in time, we obtain
\[
\sup_n\int_0^T\|z_n(t)\|_{L^{p_\delta}(\T^d)}^{p_\delta}\,\dt \le C_{\e,\zeta,T}.
\]
Equivalently,
\[
\sup_n\int_0^T\intt\rho_n^{\gamma+\delta}\,\dx\dt \le C_{\e,\zeta,T}.
\]
Thus $\{\rho_n^\gamma\}_n$ is uniformly integrable in $L^1((0,T)\times\T^d)$. By Vitali's convergence theorem,
\bq\label{eq:tau-pressure-conv}
\rho_n^\gamma \to \rho_\zeta^\gamma \quad  \text{strongly in }L^1((0,T)\times\T^d).
\eq

We next consider the momentum. From the energy estimate,
\[
\sup_n \esssup_{0\le t \le T} \intt \rho_n|u_n|^2\,\dx \le C_{\e,\zeta,T}.
\]
Thus, after extracting a subsequence,
\[
\sqrt{\rho_n}u_n\rightharpoonup U \quad \text{weakly in }L^2((0,T)\times\T^d).
\]
Since $\rho_n\to\rho_\zeta$ strongly in $L^1((0,T)\times\T^d)$, we have
\[
\sqrt{\rho_n}\to\sqrt{\rho_\zeta} \quad \text{strongly in }L^2((0,T)\times\T^d).
\]
We define
\[
u_\zeta :=
\begin{cases}
\dfrac{U}{\sqrt{\rho_\zeta}}, & \rho_\zeta>0,\\[1mm]
0, & \rho_\zeta=0.
\end{cases}
\]
Then $U=\sqrt{\rho_\zeta}u_\zeta$ a.e. and $\sqrt{\rho_\zeta}u_\zeta\in L^2((0,T)\times\T^d)$. Moreover, for every test function $\Psi\in C_c^\infty((0,T)\times\T^d;\R^d)$,
\[
\int_0^T\intt \rho_nu_n\cdot\Psi\,\dx\dt = \int_0^T\intt (\sqrt{\rho_n}u_n)\cdot(\sqrt{\rho_n}\Psi)\,\dx\dt.
\]
Since
\[
\sqrt{\rho_n}\Psi\to\sqrt{\rho_\zeta}\Psi \quad \text{strongly in }L^2((0,T)\times\T^d),
\]
we obtain
\[
\int_0^T\intt \rho_nu_n\cdot\Psi\,\dx\dt \to \int_0^T\intt U\cdot(\sqrt{\rho_\zeta}\Psi)\,\dx\dt = \int_0^T\intt \rho_\zeta u_\zeta\cdot\Psi\,\dx\dt.
\]
Hence
\bq\label{eq:tau-mom-conv}
\rho_nu_n \rightharpoonup \rho_\zeta u_\zeta \quad \text{in }\calD'((0,T)\times\T^d).
\eq

The remaining conclusions of Lemma \ref{lem:BD-convec-comp}, applied with $\zeta_n=\zeta$ fixed, give
\[
\nu_\zeta(\rho_n)\to\nu_\zeta(\rho_\zeta), \quad \lambda_\zeta(\rho_n)\to\lambda_\zeta(\rho_\zeta) \quad \text{strongly in } L^p((0,T)\times\T^d).
\]
Moreover,

\bq\label{eq:tau-stress-conv}
\mathbb S_\zeta(\rho_n,u_n) \rightharpoonup \mathbb S_\zeta(\rho_\zeta,u_\zeta) \quad\text{in }\calD'((0,T)\times\T^d).
\eq

 It remains to identify the convective term. Simiarly as before, we first observe that 
 \[
\nabla m_n
\quad\text{is bounded in }
L^2(0,T;L^1(\T^d)).
\]

On the other hand, the momentum equation
\eqref{eq:tau-sys}, together with the energy, BD, drag, and force
estimates and \eqref{eq:tau-pressure}--\eqref{eq:tau-u}, gives, for
some sufficiently large integer $M$,
\[
\pa_t m_n
\quad\text{bounded in }
L^1(0,T;W^{-M,1}(\T^d)).
\]
Indeed, the convective and pressure terms are bounded in
$L^\infty(0,T;W^{-1,1}(\T^d))$, the stress and averaged-force terms
are bounded in suitable integrable negative Sobolev spaces, and the
artificial terms are controlled by
\eqref{eq:tau-pressure}, \eqref{eq:tau-cap}, and \eqref{eq:tau-u}.

The Aubin--Lions compactness lemma therefore yields, up to a
subsequence,
\[
m_n\to\rho_\zeta u_\zeta
\quad\text{strongly in }
L^1((0,T)\times\T^d)
\]
and almost everywhere. Here the strong limit is identified by
\eqref{eq:tau-mom-conv}. Since
\[
\rho_n\to\rho_\zeta
\quad\text{almost everywhere}
\]
and $\rho_\zeta>0$ almost everywhere, we conclude that
\[
u_n\to u_\zeta
\quad\text{almost everywhere in }
(0,T)\times\T^d.
\]
Remark \ref{rem:BD-convective-comp} now gives
\bq\label{eq:tau-convective-conv}
\rho_nu_n\otimes u_n
\to
\rho_\zeta u_\zeta\otimes u_\zeta
\quad\text{strongly in }
L^1((0,T)\times\T^d).
\eq
In particular, the convergence also holds in
$\calD'((0,T)\times\T^d)$.

We now provide the compactness of the kinetic component. The conservation of
the kinetic mass and the kinetic energy bound give
\[
\sup_n\esssup_{0\le t\le T} \inttr (1+|v|^2)f_n\,\dx\dv \le C_{\e,T}.
\]
Using the density-like notation for measure-valued kinetic limits introduced in the Galerkin limit, we infer, up to a subsequence,
\bq\label{eq:tau-f-conv}
f_n\stackrel{*}{\rightharpoonup}f_\zeta \quad\text{weakly-* in }L^\infty(0,T;\calM_+(\T^d\times\R^d)).
\eq
Moreover, the uniform second velocity moment gives the velocity tightness
\[
\sup_n\int_0^T\int_{\T^d\times\{|v|>R\}} f_n\,\dx\dv\dt \le \frac{C_{\e,T}}{R^2},
\]
and the limiting kinetic component has mass one and finite second velocity
moment:
\[
\inttr f_\zeta\,\dx\dv=1 \quad\text{for a.e. }t\in(0,T), \quad \esssup_{0\le t\le T} \inttr (1+|v|^2)f_\zeta\,\dx\dv \le C_{\e,T}.
\]
As in the Galerkin limit, these facts follow from the weak-* convergence, the uniform velocity tightness, and the lower semicontinuity argument with compact velocity cutoffs.

Consequently, the velocity moments pass to the limit by the same truncation argument:
\[
\rho_{f_n}:=\intr f_n\,\dv \rightharpoonup \rho_{f_\zeta}:=\intr f_\zeta\,\dv \quad\text{in }\calD'((0,T)\times\T^d),
\]
and
\bq\label{eqLtau-f-m}
\intr v f_n\,\dv \rightharpoonup \intr v f_\zeta\,\dv \quad\text{in }\calD'((0,T)\times\T^d).
\eq
Indeed, the only additional point for the momentum is the tail estimate
\[
\int_0^T\iint_{\T^d\times\{|v|>R\}} |v|f_n\,\dx\dv\dt \le \frac1R\int_0^T\inttr |v|^2f_n\,\dx\dv\dt.
\]
 
We next verify the stability of the mollified kinetic--fluid coupling. The strong convergence of $\rho_n$ and the conservation of mass imply that, for every fixed $k\ge0$ and every finite $q\ge1$,
\bq\label{eq:tau-rhoeps-conv}
(\rho_n)_\e\to(\rho_\zeta)_\e \quad \text{strongly in }L^q(0,T;C^k(\T^d)).
\eq

We also need the strong convergence of the mollified momentum. Let $m_n:=\rho_nu_n$. We claim that, for every fixed $k\ge0$,
\bq\label{eq:tau-momeps-conv}
(m_n)_\e\to(\rho_\zeta u_\zeta)_\e \quad \text{strongly in }L^q(0,T;C^k(\T^d))
\eq
for every finite $q\ge1$. To see this, we use the momentum equation. For a test function $\psi\in C^\infty(\T^d;\R^d)$, the weak form of the $n$-th momentum equation gives
\[
\begin{aligned}
\lal \pa_t m_n,\psi\ral &= \intt \rho_nu_n\otimes u_n:\nabla\psi\,\dx +\intt(\rho_n^\gamma+\tau_n\rho_n^{10})\nabla\cdot\psi\,\dx -\intt \mathbb S_\zeta(\rho_n,u_n):\nabla\psi\,\dx \\
&\quad +\intt \rho_n(F_n)_\e\cdot\psi\,\dx  +\tau_n\intt \mathfrak K_\zeta(\rho_n)\cdot\psi\,\dx -\tau_n\intt u_n\cdot\psi\,\dx.
\end{aligned}
\]
The terms on the right-hand side are bounded uniformly in suitable negative spaces. The convective and pressure terms are bounded in $L^\infty(0,T;W^{-1,1}(\T^d))$. The stress term is bounded in $L^1(0,T;W^{-1,1}(\T^d))$, since
\[
\int_0^T\intt|\mathbb S_\zeta(\rho_n,u_n)|\,\dx\dt \le C\lt(\int_0^T\intt\nu_\zeta(\rho_n)|\nabla u_n|^2\,\dx\dt\rt)^{\frac12} \lt(\int_0^T\intt\nu_\zeta(\rho_n)\,\dx\dt\rt)^{\frac12} \le C_{\e,\zeta,T}.
\]
The force term is bounded in $L^2(0,T;W^{-k,1}(\T^d))$ for every $k\ge0$, as shown below in \eqref{eq:tau-force-bound}. The capillarity and damping terms are bounded in negative spaces by \eqref{eq:tau-cap} and \eqref{eq:tau-u}. Hence, $\pa_t m_n$ is bounded in $L^1(0,T;W^{-M,1}(\T^d))$ for some sufficiently large $M$. After convolution in space, this implies $\pa_t(m_n)_\e$ is bounded in $L^1(0,T;C^k(\T^d))$ for every fixed $k\ge0$. On the other hand, $m_n$ is bounded in $L^\infty(0,T;L^{\frac{2\gamma}{\gamma+1}}(\T^d))$, and therefore $(m_n)_\e$ is bounded in $L^\infty(0,T;C^k(\T^d))$ for every $k$. By the Aubin--Lions lemma after convolution, $(m_n)_\e$ is compact in $L^q(0,T;C^k(\T^d))$ for every finite $q\ge1$. Since $m_n\rightharpoonup\rho_\zeta u_\zeta$ in distributions, the strong limit is necessarily $(\rho_\zeta u_\zeta)_\e$, which proves \eqref{eq:tau-momeps-conv}.

From \eqref{eq:tau-rhoeps-conv}, \eqref{eq:tau-momeps-conv}, and the strict lower bound $(\rho_n)_\e\ge c_\e>0$, we infer
\[
u_{n,\e}^\e = \frac{(\rho_nu_n)_\e}{(\rho_n)_\e} \to \frac{(\rho_\zeta u_\zeta)_\e}{(\rho_\zeta)_\e} =:u_{\zeta,\e}^\e \quad \text{strongly in } L^q(0,T;C^k(\T^d))
\]
for every fixed $k\ge0$ and every finite $q\ge1$.

We now identify the weak limit of the force. Note that
\[
\intt |F_n|\,\dx \le \lt(\intt\rho_{f_n}\,\dx\rt)^{\frac12} \lt(\inttr |v-u_{n,\e}^\e|^2f_n\,\dx\dv\rt)^{\frac12} \le C_\e \calD_{{\rm drag},n}(t)^{\frac12},
\]
and thus, by convolution,
\bq\label{eq:tau-force-bound}
\|(F_n)_\e\|_{L^2(0,T;W^{k,\infty}(\T^d))} \le C_{\e,k}
\eq
for every fixed $k\ge0$. Since
\[
F_n=\intr v f_n\,\dv-u_{n,\e}^\e\intr f_n\,\dv,
\]
the convergence of the first term follows from \eqref{eqLtau-f-m}. For the second term, the strong convergence of $u_{n,\e}^\e$ and the weak convergence of $\rho_{f_n}$ imply $u_{n,\e}^\e\rho_{f_n} \rightharpoonup u_{\zeta,\e}^\e\rho_{f_\zeta}$ in  $\calD'((0,T)\times\T^d)$. Hence,
\[
F_n\rightharpoonup F_\zeta \quad \text{in }\calD'((0,T)\times\T^d), \quad \text{where } F_\zeta:=\intr (v-u_{\zeta,\e}^\e)f_\zeta\,\dv.
\]
Together with \eqref{eq:tau-force-bound}, this gives, for every $k\ge0$, $(F_n)_\e\rightharpoonup(F_\zeta)_\e$ weakly in $L^2(0,T;W^{k,\infty}(\T^d))$. Since $\rho_n\to\rho_\zeta$ strongly in $L^2(0,T;L^1(\T^d))$, we conclude that
\bq\label{eq:tau-force-conv}
\rho_n(F_n)_\e \rightharpoonup \rho_\zeta(F_\zeta)_\e \quad \text{in }\calD'((0,T)\times\T^d).
\eq

We now pass to the limit in the weak formulations. Let $\varphi\in C_c^\infty([0,T)\times\T^d\times\R^d)$. The kinetic equation for $f_n$ reads
\[
\begin{aligned}
&\int_0^T\inttr
f_n(\pa_t\varphi+v\cdot\nabla_x\varphi)\,\dx\dv\dt + \inttr f_{0,n}\varphi(0)\,\dx\dv \\
&\quad = \int_0^T\inttr (\rho_n)_\e(v-u_{n,\e}^\e)f_n\cdot\nabla_v\varphi\,\dx\dv\dt.
\end{aligned}
\]
Since $\varphi$ is compactly supported in $v$, the weak-* convergence \eqref{eq:tau-f-conv}, the strong convergence of $(\rho_n)_\e$, and the strong convergence of $u_{n,\e}^\e$ allow us to pass to the limit in every term. Using also $f_{0,n}\to f_0$ in $W_2$, we obtain the weak formulation of
\[
\pa_t f_\zeta+v\cdot\nabla_x f_\zeta = (\rho_\zeta)_\e\nabla_v\cdot\lt[(v-u_{\zeta,\e}^\e)f_\zeta\rt].
\]

The continuity equation is treated similarly. For every $\psi\in C_c^\infty([0,T)\times\T^d)$,
\[
\int_0^T\intt \lt[\rho_n\pa_t\psi+\rho_nu_n\cdot\nabla\psi\rt]\dx\dt + \intt\rho_{0,n}\psi(0)\,\dx=0.
\]
Using \eqref{eq:tau-rho-conv}, \eqref{eq:tau-mom-conv}, and $\rho_{0,n}\to\rho_0$ in $L^\gamma$, we obtain
\[
\int_0^T\intt \lt[\rho_\zeta\pa_t\psi+\rho_\zeta u_\zeta\cdot\nabla\psi\rt]\dx\dt + \intt\rho_0\psi(0)\,\dx=0.
\]

It remains to pass to the momentum equation. For every $\Psi\in C_c^\infty([0,T)\times\T^d;\R^d)$, the weak formulation of the $n$-th momentum equation is
\begin{align}
\label{eq:tau-momentum-weak}
\begin{aligned}
&\int_0^T\intt \lt[ \rho_nu_n\cdot\pa_t\Psi +\rho_nu_n\otimes u_n:\nabla\Psi +\rho_n^\gamma\nabla\cdot\Psi \rt]\dx\dt +\int_0^T\intt\tau_n\rho_n^{10}\nabla\cdot\Psi\,\dx\dt \\
&\quad  -\int_0^T\intt\mathbb S_\zeta(\rho_n,u_n):\nabla\Psi\,\dx\dt  +\int_0^T\intt\rho_n(F_n)_\e\cdot\Psi\,\dx\dt +\int_0^T\lt\lal\tau_n\mathfrak K_\zeta(\rho_n),\Psi\rt\ral\dt \\ 
&\quad -\int_0^T\intt\tau_nu_n\cdot\Psi\,\dx\dt +\intt \rho_{0,n}u_{0,n}\cdot\Psi(0)\,\dx=0.
\end{aligned}
\end{align}
The first three principal terms converge by \eqref{eq:tau-mom-conv}, \eqref{eq:tau-convective-conv}, and \eqref{eq:tau-pressure-conv}. The stress term converges by \eqref{eq:tau-stress-conv}, and the force term converges by \eqref{eq:tau-force-conv}. The initial momentum term converges by the approximation of the initial data.

We now check that the artificial terms vanish. The artificial pressure has already been estimated in \eqref{eq:tau-pressure}; in weak form,
\[
\lt| \int_0^T\intt\tau_n\rho_n^{10}\nabla\cdot\Psi\,\dx\dt \rt| \le \|\nabla\cdot\Psi\|_{L^\infty} \int_0^T\intt\tau_n\rho_n^{10}\,\dx\dt \to0.
\]
The capillarity term vanishes by \eqref{eq:tau-cap}:
\[
\lt| \int_0^T\lt\lal\tau_n\mathfrak K_\zeta(\rho_n),\Psi\rt\ral\dt \rt| \le \|\Psi\|_{L^\infty(0,T;W^{1,\infty})} \|\tau_n\mathfrak K_\zeta(\rho_n)\|_{L^1(0,T;W^{-1,1})} \to0.
\]
Finally, the damping term vanishes by \eqref{eq:tau-u}:
\[
\lt| \int_0^T\intt\tau_nu_n\cdot\Psi\,\dx\dt \rt| \le \|\Psi\|_{L^2((0,T)\times\T^d)} \|\tau_nu_n\|_{L^2((0,T)\times\T^d)} \to 0.
\]
Passing to the limit in \eqref{eq:tau-momentum-weak}, we obtain
\begin{align*}
&\int_0^T\intt \lt[ \rho_\zeta u_\zeta\cdot\pa_t\Psi +\rho_\zeta u_\zeta\otimes u_\zeta:\nabla\Psi +\rho_\zeta^\gamma\nabla\cdot\Psi \rt]\dx\dt \\
&\quad -\int_0^T\intt \mathbb S_\zeta(\rho_\zeta,u_\zeta):\nabla\Psi\,\dx\dt +\int_0^T\intt \rho_\zeta(F_\zeta)_\e\cdot\Psi\,\dx\dt +\intt\rho_0u_0\cdot\Psi(0)\,\dx=0.
\end{align*}

Hence, $(f_\zeta,\rho_\zeta,u_\zeta)$ satisfies
\begin{align}\label{eq:zeta-sys}
\begin{aligned}
&\pa_t f_\zeta+v\cdot\nabla_x f_\zeta = (\rho_\zeta)_\e\nabla_v\cdot\lt[(v-u_{\zeta,\e}^\e)f_\zeta\rt], \\
&\pa_t\rho_\zeta+\nabla\cdot(\rho_\zeta u_\zeta)=0, \\
&\pa_t(\rho_\zeta u_\zeta) +\nabla\cdot(\rho_\zeta u_\zeta\otimes u_\zeta) +\nabla\rho_\zeta^\gamma = \nabla\cdot\mathbb S_\zeta(\rho_\zeta,u_\zeta) +\rho_\zeta(F_\zeta)_\e,
\end{aligned}
\end{align}
where
\[
u_{\zeta,\e}^\e = \frac{(\rho_\zeta u_\zeta)_\e}{(\rho_\zeta)_\e}.
\]

Finally, we pass to the limit in the integrated form of the free energy inequality and the combined estimate. The kinetic energy, the fluid kinetic energy, the pressure potential, and the BD energy are lower semicontinuous under the convergences above. The dissipations are also lower semicontinuous, since they are convex in the weakly convergent quantities. Hence, for almost every  $t\in[0,T]$,
\begin{align}\label{eq:zeta-combined}
\begin{aligned}
&\calE_{{\rm free},\zeta}(t) +\eta\calE_{{\rm BD},\zeta}(t) +\int_{0}^t \lt[ \calD_{{\rm vis},\zeta}(\tau) +\frac12\calD_{{\rm drag},\zeta}(\tau) +\frac{\eta}{2}\calD_{{\rm BD},\zeta}(\tau) +\eta\calD_{{\rm rot},\zeta}(\tau) \rt]\dd\tau \\
&\quad \le \calE_{{\rm free},\zeta}(0) +\eta\calE_{{\rm BD},\zeta}(0),
\end{aligned}
\end{align}
where
\[
\calE_{{\rm BD},\zeta}(t) := \frac12\inttr |v|^2f_\zeta\,\dx\dv + \intt\lt[ \frac12\rho_\zeta|u_\zeta+\nabla\phi_\zeta(\rho_\zeta)|^2 +\pi(\rho_\zeta) \rt]\dx.
\]
The corresponding uniform bound is independent of $\zeta$.

%
%
%
%
%
%
%
%
%
%
%

\subsection{Uniform estimates and compactness for the limit $\zeta\to0$}\label{ssec:zeta-est}

We now derive estimates that are uniform with respect to the viscosity regularization parameter. Let $\zeta_n\to0$ and denote the corresponding solutions to \eqref{eq:zeta-sys} by $(f_n,\rho_n,u_n)$. For simplicity, we write
\[
\nu_n(\rho):=\rho^\alpha+\zeta_n\rho, \quad \lambda_n(\rho):=2(\alpha-1)\rho^\alpha,
\]
and
\[
u_{n,\e}^\e:=\frac{(\rho_nu_n)_\e}{(\rho_n)_\e}, \quad F_n:=\intr (v-u_{n,\e}^\e)f_n\,\dv.
\]
The system reads
\begin{align}\label{eq:zeta-sys-n}
\begin{aligned}
&\pa_t f_n+v\cdot\nabla_x f_n =(\rho_n)_\e\nabla_v\cdot\lt[(v-u_{n,\e}^\e)f_n\rt],\\
&\pa_t\rho_n+\nabla\cdot(\rho_nu_n)=0,\\
&\pa_t(\rho_nu_n)+\nabla\cdot(\rho_nu_n\otimes u_n)+\nabla\rho_n^\gamma =\nabla\cdot\mathbb S_{\zeta_n}(\rho_n,u_n)+\rho_n(F_n)_\e.
\end{aligned}
\end{align}
Here
\[
\mathbb S_{\zeta_n}(\rho_n,u_n) = 2\nu_n(\rho_n)D u_n+\lambda_n(\rho_n)(\nabla\cdot u_n)\mathbb I.
\]

By \eqref{eq:zeta-combined}, we have, uniformly with respect to $n$,
\[
 \esssup_{0\le t \le T} \lt[ \calE_{{\rm free},n}(t) +\eta\calE_{{\rm BD},n}(t) \rt]   +\int_0^T \lt[ \calD_{{\rm vis},n}(t) +\calD_{{\rm drag},n}(t) +\eta\calD_{{\rm BD},n}(t) +\eta\calD_{{\rm rot},n}(t) \rt]\dt \le C_\e.
\]
In particular,
\bq\label{eq:zeta-basic-bounds}
\sup_n \esssup_{0\le t \le T} \intt\lt[\rho_n^\gamma+\rho_n|u_n|^2\rt]\dx + \sup_n\int_0^T\intt|\nabla\rho_n^\alpha|^2\,\dx\dt \le C_\e,
\eq
and
\[
\sup_n  \esssup_{0\le t \le T}  \intt\rho_n|u_n+\nabla\phi_{\zeta_n}(\rho_n)|^2\,\dx + \sup_n\int_0^T\intt\nu_n(\rho_n)|\nabla u_n|^2\,\dx\dt \le C_\e.
\]
Moreover, using the fluid kinetic part of the free energy together with the effective part of the BD energy, we also have
\bq\label{eq:zeta-effective-gradient}
\sup_n  \esssup_{0\le t \le T} \intt\rho_n|\nabla\phi_{\zeta_n}(\rho_n)|^2\,\dx \le C_\e.
\eq
Indeed,
\[
\rho_n|\nabla\phi_{\zeta_n}(\rho_n)|^2 \le 2\rho_n|u_n+\nabla\phi_{\zeta_n}(\rho_n)|^2 + 2\rho_n|u_n|^2.
\]

The estimates \eqref{eq:zeta-basic-bounds} and \eqref{eq:zeta-effective-gradient} are uniform in $\zeta_n$. Thus, the weighted density consequences verified in Proposition \ref{prop:veri-w-den-c} apply uniformly to the sequence $(\rho_n)_n$. More precisely, in one space dimension, for $0<\alpha\le\frac12$, we have
\bq\label{eq:zeta-1d-density-bounds}
0<\underline\rho\le\rho_n(t,x)\le\overline\rho<\infty
\eq
for almost every $(t,x)\in(0,T)\times\T$, with constants independent of $n$. In the two-dimensional borderline case $d=2$, $\alpha=\frac12$, we have, for every finite $p>0$,
\bq\label{eq:zeta-density-moments-2d}
\sup_n\esssup_{0\le t\le T} \int_{\T^2}\lt(\rho_n(t,x)^p+\rho_n(t,x)^{-p}\rt)\dx \le C_{\e,p},
\eq
and
\bq\label{eq:zeta-A2-bound}
\sup_n\esssup_{0\le t\le T} [\rho_n(t)^\alpha]_{\calA_2} \le C_\e.
\eq
Here, in the case $d=2$, $\alpha=\frac12$, this is the uniform $\calA_2$ bound for the weight $\sqrt{\rho_n}$.

We next provide the velocity-gradient consequences. In one space dimension, the pointwise density bounds \eqref{eq:zeta-1d-density-bounds}, together with the viscous dissipation, give that
\bq\label{eq:zeta-unweighted-gradient-1d}
u_n \text{ is bounded in } L^2(0,T;H^1(\T)).
\eq
In the two-dimensional borderline case, the $\calA_2$ estimate \eqref{eq:zeta-A2-bound} and the weighted Korn inequality for $\calA_2$ weights yield
\[
\int_{\T^2}\sqrt{\rho_n}|\nabla u_n|^2\,\dx \le C_\e \int_{\T^2}\sqrt{\rho_n}|D^0u_n|^2\,\dx
\]
for almost every $t\in(0,T)$. Since the viscous dissipation controls the right-hand side, we obtain
\bq\label{eq:zeta-weighted-velocity-gradient}
\int_0^T\int_{\T^2} \sqrt{\rho_n}|\nabla u_n|^2\,\dx\dt \le C_\e.
\eq
Thus, in every admissible case,
\bq\label{eq:zeta-weighted-gradient-all}
\int_0^T\intt\rho_n^\alpha|\nabla u_n|^2\,\dx\dt \le C_\e.
\eq
Indeed, in one space dimension this follows from \eqref{eq:zeta-1d-density-bounds} and \eqref{eq:zeta-unweighted-gradient-1d}, while in the two-dimensional borderline case it is exactly \eqref{eq:zeta-weighted-velocity-gradient}.

In two space dimensions, \eqref{eq:zeta-density-moments-2d} and \eqref{eq:zeta-weighted-velocity-gradient} also imply, for every $1<q<2$, $\nabla u_n$ is bounded in $L^2(0,T;L^q(\T^2))$. Indeed, H\"older's inequality gives
\[
\int_{\T^2}|\nabla u_n|^q\,\dx \le \lt( \int_{\T^2}\rho_n^{\frac12}|\nabla u_n|^2\,\dx \rt)^{q/2} \lt( \int_{\T^2}\rho_n^{-\frac{q}{2(2-q)}}\,\dx \rt)^{(2-q)/2}.
\]
The last factor is uniformly bounded by \eqref{eq:zeta-density-moments-2d}.

We now provide the compactness consequences. The density compactness part of Lemma \ref{lem:BD-convec-comp}, applied with $\nu_n(\rho)=\rho^\alpha+\zeta_n\rho$, $\phi_n=\phi_{\zeta_n}$, gives, up to a subsequence, $\rho_n \to \rho$ strongly in $L^1((0,T)\times\T^d)$ and almost everywhere. The uniform density consequences above improve this convergence. In one space dimension, the pointwise bounds \eqref{eq:zeta-1d-density-bounds} and dominated convergence give $\rho_n\to\rho$ strongly in $L^q((0,T)\times\T)$ for every finite $q$. In the two-dimensional borderline case, the moment estimates \eqref{eq:zeta-density-moments-2d} imply uniform integrability of $\rho_n^q$ for every finite $q$. Hence, by Vitali's theorem, $\rho_n\to\rho$ strongly in $L^q((0,T)\times\T^2)$ for every finite $q$. In particular,
\bq\label{eq:zeta-pressure-conv}
\rho_n^\gamma \to \rho^\gamma \quad \text{strongly in }L^1((0,T)\times\T^d).
\eq

We next provide the corresponding coefficient convergence. Since $\nu_n(\rho_n)=\rho_n^\alpha+\zeta_n\rho_n$, the strong convergence of $\rho_n$ in every finite $L^p$ gives $\rho_n^\alpha\to\rho^\alpha$ strongly in $L^p((0,T)\times\T^d)$ for every finite $p\ge1$.   Moreover, by the uniform density moment
bounds,
\[
\|\zeta_n\rho_n\|_{L^p((0,T)\times\T^d)} \le \zeta_n \|\rho_n\|_{L^p((0,T)\times\T^d)} \to0.
\] 
Hence, $\nu_n(\rho_n) \to \rho^\alpha$ strongly in $L^p((0,T)\times\T^d)$. Similarly, $\lambda_n(\rho_n) = 2(\alpha-1)\rho_n^\alpha \to 2(\alpha-1)\rho^\alpha$ strongly in $L^p((0,T)\times\T^d)$ for every finite $p\ge1$.

Thus, applying the full conclusion of Lemma \ref{lem:BD-convec-comp}, we obtain a velocity $u$ such that
\[
\rho_nu_n \rightharpoonup \rho u, \quad \mathbb S_{\zeta_n}(\rho_n,u_n) \rightharpoonup \mathbb S(\rho,u) \quad 
\text{in }\calD'((0,T)\times\T^d),
\]
where $\mathbb S(\rho,u) = 2\rho^\alpha D u + 2(\alpha-1)\rho^\alpha(\nabla\cdot u)\mathbb I$.

It remains to identify the convective term. We use the weighted-velocity compactness argument of \cite[Lemmas 2.3--2.4]{BVY22}. Let $\Phi\in C^\infty((0,\infty))$ be strictly positive and satisfy
\[
\Phi(r)+|\Phi'(r)| \le
\begin{cases}
Ce^{-\frac 1r},&0<r\le1,\\
Ce^{-r},&r\ge2.
\end{cases}
\]
Since $\nu_n(r)\ge r^\alpha$ and $\nu_n'(r)\ge\alpha r^{\alpha-1}$, we have
\[
\sup_n \lt\| \frac{\Phi}{\sqrt{\nu_n}} \rt\|_{L^\infty(0,\infty)} + \sup_n \lt\| \frac{\Phi'}{\nu_n'} \rt\|_{L^\infty(0,\infty)} <\infty.
\]
The energy and BD estimates therefore provide the spatial bounds required in the proof of \cite[Lemma 2.3]{BVY22} for $\Phi(\rho_n)u_n$.

The additional averaged-force term is also admissible. Indeed, the drag estimate and the spatial convolution give, for every fixed $k\ge0$,
\[
\|(F_n)_\e\|_{L^2(0,T;W^{k,\infty}(\T^d))} \le C_{\e,k}.
\]
Since $\Phi$ is bounded,
\[
\Phi(\rho_n)(F_n)_\e \quad\text{is bounded in } L^1((0,T)\times\T^d).
\]
Consequently, the weighted-velocity compactness argument gives, after extracting a subsequence,
\[
\Phi(\rho_n)u_n \to \Phi(\rho)u \quad\text{strongly in } L^1((0,T)\times\T^d)
\]
and almost everywhere.

The density estimates imply
\[
\rho>0 \quad\text{almost everywhere}.
\]
Since $\rho_n\to\rho$ almost everywhere and $\Phi(\rho)>0$, we conclude that
\[
u_n\to u \quad\text{almost everywhere in } (0,T)\times\T^d.
\]
Remark \ref{rem:BD-convective-comp} thus gives
\[
\rho_nu_n\otimes u_n \to \rho u\otimes u \quad\text{strongly in } L^1((0,T)\times\T^d).
\]

Altogether, we have
\begin{align}\label{eq:zeta-conv}
\begin{aligned}
\rho_nu_n &\rightharpoonup \rho u &&\text{in }\calD'((0,T)\times\T^d), \\
\rho_nu_n\otimes u_n &\to \rho u\otimes u &&\text{strongly in }L^1((0,T)\times\T^d), \\
\mathbb S_{\zeta_n}(\rho_n,u_n) &\rightharpoonup \mathbb S(\rho,u) &&\text{in }\calD'((0,T)\times\T^d).
\end{aligned}
\end{align}

Finally, we note that the purely artificial part of the viscosity vanishes strongly. From the viscous dissipation,
\[
\zeta_n\int_0^T\intt\rho_n|\nabla u_n|^2\,\dx\dt
\le C_\e.
\]
Using the conservation of mass, we get
\[
\lt\|\zeta_n\rho_n\nabla u_n\rt\|_{L^1((0,T)\times\T^d)} \le \lt( \zeta_n\int_0^T\intt\rho_n|\nabla u_n|^2\,\dx\dt \rt)^{\frac12} \lt( \zeta_n\int_0^T\intt\rho_n\,\dx\dt \rt)^{\frac12} \le C_{\e,T}\sqrt{\zeta_n} \to0.
\]
This is consistent with the stress convergence in \eqref{eq:zeta-conv} and will be used below when passing to the weak formulation.

%
%
%
%
%
%
%
%
%
%
%

\subsection{Passage to the limit in the weak formulations}\label{ssec:zeta-limit}

We now pass to the limit in the kinetic--fluid coupling and in the weak formulations of \eqref{eq:zeta-sys-n}. Throughout this subsection, we use the convergences obtained in Section \ref{ssec:zeta-est}.  

We first treat the mollified density. Since convolution is continuous from $L^1(\T^d)$ to $C^k(\T^d)$ for every fixed $k\ge0$, we have $\|(\rho_n-\rho)_\e\|_{C^k(\T^d)} \le \|\nabla^k\theta_\e\|_{L^\infty(\T^d)} \|\rho_n-\rho\|_{L^1(\T^d)}$. Consequently, for every fixed $k\ge0$ and every finite $q\ge1$,
\bq\label{eq:zeta-rhoeps-conv}
(\rho_n)_\e\to\rho_\e \quad \text{strongly in }L^q(0,T;C^k(\T^d)).
\eq
Moreover, since the total mass is normalized,
\[
(\rho_n)_\e(t,x) = \intt\theta_\e(x-y)\rho_n(t,y)\,\dy \ge c_\e\intt\rho_n(t,y)\,\dy = c_\e>0,
\]
where $c_\e>0$ depends only on the mollifier and on $\e$. The same lower bound holds for $\rho_\e$.

We next prove the strong convergence of the mollified momentum. Set $m_n:=\rho_nu_n$. The momentum equation in \eqref{eq:zeta-sys-n} gives, for every $\psi\in C^\infty(\T^d;\R^d)$,
\[
\lt\lal\pa_t m_n,\psi\rt\ral = \intt \rho_nu_n\otimes u_n:\nabla\psi\,\dx + \intt \rho_n^\gamma\nabla\cdot\psi\,\dx - \intt \mathbb S_{\zeta_n}(\rho_n,u_n):\nabla\psi\,\dx + \intt \rho_n(F_n)_\e\cdot\psi\,\dx.
\]
The first term is uniformly bounded in $L^\infty(0,T;W^{-1,1}(\T^d))$, due to
\[
\sup_n  \esssup_{0 \le t \le T} \intt\rho_n|u_n|^2\,\dx\le C_\e.
\]
The pressure term is uniformly bounded in $L^\infty(0,T;W^{-1,1}(\T^d))$, since
\[
\sup_n \esssup_{0 \le t \le T} \intt\rho_n^\gamma\,\dx\le C_\e.
\]
The stress term is uniformly bounded in $L^1(0,T;W^{-1,1}(\T^d))$. Indeed, using the viscous estimate and the uniform bound
\[
\sup_n  \esssup_{0 \le t \le T} \intt\nu_n(\rho_n)\,\dx\le C_\e,
\]
we have
\[
\int_0^T\intt|\mathbb S_{\zeta_n}(\rho_n,u_n)|\,\dx\dt \le C_\e.
\]
Finally, the force term is bounded in a negative norm by the drag estimate, as shown in \eqref{eq:tau-force-bound}. Hence, for some sufficiently large $M$, $\pa_t m_n$ is bounded in $L^1(0,T;W^{-M,1}(\T^d))$. After convolution in space, this implies that, for every fixed $k\ge0$, $\pa_t(m_n)_\e$ is bounded in $L^1(0,T;C^k(\T^d))$. On the other hand, $m_n$ is bounded in $L^\infty(0,T;L^{\frac{2\gamma}{\gamma+1}}(\T^d))$. Thus $(m_n)_\e$ is bounded in $L^\infty(0,T;C^k(\T^d))$. By the Aubin--Lions compactness lemma after convolution, and by the distributional convergence $m_n=\rho_nu_n\rightharpoonup\rho u$, we obtain, for every fixed $k\ge0$ and every finite $q\ge1$,
\bq\label{eq:zeta-momeps-conv}
(\rho_nu_n)_\e\to(\rho u)_\e \quad \text{strongly in }L^q(0,T;C^k(\T^d)).
\eq
Combining \eqref{eq:zeta-rhoeps-conv} and \eqref{eq:zeta-momeps-conv}, and using the lower bound $(\rho_n)_\e\ge c_\e>0$, we infer
\bq\label{eq:zeta-ueps-conv}
u_{n,\e}^\e = \frac{(\rho_nu_n)_\e}{(\rho_n)_\e} \to \frac{(\rho u)_\e}{\rho_\e} =:u_\e^\e \quad \text{strongly in } L^q(0,T;C^k(\T^d))
\eq
for every fixed $k\ge0$ and every finite $q\ge1$.

We now consider the kinetic component. The conservation of the kinetic mass and the uniform kinetic energy bound give
\[
\sup_n\esssup_{0\le t\le T} \inttr (1+|v|^2)f_n\,\dx\dv \le C_{\e,T}.
\]
Hence, using the density-like notation for measure-valued kinetic limits introduced above, we infer, up to a subsequence, that
\bq\label{eq:zeta-f-conv}
f_n\stackrel{*}{\rightharpoonup} f \quad\text{weakly-* in }L^\infty(0,T;\calM_+(\T^d\times\R^d)).
\eq
Moreover, the uniform second velocity moment gives velocity tightness:
\[
\sup_n\int_0^T\iint_{\T^d\times\{|v|>R\}}f_n\,\dx\dv\dt \le \frac{C_{\e,T}}{R^2}.
\]
As in the previous limiting steps, this tightness implies that the limiting kinetic component has mass one and finite second velocity moment, and that the velocity moments pass to the limit by truncation:
\[
\intr f_n\,\dv \rightharpoonup \intr f\,\dv, \quad \intr v f_n\,\dv \rightharpoonup \intr v f\,\dv \quad\text{in }\calD'((0,T)\times\T^d).
\]
 
Let $\varphi\in C_c^\infty([0,T)\times\T^d\times\R^d)$. The weak formulation of the kinetic equation is
\begin{align*}
&\int_0^T\inttr f_n(\pa_t\varphi+v\cdot\nabla_x\varphi)\,\dx\dv\dt + \inttr f_{0,n}\varphi(0)\,\dx\dv \\
&\quad = \int_0^T\inttr (\rho_n)_\e(v-u_{n,\e}^\e)f_n\cdot\nabla_v\varphi\,\dx\dv\dt.
\end{align*}
Then, by using the same arguments as in Section \ref{ssec:tau-limit}, we obtain that the limiting $f$ satisfies
\bq\label{eq:limit-kinetic}
\pa_t f+v\cdot\nabla_x f = \rho_\e\nabla_v\cdot\lt[(v-u_\e^\e)f\rt], \quad u_\e^\e=\frac{(\rho u)_\e}{\rho_\e}
\eq
in the sense of distributions.

It remains to identify the force term. Again, we observe that 
\[
\intt |F_n|\,\dx = \intt \lt|\intr(v-u_{n,\e}^\e)f_n\,\dv \rt| \dx \le C_\e\calD_{{\rm drag},n}(t)^{\frac12}.
\]
Consequently, for every fixed $k\ge0$, $\|(F_n)_\e\|_{L^2(0,T;W^{k,\infty}(\T^d))} \le C_{\e,k}$.   Using \eqref{eq:zeta-f-conv}, the velocity tightness, and the strong convergence \eqref{eq:zeta-ueps-conv}, the same truncation argument as above identifies the weak limit:
\[
F_n\rightharpoonup F:=\intr(v-u_\e^\e)f\,\dv \quad\text{in }\calD'((0,T)\times\T^d).
\]
Together with the uniform bound on $(F_n)_\e$, this gives, for every fixed $k\ge0$,
\[ 
(F_n)_\e \rightharpoonup F_\e \quad\text{weakly in }L^2(0,T;W^{k,\infty}(\T^d)).
\]

We also need the product with the fluid density. Since the approximate and limiting densities have unit mass, $\|\rho_n(t)-\rho(t)\|_{L^1(\T^d)}\le2$ for a.e. $t\in(0,T)$. Thus, the strong convergence in $L^1((0,T)\times\T^d)$ implies $\rho_n\to\rho$ strongly in $L^2(0,T;L^1(\T^d))$. Combining this with the weak convergence of $(F_n)_\e$, we get
\bq\label{eq:zeta-force-conv}
\rho_n(F_n)_\e \rightharpoonup \rho F_\e \quad \text{in }\calD'((0,T)\times\T^d).
\eq

We now pass to the limit in the fluid weak formulations. The continuity equation follows immediately from $\rho_n\to\rho$ strongly in $L^1((0,T)\times\T^d)$ and $\rho_nu_n\rightharpoonup\rho u$ in $\calD'((0,T)\times\T^d)$. Thus,
\[
\pa_t\rho+\nabla\cdot(\rho u)=0 \quad \text{in }\calD'((0,T)\times\T^d).
\]

For the momentum equation, let $\Psi\in C_c^\infty([0,T)\times\T^d;\R^d)$. The weak formulation of the $n$-th momentum equation is
\begin{align*}
&\int_0^T\intt \lt[ \rho_nu_n\cdot\pa_t\Psi + \rho_nu_n\otimes u_n:\nabla\Psi + \rho_n^\gamma\nabla\cdot\Psi \rt]\dx\dt \\
&\quad - \int_0^T\intt \mathbb S_{\zeta_n}(\rho_n,u_n):\nabla\Psi\,\dx\dt + \int_0^T\intt \rho_n(F_n)_\e\cdot\Psi\,\dx\dt + \intt\rho_{0,n}u_{0,n}\cdot\Psi(0)\,\dx = 0.
\end{align*}
Passing to the limit in the above by \eqref{eq:zeta-pressure-conv}, \eqref{eq:zeta-conv}, and \eqref{eq:zeta-force-conv}, and using the convergence of the initial momentum, we obtain
\begin{align}\label{eq:weak-momentum-final}
\begin{aligned}
&\int_0^T\intt \lt[ \rho u\cdot\pa_t\Psi + \rho u\otimes u:\nabla\Psi + \rho^\gamma\nabla\cdot\Psi \rt]\dx\dt \\
&\quad - \int_0^T\intt \mathbb S(\rho,u):\nabla\Psi\,\dx\dt + \int_0^T\intt \rho F_\e\cdot\Psi\,\dx\dt + \intt\rho_0u_0\cdot\Psi(0)\,\dx = 0.
\end{aligned}
\end{align}
Consequently, the limiting triple $(f,\rho,u)$ satisfies the kinetic equation, the continuity equation, and the momentum equation in the sense of distributions.
 
%
%
%
%
%
%
%
%
%
%
%

\subsection{Completion of the construction of weak solutions}

It remains to pass to the limit in the energy inequalities and to verify that the limiting triple $(f,\rho,u)$ obtained above is a finite-energy BD entropy weak solution in the sense of Definition \ref{def:weak-sol}.

We first pass to the limit in the free energy inequality. For the approximate solutions, we have, for almost every $s\in[0,T]$ and every $t\in[s,T]$,
\[
\calE_{{\rm free},n}(t) + \int_s^t\lt[ \calD_{{\rm vis},n}(\tau) + \calD_{{\rm drag},n}(\tau) \rt]\dd\tau \le \calE_{{\rm free},n}(s).
\]
The convergences established in Sections \ref{ssec:zeta-est} and \ref{ssec:zeta-limit} imply
\[
f_n\stackrel{*}{\rightharpoonup} f \quad\text{weakly-* in }L^\infty(0,T;\calM_+(\T^d\times\R^d)),
\]
with a uniform second velocity moment, and
\[
\rho_n\to\rho \quad\text{strongly in }L^1((0,T)\times\T^d) \quad\text{and almost everywhere}.
\]
Moreover,
\[
\rho_nu_n\rightharpoonup\rho u \quad\text{in }\calD'((0,T)\times\T^d).
\]
By weak lower semicontinuity of nonnegative measures with second moments, we have
\[
\frac12\inttr |v|^2f(t)\,\dx\dv \le \liminf_{n\to\infty} \frac12\inttr |v|^2f_n(t)\,\dx\dv
\]
for almost every $t\in[0,T]$. Similarly, the fluid kinetic energy is lower semicontinuous:
\[
\frac12\intt \rho|u|^2\,\dx \le \liminf_{n\to\infty} \frac12\intt \rho_n|u_n|^2\,\dx.
\]
The pressure potential term converges strongly, since $\rho_n^\gamma\to\rho^\gamma$ strongly in $L^1((0,T)\times\T^d)$. Thus,
\[
\calE_{\rm free}(t) \le \liminf_{n\to\infty}\calE_{{\rm free},n}(t) \quad \text{for a.e. } t,
\]
where
\[
\calE_{\rm free}(t) = \frac12\inttr |v|^2f\,\dx\dv + \intt\lt[ \frac12\rho|u|^2+\pi(\rho) \rt]\dx.
\]

The viscous dissipation is lower semicontinuous as well. Indeed, by by Lemma \ref{lem:BD-convec-comp}, up to a subsequence,
\[
\rho_n^{\frac{\alpha}{2}}\nabla u_n \rightharpoonup \rho^{\frac{\alpha}{2}}\nabla u \quad \text{weakly in }L^2((0,T)\times\T^d).
\]

In one space dimension, we have
\[
\calD_{{\rm vis},n}(t) = 2\int_{\T} \lt( \alpha\rho_n^\alpha+\zeta_n\rho_n \rt) |\pa_xu_n|^2\,\dx.
\]
Since the term involving $\zeta_n$ is nonnegative, weak lower semicontinuity gives
\[
2\alpha \int_s^t\int_{\T} \rho^\alpha|\pa_xu|^2\,\dx\dd\tau \le \liminf_{n\to\infty} \int_s^t \calD_{{\rm vis},n}(\tau)\,\dd\tau.
\]

In the two-dimensional case, where $\alpha=\frac12$, we get
\[
\calD_{{\rm vis},n}(t) = 2\int_{\T^2} \sqrt{\rho_n}|D^0u_n|^2\,\dx + 2\zeta_n \int_{\T^2} \rho_n|D u_n|^2\,\dx.
\]
Moreover,
\[
\rho_n^{\frac14}D^0u_n \rightharpoonup \rho^{\frac14}D^0u \quad\text{weakly in } L^2((0,T)\times\T^2).
\]
Since the second term is nonnegative, it follows that
\[
2
\int_s^t\int_{\T^2} \sqrt{\rho}|D^0u|^2\,\dx\dd\tau \le \liminf_{n\to\infty} \int_s^t \calD_{{\rm vis},n}(\tau)\,\dd\tau.
\]

Consequently, in either admissible case,
\[
\rho^{\frac{\alpha}{2}}\nabla u \in L^2((0,T)\times\T^d),
\]
and
\[
\int_s^t\calD_{\rm vis}(\tau)\,\dd\tau \le \liminf_{n\to\infty} \int_s^t\calD_{{\rm vis},n}(\tau)\,\dd\tau,
\]
where
\[
\calD_{\rm vis}(t) = \intt \lt[ 2\rho^\alpha|D u|^2 + 2(\alpha-1)\rho^\alpha |\nabla\cdot u|^2 \rt]\dx.
\]

\begin{lemma}\label{lem:time-slice-f-energy}
After extracting a subsequence, there exists a set $\calS\subset(0,T)$ of full measure such that
\[
\calE_{{\rm free},n}(s) \to \calE_{\rm free}(s) \quad\text{for every }s\in\calS.
\]
\end{lemma}
\begin{proof}
We first consider the fluid kinetic energy. Set
\[
g_n:=\sqrt{\rho_n}u_n, \quad g:=\sqrt{\rho}u.
\]
The energy bound and the strong convergence of the density give
\[
g_n\rightharpoonup g \quad\text{weakly in } L^2((0,T)\times\T^d).
\]
On the other hand, by the convergence of the convective tensors in \eqref{eq:zeta-conv}, for every nonnegative $\eta\in C_c^\infty(0,T)$,
\[
\int_0^T\eta(t)\intt\rho_n|u_n|^2\,\dx\dt \to \int_0^T\eta(t)\intt\rho|u|^2\,\dx\dt.
\]
Consequently,
\[
\sqrt{\eta}\,g_n\to\sqrt{\eta}\,g \quad\text{strongly in } L^2((0,T)\times\T^d).
\]
It follows, after a diagonal extraction, that
\[
\sqrt{\rho_n(s)}u_n(s) \to \sqrt{\rho(s)}u(s) \quad\text{strongly in }L^2(\T^d)
\]
for almost every $s\in(0,T)$.

Moreover, the strong convergence $\rho_n^\gamma\to\rho^\gamma$ in $L^1((0,T)\times\T^d)$ and $\rho_n\to\rho$ in $L^1((0,T)\times\T^d)$ imply
\[
\pi(\rho_n)\to\pi(\rho) \quad\text{strongly in } L^1((0,T)\times\T^d).
\]
Hence, up to a further subsequence,
\[
\intt\pi(\rho_n(s))\,\dx \to \intt\pi(\rho(s))\,\dx
\]
for almost every $s\in(0,T)$.

It remains to treat the kinetic energy. Set
\[
a_n:=(\rho_n)_\e, \quad b_n:=(\rho_nu_n)_\e.
\]
The kinetic equation can be written as
\[
\partial_t f_n + \nabla_{x,v}\cdot \lt[ \bigl(v,-a_nv+b_n\bigr)f_n \rt] =0.
\]
The mass and energy bounds imply
\[
0<c_\e\leq a_n\leq C_\e, \quad \|b_n\|_{L^\infty(0,T;W^{k,\infty}(\T^d))} \leq C_{\e,k}
\]
for every fixed $k\geq0$. Let $(X_n,V_n)$ denote the corresponding characteristic flow. Then
\[
|V_n(t)|\leq |v|+C_{\e,T}.
\]
In particular, for $R>C_{\e,T}$,
\[
|V_n(t;x,v)|>R \quad\Longrightarrow\quad |v|>R-C_{\e,T}.
\]
Since $f_n(t)$ is the push-forward of $f_{0,n}$ under the
characteristic flow, we obtain
\begin{align*}
\iint_{\T^d\times\{|v|>R\}} |v|^2 f_n(t,x,v)\,\dx\dv &\leq 2\iint_{\T^d\times\{|v|>R-C_{\e,T}\}}
|v|^2 f_{0,n}(x,v)\,\dx\dv \\
&\quad + 2C_{\e,T}^2 \iint_{\T^d\times\{|v|>R-C_{\e,T}\}} f_{0,n}(x,v)\,\dx\dv.
\end{align*}
Since $f_{0,n}\to f_0$ in $W_2$, the sequence $\{|v|^2f_{0,n}\}_n$ is uniformly integrable. Consequently,
\bq\label{eq:uniform-kinetic-tail}
\lim_{R\to\infty} \sup_n\sup_{0\leq t\leq T} \iint_{\T^d\times\{|v|>R\}} |v|^2 f_n(t,x,v)\,\dx\dv
=0.
\eq

The kinetic equation also gives time equicontinuity against compactly supported $C^1$ test functions. Indeed, for every $\varphi\in C_c^1(\T^d\times\R^d)$,
\[
\frac{\rd}{\dt}\inttr \varphi(x,v)f_n(t,x,v)\,\dx\dv = \inttr \lt[ v\cdot\nabla_x\varphi + (-a_n v+b_n)\cdot\nabla_v\varphi \rt] f_n\,\dx\dv.
\]
The right-hand side is bounded uniformly in $n$ and $t$. Together
with the uniform moment bounds, this yields, after extracting a
subsequence,
\[
f_n \to f \quad\text{in }  C\lt( [0,T]; \calP(\T^d\times\R^d) \rt)
\]
with respect to the narrow topology.

We now prove the convergence of the second velocity moments. Fix $t\in[0,T]$, and let $\chi_R\in C_c^\infty(\R^d)$ satisfy ${\bf 1}_{B(0,R)}\le\chi_R\le{\bf 1}_{B(0,2R)}$.
Since $v\mapsto\chi_R(v)|v|^2$ is bounded and continuous, the narrow
convergence gives
\[
\inttr \chi_R(v)|v|^2f_n(t,x,v)\,\dx\dv \to \inttr \chi_R(v)|v|^2f(t,x,v)\,\dx\dv.
\]
Moreover, the uniform tail estimate \eqref{eq:uniform-kinetic-tail} is inherited by the limiting measure:
\[
\lim_{R\to\infty} \sup_{0\leq t\leq T} \int_{\T^d\times\{|v|>R\}} |v|^2f(t,x,v)\,\dx\dv =0.
\]
Therefore,
\begin{align*}
& \lt| \inttr |v|^2f_n(t,x,v)\,\dx\dv - \inttr |v|^2f(t,x,v)\,\dx\dv \rt| \\
&\quad\leq \lt| \inttr \chi_R(v)|v|^2 \bigl(f_n(t,x,v)-f(t,x,v)\bigr)\,\dx\dv \rt|
+ \int_{\T^d\times\{|v|>R\}} |v|^2f_n(t,x,v)\,\dx\dv \\
&\quad \quad + \int_{\T^d\times\{|v|>R\}} |v|^2f(t,x,v)\,\dx\dv.
\end{align*}
First letting $n\to\infty$ and then $R\to\infty$, we conclude that
\[
\inttr |v|^2f_n(t,x,v)\,\dx\dv \to \inttr |v|^2f(t,x,v)\,\dx\dv
\]
for every $t\in[0,T]$.
 
Combining the preceding three convergences completes the proof.
\end{proof}

The drag dissipation is also lower semicontinuous. Set
\[
A_n(t,x,v):=(\rho_n)_\e(t,x)|v-u_{n,\e}^\e(t,x)|^2, \quad A(t,x,v):=\rho_\e(t,x)|v-u_\e^\e(t,x)|^2.
\]
Since $(\rho_n)_\e\to\rho_\e$ and
$u_{n,\e}^\e\to u_\e^\e$ strongly in
$L^q(0,T;C^k(\T^d))$ for every finite $q\ge1$ and every fixed $k\ge0$,
we have local convergence of $A_n$ to $A$ in the phase variables, with at most
quadratic growth in $v$.

Let $\chi_R\in C_c^\infty(\R^d)$ be such that ${\bf 1}_{B(0,R)}\le\chi_R\le{\bf 1}_{B(0,2R)}$. For each fixed $R>0$, the functions $\chi_R A_n$ converge strongly to $\chi_R A$ in $L^1(0,T;C_0(\T^d\times\R^d))$. Thus, using the weak-* convergence of $f_n$ and the uniform kinetic mass bound,
\[
\int_s^t\inttr \chi_R A f\,\dx\dv\dd\tau = \lim_{n\to\infty} \int_s^t\inttr \chi_R A_n f_n\,\dx\dv\dd\tau.
\]
Since $0\le\chi_R\le1$, we have
\[
\int_s^t\inttr \chi_R A_n f_n\,\dx\dv\dd\tau \le \int_s^t\calD_{{\rm drag},n}(\tau)\,\dd\tau.
\]
Hence
\[
\int_s^t\inttr \chi_R A f\,\dx\dv\dd\tau \le \liminf_{n\to\infty} \int_s^t\calD_{{\rm drag},n}(\tau)\,\dd\tau.
\]
Letting $R\to\infty$ and using monotone convergence, we obtain
\[
\int_s^t\calD_{\rm drag}(\tau)\,\dd\tau \le \liminf_{n\to\infty} \int_s^t\calD_{{\rm drag},n}(\tau)\,\dd\tau, 
\]
where, in the diffusionless case,
\[
\calD_{\rm drag}(t) = \inttr \rho_\e |v-u_\e^\e|^2 f\,\dx\dv.
\]

Let $\calS_0\subset(0,T)$ be the intersection of the full-measure set provided by Lemma \ref{lem:time-slice-f-energy} with the set of admissible initial times for the approximate free energy inequalities. Then, for every $s\in\calS_0$ and every $t\in[s,T]$, the preceding lower semicontinuity estimates yield
\begin{align*}
\calE_{\rm free}(t) +\int_s^t \lt[ \calD_{\rm vis}(\tau)+\calD_{\rm drag}(\tau) \rt]\dd\tau &\leq \liminf_{n\to\infty} \lt\{ \calE_{{\rm free},n}(t) +\int_s^t \lt[ \calD_{{\rm vis},n}(\tau) +\calD_{{\rm drag},n}(\tau) \rt]\dd\tau \rt\} \\
&\leq \lim_{n\to\infty}\calE_{{\rm free},n}(s) = \calE_{\rm free}(s).
\end{align*}

Thus
\bq\label{eq:limit-energy-two-time}
\calE_{\rm free}(t) + \int_s^t\lt[ \calD_{\rm vis}(\tau) + \calD_{\rm drag}(\tau) \rt]\dd\tau \le \calE_{\rm free}(s)
\eq
for almost every $s\in[0,T]$ and every $t\in[s,T]$.

We next pass to the limit in the BD entropy estimate. Recall that $\phi'(\rho)=2\alpha\rho^{\alpha-2}$. The limiting BD energy is the one introduced in the Introduction. In the present diffusionless case $\sigma=0$, it reads
\[
\calE_{\rm BD}(t) = \frac12\inttr |v|^2f\,\dx\dv + \intt\lt[ \frac12\rho|u+\nabla\phi(\rho)|^2+\pi(\rho) \rt]\dx.
\]

 We next identify the limiting effective velocity. From the almost everywhere convergence of $\rho_n$ and $u_n$, together with the uniform kinetic-energy bound, we have
\[
\sqrt{\rho_n}u_n \rightharpoonup \sqrt{\rho}\,u \quad\text{weakly in } L^2((0,T)\times\T^d).
\]

We now consider the drift part. If $\alpha<\frac12$, then $d=1$,
and
\[
\sqrt{\rho_n}\nabla\phi_{\zeta_n}(\rho_n) = \frac{2\alpha}{\alpha-\frac12} \nabla\rho_n^{\alpha-\frac12} + 4\zeta_n\nabla\sqrt{\rho_n}.
\]
The uniform density bounds and the singular BD estimate imply
\[
\nabla\rho_n^{\alpha-\frac12} \rightharpoonup \nabla\rho^{\alpha-\frac12} \quad\text{weakly in }L^2((0,T)\times\T).
\]
Moreover,
\[
4\zeta_n\nabla\sqrt{\rho_n} = \frac{2\zeta_n}{\alpha} \rho_n^{\frac12-\alpha}\nabla\rho_n^\alpha \to0 \quad\text{strongly in }L^2((0,T)\times\T).
\]
Hence
\[
\sqrt{\rho_n}\nabla\phi_{\zeta_n}(\rho_n) \rightharpoonup \sqrt{\rho}\nabla\phi(\rho) \quad\text{weakly in }L^2((0,T)\times\T).
\]

If $\alpha=\frac12$, then
\[
\sqrt{\rho_n}\nabla\phi_{\zeta_n}(\rho_n) = \nabla\log\rho_n + 4\zeta_n\nabla\sqrt{\rho_n}.
\]
The density consequences imply
\[
\log\rho_n\to\log\rho \quad\text{strongly in every finite }L^p,
\]
while the singular BD estimate gives the weak compactness of $\nabla\log\rho_n$ in $L^2$. Therefore,
\[
\nabla\log\rho_n \rightharpoonup \nabla\log\rho \quad\text{weakly in }L^2.
\]
Furthermore,
\[
4\zeta_n\nabla\sqrt{\rho_n} = 4\zeta_n\nabla\rho_n^{1/2} \to0 \quad\text{strongly in }L^2.
\]
Thus, also in this case,
\[
\sqrt{\rho_n}\nabla\phi_{\zeta_n}(\rho_n) \rightharpoonup \sqrt{\rho}\nabla\phi(\rho) \quad\text{weakly in }L^2.
\]

Consequently,
\[
\sqrt{\rho_n} \lt( u_n+\nabla\phi_{\zeta_n}(\rho_n) \rt) \rightharpoonup \sqrt{\rho} \lt( u+\nabla\phi(\rho) \rt) \quad\text{weakly in } L^2((0,T)\times\T^d).
\]
The weak lower semicontinuity of the $L^2$ norm therefore gives
\[
\calE_{\rm BD}(t) \le \liminf_{n\to\infty} \calE_{{\rm BD},n}(t)
\]
at every time slice at which the corresponding weak convergences hold.

The limiting BD density dissipation is
\[
\calD_{\rm BD}(t) := \intt\nabla\rho^\gamma\cdot\nabla\phi(\rho)\,\dx.
\]
Since $\gamma=\alpha+1$, we have
\[
\calD_{\rm BD}(t) = \frac{2(\alpha+1)}{\alpha}\intt|\nabla\rho^\alpha|^2\,\dx.
\]
This term is lower semicontinuous by the weak convergence of $\nabla\rho_n^\alpha$ in $L^2((0,T)\times\T^d)$. The rotational dissipation is
\[
\calD_{\rm rot}(t) := 2\intt\rho^\alpha|Au|^2\,\dx,
\]
and it is lower semicontinuous by the same weighted weak compactness used for the viscous dissipation.

Passing to the limit in the integrated form of the combined estimate \eqref{eq:zeta-combined}, we obtain
\bq\label{eq:limit-combined-two-time}
\calE_{\rm free}(t) +\eta\calE_{\rm BD}(t) +\int_{0}^t\lt[ \calD_{\rm vis}(\tau) +\frac12\calD_{\rm drag}(\tau) +\frac{\eta}{2}\calD_{\rm BD}(\tau) +\eta\calD_{\rm rot}(\tau) \rt]\dd\tau \le \calE_{\rm free}(0) +\eta\calE_{\rm BD}(0)
\eq
for almost every $t\in[0,T]$.

Consequently,
\[
\rho^\alpha\in L^2(0,T;H^1(\T^d)), \quad \rho^{\frac{\alpha}{2}}\nabla u\in L^2((0,T)\times\T^d),
\quad \text{and} \quad 
\sqrt{\rho}\lt(u+\nabla\phi(\rho)\rt)\in L^\infty(0,T;L^2(\T^d)).
\]
Moreover, the kinetic estimates inherited from the approximation give 
\[
f\in L^\infty(0,T;\calP_2(\T^d\times\R^d)).
\]

We now verify the equations in Definition \ref{def:weak-sol}. The kinetic equation was obtained in \eqref{eq:limit-kinetic}. The continuity equation follows from the strong convergence of the density and the weak convergence of the momentum. The momentum equation was obtained in \eqref{eq:weak-momentum-final}. Hence, $(f,\rho,u)$ satisfies the kinetic equation, the continuity equation, and the momentum equation in the sense of distributions.

The mass constraints are preserved by the limiting procedure. Indeed, the approximate solutions satisfy
\[
\inttr f_n(t,x,v)\,\dx\dv=1, \quad \intt\rho_n(t,x)\,\dx=1
\]
for every $t\ge0$. Passing to the limit gives
\[
\inttr f(t,x,v)\,\dx\dv=1, \quad \intt\rho(t,x)\,\dx=1
\]
for almost every $t\ge0$.

The free energy inequality \eqref{eq:limit-energy-two-time} gives the finite-energy inequality required in Definition \ref{def:weak-sol}. The combined inequality \eqref{eq:limit-combined-two-time} gives the BD entropy condition in the stable form used in Definition \ref{def:weak-sol}. Thus all conditions in Definition \ref{def:weak-sol} are satisfied on $[0,T]$. Since $T>0$ was arbitrary, a diagonal argument yields a global finite-energy BD entropy weak solution on $[0,\infty)$. This proves the existence part of Theorem \ref{thm:weak-exist}  in the case $\sigma=0$.

%
%
%
%
%
%
%
%
%
%
%
\section{Relaxation of entropy weak solutions and the diffusive case}\label{sec:weak-relax-diff}

This section completes the proof of Theorem \ref{thm:weak-exist}. We first prove the exponential relaxation asserted in Theorem \ref{thm:weak-exist} for the diffusionless system $\sigma=0$. We then explain how the construction extends to the diffusive case $\sigma>0$, and  use the additional hypoelliptic regularity to complete the corresponding hypocoercive relaxation argument.

%
%
%
%
%
%
%
%
%
%
%

\subsection{Exponential relaxation of entropy weak solutions}
\label{ssec:weak-relax}

We work throughout this subsection with the diffusionless system $\sigma=0$. The proof is based on the smooth relaxation argument of Section \ref{sec:energy}, combined with the low-dimensional BD density estimates of Section \ref{sec:bd-estimate}. Since the entropy weak solution does not have sufficient temporal regularity to differentiate the compensating functional pointwise in time, the argument is formulated instead in an integrated form.

We first verify that the density assumptions used in Theorem \ref{thm:main1} are available for the limiting weak solution. By the limiting combined BD entropy inequality \eqref{eq:limit-combined-two-time}, we have
\[
\esssup_{t\ge0} \intt \rho|\nabla\phi(\rho)|^2\,\dx <\infty .
\]
Thus, arguing exactly as in Proposition \ref{prop:veri-w-den-c}, we obtain
\[
\esssup_{t\ge0} \intt\lt(\rho(t,x)^p+\rho(t,x)^{-s}\rt)\dx <\infty
\]
for every finite $p,s>0$ in the admissible low-dimensional cases
\[
\begin{cases}
d=1,\quad 0<\alpha\le\frac12,\quad \gamma=\alpha+1,\\[1mm]
d=2,\quad \alpha=\frac12,\quad \gamma=\frac32.
\end{cases}
\]
Moreover,
\[
\esssup_{t\ge0} [\rho(t)^\alpha]_{\calA_2} <\infty .
\]
Indeed, in the case $d=1$, the BD bound gives uniform upper and lower bounds on $\rho$. In the borderline case $d=2$, $\alpha=\frac12$, the estimate
\[
\esssup_{t\ge0} \|\nabla\log\rho(t)\|_{L^2(\T^2)} <\infty
\]
together with the Moser--Trudinger argument used in Proposition \ref{prop:veri-w-den-c} yields the desired $\calA_2$ bound.

We next introduce the modulated energy at the weak-solution level. Since $\sigma=0$ in this subsection, the centered kinetic functional $\calH^{\bar v}$ is simply the kinetic fluctuation energy. Set
\[
\bar v (t):=\inttr v f(t,x,v)\,\dx\dv, \quad \bar m (t):=\intt\rho(t,x)u(t,x)\,\dx.
\]
These quantities are well defined for almost every $t\ge0$ by the finite kinetic and fluid energies. Moreover, the total momentum is conserved:
\[
\bar v (t)+\bar m (t)=\bar v (0)+\bar m (0)
\]
for almost every $t\ge0$. This follows by testing the weak fluid momentum equation with constant vector fields and the kinetic equation with velocity cutoffs $v\chi_R(v)$, and then passing to the limit $R\to\infty$ using the finite second velocity moment.

Define
\[
\calE(t) := \frac12\inttr |v-\bar v (t)|^2 f\,\dx\dv + \frac12\intt \rho|u-\bar m (t)|^2\,\dx + \intt \pi(\rho)\,\dx + \frac14|\bar v (t)-\bar m (t)|^2 .
\]
Then the free energy inequality implies
\bq\label{eq:weak-E0}
\calE(t) + \int_s^t\calD_0(\tau)\,\dd\tau \le \calE(s)
\eq
for almost every $s\ge0$ and every $t\ge s$, where
\[
\calD_0(t):=\calD_{\rm drag}(t)+\calD_{\rm vis}(t).
\]
Indeed, by the conservation of $\bar v +\bar m $, we have the algebraic identity
\[\begin{aligned}
&\frac12\inttr |v-\bar v |^2 f\,\dx\dv + \frac12\intt \rho|u-\bar m |^2\,\dx + \frac14|\bar v -\bar m |^2 \\
&\quad = \frac12\inttr |v|^2 f\,\dx\dv + \frac12\intt \rho|u|^2\,\dx - \frac14|\bar v +\bar m |^2 .
\end{aligned}\]
Thus \eqref{eq:weak-E0} follows from the free energy inequality.

By the density bounds verified above, the coercivity estimate from Section
\ref{ssec:coer-ex-den} applies for almost every time. Hence
\bq\label{eq:weak-coerc}
\calE(t) \le C_\e\lt[ \calD_0(t)+ \intt \pi_\aux(\rho(t)) \,\dx  \rt],
\eq
where $\pi_\aux(r)=(r^\gamma-1)(r-1)$.

It remains to justify the compensating functional estimate in the weak setting. Let $\Phi(t,\cdot)$ be the unique zero-average solution of
\[
-\Delta\Phi=\rho-1, \quad \intt\Phi\,\dx=0,
\]
and define
\[
\calX(t):= \intt \rho\bigl(u-\bar m (t)\bigr)\cdot\nabla\Phi\,\dx .
\]
This quantity is well defined for almost every $t\ge0$, and the same estimate as in the smooth case gives
\bq\label{eq:weak-X-bound}
|\calX(t)|\le C\calE(t).
\eq

\begin{lemma} 
There exist constants $c_0>0$ and $C_0>0$ such that
\bq\label{eq:weak-comp}
\calX(s)-\calX(t) \ge c_0\int_s^t\intt \pi_\aux(\rho(\tau)) \,\dx \dd\tau - C_0\int_s^t\calD_0(\tau)\,\dd\tau
\eq
for almost every $s\ge0$ and every $t\ge s$.
\end{lemma}

\begin{proof}
For smooth solutions, \eqref{eq:weak-comp} is exactly the integrated form of the compensating estimate proved in Section \ref{ssec:comp-func}. We justify the same identity for the entropy weak solution by a standard Steklov averaging argument in time, combined with a spatial regularization of the elliptic test function.

For $h>0$, set
\[
g^h(t):=\frac1h\int_t^{t+h}g(\tau)\,\dd\tau,
\]
and let $\Phi^h$ be the zero-average solution of
\[
-\Delta\Phi^h=\rho^h-1 .
\]
After spatial regularization, $\nabla\Phi^h$ is an admissible test function in the Steklov-averaged momentum equation. Combining this identity with the Steklov-averaged continuity equation gives the integrated analogue of the formal relation
\[
-\frac{\dd}{\dt}\calX(t) = \intt \pi_\aux(\rho) \,\dx + \text{lower-order terms}.
\]
The pressure term gives
\[
\int_s^t \intt \pi_\aux(\rho(\tau)) \,\dx \dd\tau .
\]
All remaining terms are estimated as in Section \ref{ssec:comp-func}. The bounds obtained above, the lower bound $\rho_\e\ge c_\e>0$, and the energy inequality ensure that each term is integrable on finite time intervals. More precisely, the lower-order terms are bounded by
\[
\frac{c_0}{2} \int_s^t\intt \pi_\aux(\rho(\tau)) \,\dx \dd\tau + C_0 \int_s^t\calD_0(\tau)\,\dd\tau .
\]
Passing first with the spatial regularization parameter to zero and then letting $h\to0$, we obtain \eqref{eq:weak-comp}.
\end{proof}

We now conclude the proof. Let $\delta>0$ be sufficiently small and define
\[
\calE^\delta(t):=\calE(t)+\delta\calX(t).
\]
By \eqref{eq:weak-X-bound}, we have $\calE^\delta(t)\sim\calE(t)$ for sufficiently small $\delta>0$. Adding \eqref{eq:weak-E0} and $\delta$ times \eqref{eq:weak-comp}, and then choosing $\delta>0$ sufficiently small, we obtain
\[
\calE^\delta(t) + c\int_s^t \lt[ \calD_0(\tau)+\intt \pi_\aux(\rho(\tau)) \,\dx \rt]\dd\tau \le \calE^\delta(s)
\]
for almost every $s\ge0$ and every $t\ge s$. Using the coercivity estimate \eqref{eq:weak-coerc}, we infer that
\[
\calE^\delta(t) + \lambda\int_s^t\calE^\delta(\tau)\,\dd\tau \le \calE^\delta(s)
\]
for some $\lambda>0$. Applying the integral form of Gronwall's inequality gives
\[
\calE(t)\le C e^{-\lambda t}\calE(0)
\]
for almost every $t\ge0$.

Finally, choosing $p>2$ in the density bounds verified above and interpolating between $L^1(\T^d)$ and $L^p(\T^d)$ yields the $L^2$ decay of the density fluctuation. Consequently,
\[
\|v-\bar v (t)\|_{L^2(f(t))} + \|u(t,x)-\bar m (t)\|_{L^2(\rho(t))} + \|\rho(t)-1\|_{L^2} + |\bar v (t)-\bar m (t)| \le C e^{-\frac\lambda2 t}
\]
for almost every $t\ge0$. This proves the relaxation part of Theorem \ref{thm:weak-exist} in the case $\sigma=0$.

%
%
%
%
%
%
%
%
%
%
%

\subsection{Extension to  the  diffusive system $\sigma >0$}\label{ssec:diff-ext}
We have so far carried out the construction and the relaxation argument for the diffusionless system $\sigma=0$. This is the case in which the weakest kinetic compactness naturally gives a measure-valued kinetic component,
\[
f\in L^\infty_{\rm loc}(0,\infty;\calP_2(\T^d\times\R^d)).
\]
We now complete the proof of Theorem \ref{thm:weak-exist} in the diffusive case $\sigma>0$. We first describe the modifications required in the weak-solution construction and then use the additional hypoelliptic regularity to justify the hypocoercive relaxation argument.

When $\sigma>0$, the same approximation scheme can be adapted with only minor changes in the kinetic part. In this case,  the kinetic part of the free energy
is the entropy $H$, and one assumes the corresponding finite-entropy initial regularity
\[
f_0\in L^1_2 \cap L\log L(\T^d\times\R^d).
\]
The kinetic part of the initial-data approximation can be modified so that the smooth approximations are nonnegative, have unit mass, and converge to $f_0$ in weighted $L^1$ together with the entropy. If, in addition, $f_0\in L^2_q(\T^d\times\R^d)$, they may be chosen uniformly bounded in $L^2_q$.

The kinetic entropy estimate then yields
\[
f\in L^\infty_{\rm loc} \lt(0,\infty;L^1_2 \cap L\log L(\T^d\times\R^d)\rt).
\]
The weak formulation of the kinetic equation is then modified by adding the corresponding velocity-diffusion term.

Apart from this modification, the fluid part of the argument is unchanged. Indeed, the free energy estimate and the BD entropy estimate retain the same fluid coercive terms; the kinetic diffusion contributes only the standard
nonnegative entropy dissipation on the kinetic side. The compactness of $\rho$, the identification of the convective term, the passage to the limit in the density-dependent viscosity tensor, and the treatment of the
kinetic--fluid drag term are therefore identical to the diffusionless case, or simpler due to the additional kinetic integrability. Thus the existence result extends to $\sigma>0$ in the natural finite-entropy class, with the
kinetic component belonging to $L^\infty_{\rm loc}(0,\infty;L^1_2\cap L\log L(\T^d\times\R^d))$ instead of $L^\infty_{\rm loc}(0,\infty;\calP_2(\T^d\times\R^d))$.

 The hypocoercivity argument, however, requires more regularity, since it is formulated at the level of Fisher information and higher-order mixed derivatives. This regularity follows from the hypoelliptic smoothing of kinetic Fokker--Planck equations; see \cite{Villani} for the general framework and \cite{Shv-CPAM} for a result closer to the present setting. We recall a regularity result in a form adapted to our setting.

\def \bk {{\bf k}}
\def \bl {{\bf l}}
\def \rmA {{\rm A}}
\def \rmb {{\rm b}}

 For $m,q\in\N$, we define the weighted Sobolev space
\begin{equation}\label{e:Sobdef}
H^{m}_q(\T^d \times \R^d) =  \lt\{ f :  \sum_{ 2|\bk| + |\bl | \leq 2m}   \inttr  \jap{v}^{q - 2|\bk| - | \bl |  } | \pa^{\bk}_{x} \pa_v^{\bl} f |^2 \dv\dx <\infty \rt\}.
\end{equation}
Lower-order derivatives are equipped with stronger velocity weights.
 
Consider the following linear non-homogeneous Fokker-Planck equation
\begin{equation}\label{e:FPgen}
\pa_t f + v \cdot \nabla_x f = \nabla_v \cdot (\rmA \nabla_v f) + \nabla_v \cdot (\rmb f).
\end{equation}
Here, $\rmA = \rmA(x,v,t)\in \R^{d \times d}$ is a given matrix, and $\rmb = \rmb(x,v,t) \in \R^d$ is a field satisfying
\begin{equation}\label{e:Ab1}
\lambda \I \leq \rmA(x,v,t) \leq \Lambda \I,  \quad (x,v,t) \in \T^d \times \R^d \times [0,T)
\end{equation}
and for any multi-indices $\bk,\bl \geq 0$,
\begin{equation}\label{e:Ab2}
\| \pa^{\bk}_{x} \pa^{\bl}_v \rmA\|_{L^\infty} < \infty, \quad   \| \jap{v}^{-1} \pa^{\bk}_{x}  \rmb\|_{L^\infty} + \| \pa^{\bk}_{x} \pa^{\bl + 1}_v \rmb\|_{L^\infty}  <\infty.
\end{equation}

\begin{proposition}[\cite{Shv-CPAM}]\label{p:reg}
Let $f\in L^\infty([0,T); L^1\cap L^2) \cap C([0,T); \calD')$ be a weak solution to \eqref{e:FPgen} satisfying \eqref{e:Ab1} - \eqref{e:Ab2}, and with initial condition in class $f_0 \in L^2_q$, for some $q \geq 0$. Then, $f \in L^\infty([0,T); L^2_q)$, and for any $m \in \N$, there exist  constants $C_{m,T},\kappa>0$ such that 
\begin{equation}\label{e:reg}
\|f(t) \|_{H^m_q} + \|\pa_t f(t)\|_{H^{m-1}_{q-3}}\leq \frac{C_{m,T}}{t^\kappa}, \quad t <T.
\end{equation}
\end{proposition}

In the present system, the kinetic equation can be written in the form \eqref{e:FPgen} with
\[
\rmA=\sigma\rho_\e \I, \quad \rmb=\rho_\e(v-u_\e^\e).
\]
Since $\rho_\e$ is bounded above and below by positive constants depending on $\e$, the uniform ellipticity condition \eqref{e:Ab1} holds. Moreover, the spatial mollification in the definitions of $\rho_\e$ and $u_\e^\e$ gives the required bounds on the $x$-derivatives of $\rmA$ and $\rmb$, while the dependence of $\rmb$ on $v$ is only linear. Thus the hypotheses of Proposition \ref{p:reg} are satisfied on every finite time interval, with constants depending on $\e$ and on the corresponding energy bounds.

If, in addition, $f_0\in L^2_q$, the standard weighted $L^2$ estimate, first obtained at the approximation level and then passed to the limit, gives
\[
f\in L^\infty(0,T;L^2_q(\T^d\times\R^d))
\]
for every $T>0$. The weak kinetic equation also gives $f\in C([0,T];\calD')$ after choosing the standard representative. Thus, all the hypotheses of Proposition \ref{p:reg} are satisfied.

Consequently, if $f_0\in L^2_q$ with $q\ge d+4$, then Proposition \ref{p:reg} yields, for every $0< t_0 < T < \infty$, the Fisher information and higher-order regularity required to make the hypocoercive estimates of Section \ref{s:hypo} classical on every interval $[t_0, T]$, see also \cite{TV2000} for related estimates controlling Fisher-type quantities by weighted Sobolev norms. Since $T>t_0$ is arbitrary and the constants in the hypocoercive estimate depend only on the uniform BD density bounds, the resulting decay estimate holds on $[t_0,\infty)$.  We emphasize that the fluid variables remain at the entropy weak-solution level. Accordingly, the kinetic Fisher-information estimate is combined with the integrated weak energy and compensating-functional argument of Section \ref{ssec:cons}. The latter argument is unchanged when $\sigma>0$, since the velocity diffusion contributes an additional nonnegative term to the kinetic entropy dissipation. Combining the kinetic Fisher-information estimate with this integrated weak argument and the uniform BD density bounds yields \eqref{eq:cond-rel-pos},with $\|\rho-1\|_{L^1(\T^d)}^2$ replaced by $\|\rho-1\|_{L^2(\T^d)}^2$. The estimates on the initial interval $[0,t_0]$ are absorbed into the constant $C$. This proves the existence and relaxation assertions of Theorem \ref{thm:weak-exist} in the case $\sigma>0$ and completes the proof.

%
%
%
%
%
%
%
%
%
%

\section*{Acknowledgments}
 The work of Y.-P. Choi was supported by NRF grant no. 2022R1A2C1002820 and no. RS-2024-00406821.  	The work of R. Shvydkoy was  supported in part by NSF
	grant  DMS-2405326 and Simons Foundation.

%
%
%
%
%
%
%
%
%
%
%

\appendix
 
%
%
%
%
%
%
%
%
%
%
%

\section{Approximation of the initial data}\label{app:init}

In this appendix, we prove Lemma \ref{lem:init-app} and provide the diagonal compatibility of the approximate initial data with the regularization parameters. Using the BD potential $\phi$ introduced above, define
\[
\mathfrak z(\rho):=\int_1^\rho\sqrt{s}\,\phi'(s)\,\ds, \quad \nabla\mathfrak z(\rho)=\sqrt{\rho}\nabla\phi(\rho).
\]
Equivalently,
\[
\mathfrak z(\rho)=
\begin{cases}
\dfrac{2\alpha}{\alpha-\frac12}\lt(\rho^{\alpha-\frac12}-1\rt), & \alpha\neq\frac12,\\[2mm]
\log\rho, & \alpha=\frac12.
\end{cases}
\]
When $0<\alpha<\frac12$, the range of $\mathfrak z$ is $\lt(-\infty, \frac{2\alpha}{\frac12-\alpha}\rt)$. When $\alpha=\frac12$, the range of $\mathfrak z$ is the whole real line.

\begin{proof}[Proof of Lemma \ref{lem:init-app}]
We divide the construction into three steps.

First, we construct the kinetic approximation. Let $\chi_n\in C_c^\infty(\R^d)$ be such that ${\bf 1}_{B(0,n)} \le \chi_n \le {\bf 1}_{B(0,2n)}$. Set $\widetilde f_{0,n}(x,v):=\chi_n(v)f_0(x,v)$. Then $\widetilde f_{0,n}\to f_0$ in $W_2(\T^d\times\R^d)$ by dominated convergence. Let $\eta_n$ be a standard nonnegative mollifier on $\T^d\times\R^d$, and define $\overline f_{0,n}:=\eta_n*\widetilde f_{0,n}$. Then
\[
\overline f_{0,n}\in C^\infty, \quad \overline f_{0,n}\ge0, \quad \text{and} \quad   \overline f_{0,n}\to f_0
\quad\text{in } W_2(\T^d\times\R^d).
\]
Let
\[
M_n:=\inttr \overline f_{0,n}\,\dx\dv.
\]
Since $M_n\to1$, we may assume $M_n\ge\frac12$. We set $f_{0,n}:=\frac{\overline f_{0,n}}{M_n}$. Then
\[
\inttr f_{0,n}\,\dx\dv=1, \quad f_{0,n}\to f_0 \quad\text{in }W_2(\T^d\times\R^d).
\]

We next construct the density approximation. Set $q_0:=\sqrt{\rho_0}u_0$. By the initial BD assumption, we have
\[
q_0\in L^2(\T^d),
\quad
q_0+\nabla\mathfrak z(\rho_0)\in L^2(\T^d).
\]
In particular, $\nabla\mathfrak z(\rho_0)\in L^2(\T^d)$. We first truncate the density away from $0$ and $\infty$. Define
\[
\rho_0^{[n]}:=\min\lt\{\max\lt\{\rho_0, \frac1n\rt\},n\rt\} 
\]
Then $\frac1n\le \rho_0^{[n]}\le n$ and $\rho_0^{[n]}\to\rho_0$ a.e. and in $L^\gamma(\T^d)$. Indeed,
\[
|\rho_0^{[n]}-\rho_0| \le \frac1n{\bf 1}_{\{\rho_0<\frac1n\}} + \rho_0{\bf 1}_{\{\rho_0>n\}},
\]
and thus
\[
\|\rho_0^{[n]}-\rho_0\|_{L^\gamma(\T^d)}^\gamma
\le
\frac{|\T^d|}{n^\gamma}
+
\int_{\{\rho_0>n\}}\rho_0^\gamma\,\dx
\to0.
\]
 Set $Z_0^{[n]}:=\mathfrak z(\rho_0^{[n]})$. Since $\rho_0^{[n]}\in[\frac1n,n]$, the values of $Z_0^{[n]}$ lie in the compact interval $I_n:=\mathfrak z\lt(\lt[\frac1n,n\rt]\rt)$ contained in the range of $\mathfrak z$. Moreover, $Z_0^{[n]}$ is a bounded Lipschitz truncation of the possibly unbounded function $\mathfrak z(\rho_0)$. At the gradient level, we have $\nabla Z_0^{[n]} = {\bf 1}_{\{\frac1n<\rho_0<n\}}\nabla\mathfrak z(\rho_0)$ a.e. in $\T^d$. Thus, by dominated convergence, $\nabla Z_0^{[n]} \to \nabla\mathfrak z(\rho_0)$ in $L^2(\T^d)$.

We choose a nonnegative mollifier $\eta_{\delta_n}$, with $\delta_n>0$ sufficiently small, and define $Z_{0,n}:=\eta_{\delta_n}*Z_0^{[n]}$. Since $\eta_{\delta_n}$ is nonnegative and $Z_0^{[n]}$ takes values in $I_n$, the function $Z_{0,n}$ also takes values in $I_n$. Hence $\overline\rho_{0,n}:=\mathfrak z^{-1}(Z_{0,n})$ is well defined. Moreover, $\overline\rho_{0,n}\in C^\infty(\T^d)$, $\overline\rho_{0,n}>0$, and $\overline\rho_{0,n}$ is bounded above and away from zero for each fixed $n$.

We choose $\delta_n$ so small that
\[
\|\nabla Z_{0,n}-\nabla Z_0^{[n]}\|_{L^2(\T^d)} \le \frac1n.
\]
Then
\[
\nabla\mathfrak z(\overline\rho_{0,n}) = \nabla Z_{0,n} \to \nabla\mathfrak z(\rho_0) \quad\text{in }L^2(\T^d).
\]
Since $\mathfrak z^{-1}$ is smooth on the compact interval $I_n$, decreasing $\delta_n$ if necessary, we may also ensure
\[
\|\overline\rho_{0,n}-\rho_0^{[n]}\|_{L^\gamma(\T^d)}
\le
\frac1n.
\]
Here we used that $\mathfrak z^{-1}$ is Lipschitz on $I_n$. Since $\rho_0^{[n]}\to\rho_0$ in $L^\gamma(\T^d)$, we obtain $\overline\rho_{0,n}\to\rho_0$ in $L^\gamma(\T^d)$.

We now normalize the mass. Let
\[
a_n:=\intt \overline\rho_{0,n}\,\dx.
\]
Since $\overline\rho_{0,n}\to\rho_0$ in $L^1(\T^d)$ and $\intt\rho_0\,\dx=1$, we have $a_n\to1$. Set $\rho_{0,n}:=\frac{\overline\rho_{0,n}}{a_n}$. Then
\[
\rho_{0,n}\in C^\infty(\T^d), \quad \rho_{0,n}>0, \quad \intt\rho_{0,n}\,\dx=1, \quad \text{and} \quad \rho_{0,n}\to\rho_0 \quad\text{in }L^\gamma(\T^d).
\]

The normalization also preserves the BD convergence. Indeed, if $\alpha\neq\frac12$, then $\nabla\mathfrak z(\rho_{0,n}) = a_n^{-(\alpha-\frac12)} \nabla\mathfrak z(\overline\rho_{0,n})$, while for $\alpha=\frac12$, $\nabla\mathfrak z(\rho_{0,n}) = \nabla\log\rho_{0,n} = \nabla\log\overline\rho_{0,n}$. Since $a_n\to1$, we obtain
\[
\nabla\mathfrak z(\rho_{0,n}) \to \nabla\mathfrak z(\rho_0) \quad\text{in }L^2(\T^d).
\]

Finally, we construct the velocity approximation. Choose $q_{0,n}\in C^\infty(\T^d;\R^d)$ such that $q_{0,n}\to q_0$ in $L^2(\T^d)$. Since $\rho_{0,n}$ is smooth and strictly positive, we may define $u_{0,n}:=\frac{q_{0,n}}{\sqrt{\rho_{0,n}}}$. Then $u_{0,n}\in C^\infty(\T^d;\R^d)$, and
\[
\sqrt{\rho_{0,n}}u_{0,n} = q_{0,n} \to q_0 = \sqrt{\rho_0}u_0 \quad\text{in }L^2(\T^d).
\]
Moreover,
\[
\sqrt{\rho_{0,n}}\lt(u_{0,n}+\nabla\phi(\rho_{0,n})\rt) = q_{0,n} + \nabla\mathfrak z(\rho_{0,n}).
\]
Consequently,
\bq\label{app:BD-part}
\sqrt{\rho_{0,n}}\lt(u_{0,n}+\nabla\phi(\rho_{0,n})\rt) \to q_0+\nabla\mathfrak z(\rho_0) = \sqrt{\rho_0}\lt(u_0+\nabla\phi(\rho_0)\rt) \quad\text{in }L^2(\T^d).
\eq
This completes the proof of Lemma \ref{lem:init-app}.
\end{proof}

We next provide the diagonal compatibility with the regularized initial energies. For $\zeta,\tau>0$, we use the following regularized analogues of the free energy and BD entropy at the initial time:
\[
\calE_{{\rm free},\zeta,\tau}[\rho,u,f] := \frac12\inttr |v|^2f\,\dx\dv +\intt\lt[\frac12\rho|u|^2+\pi(\rho)+\tau \pi_{10}(\rho) +\frac{\tau}{2}|\nabla Y_\zeta(\rho)|^2\rt]\dx,
\]
and
\begin{align*}
\calE_{{\rm BD},\zeta,\tau}[\rho,u,f] &:= \frac12\inttr |v|^2f\,\dx\dv \cr
&\quad + \intt\lt[\frac12\rho|u+\nabla\phi_\zeta(\rho)|^2+\pi(\rho)+\tau \pi_{10}(\rho) +\frac{\tau}{2}|\nabla Y_\zeta(\rho)|^2+2\tau J_\zeta(\rho)\rt]\dx.
\end{align*}

The following lemma is the only place where the approximation parameters are tied to the smooth initial data.
\begin{lemma}\label{lem:init-energy-diag}
Let $(f_{0,n},\rho_{0,n},u_{0,n})$ be the approximate initial data constructed in Lemma \ref{lem:init-app}. Then there exist sequences $\zeta_n\to0$ and $\tau_n\to0$ such that
\[
\limsup_{n\to\infty} \lt[ \calE_{{\rm free},\zeta_n,\tau_n}[\rho_{0,n},u_{0,n},f_{0,n}] + \calE_{{\rm BD},\zeta_n,\tau_n}[\rho_{0,n},u_{0,n},f_{0,n}] \rt] \le \calE_{\rm free}(0)+\calE_{\rm BD}(0).
\]
\end{lemma}

\begin{proof}
The convergence of the physical kinetic energy follows from $f_{0,n}\to f_0$ in $W_2(\T^d\times\R^d)$. The convergence of the fluid kinetic energy follows from $\sqrt{\rho_{0,n}}u_{0,n}\to\sqrt{\rho_0}u_0$ in $L^2(\T^d)$. The convergence of the pressure potential follows from $\rho_{0,n}\to\rho_0$ in $L^\gamma(\T^d)$ and the growth of $\pi$. Finally, \eqref{app:BD-part} gives convergence of the unregularized BD kinetic part.

It remains to choose the   regularization parameters along the approximate initial data. Since each $\rho_{0,n}$ is smooth, strictly positive, and bounded above and below, all regularized initial quantities are finite for fixed $n$. Moreover, $\phi_\zeta'(\rho)=\phi'(\rho)+\frac{2\zeta}{\rho}$. Thus $\nabla\phi_\zeta(\rho_{0,n}) - \nabla\phi(\rho_{0,n}) = 2\zeta\nabla\log\rho_{0,n}$. Since the quantity
\[
\intt\rho_{0,n}|\nabla\log\rho_{0,n}|^2\,\dx
\]
is finite for each fixed $n$, we may choose $\zeta_n\to0$ sufficiently slowly so that
\[
\intt \rho_{0,n} \lt| \nabla\phi_{\zeta_n}(\rho_{0,n}) - \nabla\phi(\rho_{0,n}) \rt|^2\dx \to0.
\]
After $\zeta_n$ has been fixed, we choose $\tau_n\to0$ sufficiently slowly so that
\[
\tau_n\intt \lt[ \rho_{0,n}^{10} + |\nabla Y_{\zeta_n}(\rho_{0,n})|^2 + G_{\zeta_n}(\rho_{0,n}) \rt]\dx \to0.
\]
The additional $\zeta_n$- and $\tau_n$-dependent terms in the regularized energies therefore vanish, while the physical parts converge to the corresponding unregularized initial energies. This proves the claim. The choice is diagonal in the sense that the smooth initial data are fixed first, and the parameters are then chosen sufficiently slowly along this sequence.
\end{proof}
 
%
%
%
%
%
%
%
%
%
%
%

 \section{Auxiliary augmented entropy estimate}\label{app:aug}

In this appendix, we prove the auxiliary $\kappa$-entropy estimate used in Section \ref{ssec:aux-aug-app}. Throughout the appendix, the parameters
$\zeta,\tau>0$ are fixed, and $\kappa\in(0,\frac12)$ is fixed once and for all. The constants below may depend on $\e,\zeta,\tau,T$ and on this fixed $\kappa$, but are independent of $N$.

Let $(f_N,\rho_N,w_N)$ be a smooth solution of the auxiliary system \eqref{eq:aux-rho}--\eqref{eq:aux-w}. Recall that
\[
\mathfrak v_N=\nabla\phi_\zeta(\rho_N), \quad u_N=w_N-\kappa\mathfrak v_N.
\]
Then $\rho_N\mathfrak v_N=2\nabla\nu_\zeta(\rho_N)$, and the auxiliary density equation can be written in the physical form
\bq\label{eq:app-rhon}
\pa_t\rho_N+\nabla\cdot(\rho_Nu_N)=0.
\eq
Moreover,
\[
D u_N=D w_N-\kappa\nabla\mathfrak v_N, \quad A u_N=A w_N, \quad \nabla\cdot u_N=\nabla\cdot w_N-\kappa\nabla\cdot\mathfrak v_N,
\]
since $\nabla\mathfrak v_N=\nabla^2\phi_\zeta(\rho_N)$ is symmetric.

We recall the regularized functions
\[
Y_\zeta'(\rho):=\frac{2\nu_\zeta'(\rho)}{\sqrt{\rho}}, \quad s_\zeta'(\rho):=\frac{\nu_\zeta'(\rho)}{\rho}, \quad J_\zeta''(\rho):=\frac{\nu_\zeta'(\rho)}{\rho^2}.
\]
Then
\[
\nabla Y_\zeta(\rho_N)=\sqrt{\rho_N}\,\mathfrak v_N, \quad \nabla s_\zeta(\rho_N)=\frac12\mathfrak v_N, \quad \rho_N\nabla J_\zeta'(\rho_N)=\frac12\mathfrak v_N.
\]
We also recall the generalized Bohm identity:
\bq\label{eq:app-Bohm}
\mathfrak K_\zeta(\rho) = 2\nabla\cdot\lt[\nu_\zeta(\rho)\nabla^2(2s_\zeta(\rho))\rt] + \nabla\lt[\lambda_\zeta(\rho)\Delta(2s_\zeta(\rho))\rt].
\eq
Since
\[
(2s_\zeta)'(\rho)=\frac{2\nu_\zeta'(\rho)}{\rho}=\phi_\zeta'(\rho),
\]
we have
\[
\nabla(2s_\zeta(\rho_N))=\nabla\phi_\zeta(\rho_N)=\mathfrak v_N.
\]

We begin by testing the augmented Galerkin equation \eqref{eq:aux-w} by $w_N$. Since $w_N\in X_N$, this is an admissible test function. Using \eqref{eq:app-rhon}, the inertial part satisfies
\[
 \intt \pa_t (\rho_Nw_N)\cdot w_N\,\dx - \intt\rho_Nu_N\otimes w_N:\nabla w_N\,\dx = \frac{\dd}{\dt}\frac12\intt\rho_N|w_N|^2\,\dx .
\]
Thus,
\begin{align}\label{eq:app-w-tested}
\begin{aligned}
&\frac{\dd}{\dt}\frac12\intt\rho_N|w_N|^2\,\dx -\intt p_\tau(\rho_N)\nabla\cdot w_N\,\dx +2(1-\kappa)\intt\nu_\zeta(\rho_N)|D w_N|^2\,\dx \\
&\quad +2\kappa\intt\nu_\zeta(\rho_N)|A w_N|^2\,\dx +(1-\kappa)\intt\lambda_\zeta(\rho_N)|\nabla\cdot w_N|^2\,\dx \\
&\quad -2\kappa(1-\kappa)\intt\nu_\zeta(\rho_N) \nabla\mathfrak v_N:\nabla w_N\,\dx   -\kappa(1-\kappa)\intt\lambda_\zeta(\rho_N) \nabla\cdot\mathfrak v_N\,\nabla\cdot w_N\,\dx +\epsilon_N\calG_N(t) \\
&\quad \quad = \intt\rho_N(F_N)_\e\cdot w_N\,\dx +\tau\intt\mathfrak K_\zeta(\rho_N)\cdot w_N\,\dx -\tau\intt u_N\cdot w_N\,\dx ,
\end{aligned}
\end{align}
where
\[
\calG_N(t) = \intt\lt[ |\Delta^s w_N|^2 + (1+|\nabla w_N|^2)|\nabla w_N|^2 \rt]\dx .
\]

We next differentiate the BD drift energy. The standard BD calculation based on the continuity equation gives
\begin{align}\label{eq:app-drift-energy}
\begin{aligned}
&\frac{\dd}{\dt}\frac{\kappa(1-\kappa)}2 \intt\rho_N|\mathfrak v_N|^2\,\dx \\
&\quad -2\kappa(1-\kappa)\intt\nu_\zeta(\rho_N) \nabla\mathfrak v_N:\nabla u_N\,\dx   -\kappa(1-\kappa)\intt\lambda_\zeta(\rho_N) \nabla\cdot\mathfrak v_N\,\nabla\cdot u_N\,\dx =0 .
\end{aligned}
\end{align}
Using $u_N=w_N-\kappa\mathfrak v_N$, the last two terms in
\eqref{eq:app-drift-energy} become
\begin{align}\label{eq:app-drift-expand}
\begin{aligned}
&-2\kappa(1-\kappa)\intt\nu_\zeta(\rho_N) \nabla\mathfrak v_N:\nabla u_N\,\dx   -\kappa(1-\kappa)\intt\lambda_\zeta(\rho_N) \nabla\cdot\mathfrak v_N\,\nabla\cdot u_N\,\dx \\
&= -2\kappa(1-\kappa)\intt\nu_\zeta(\rho_N) \nabla\mathfrak v_N:\nabla w_N\,\dx   +2\kappa^2(1-\kappa)\intt\nu_\zeta(\rho_N) |\nabla\mathfrak v_N|^2\,\dx \\
&\quad -\kappa(1-\kappa)\intt\lambda_\zeta(\rho_N) \nabla\cdot\mathfrak v_N\,\nabla\cdot w_N\,\dx  +\kappa^2(1-\kappa)\intt\lambda_\zeta(\rho_N) |\nabla\cdot\mathfrak v_N|^2\,\dx .
\end{aligned}
\end{align}
Adding \eqref{eq:app-drift-energy} to \eqref{eq:app-w-tested}, and then using
\eqref{eq:app-drift-expand}, the shear viscosity terms are completed as
\begin{align*}
&2(1-\kappa)\intt\nu_\zeta(\rho_N)|D w_N|^2\,\dx -4\kappa(1-\kappa)\intt\nu_\zeta(\rho_N) \nabla\mathfrak v_N:\nabla w_N\,\dx \\
&\quad +2\kappa^2(1-\kappa)\intt\nu_\zeta(\rho_N) |\nabla\mathfrak v_N|^2\,\dx + 2\kappa\intt\nu_\zeta(\rho_N)|A w_N|^2\,\dx \\
&\quad \quad = 2(1-\kappa)\intt\nu_\zeta(\rho_N) |D w_N-\kappa\nabla\mathfrak v_N|^2\,\dx + 2\kappa\intt\nu_\zeta(\rho_N)|A w_N|^2\,\dx .
\end{align*}
Here we used that $\nabla\mathfrak v_N$ is symmetric. Similarly, the bulk viscosity terms satisfy
\begin{align*}
&(1-\kappa)\intt\lambda_\zeta(\rho_N)|\nabla\cdot w_N|^2\,\dx   -2\kappa(1-\kappa)\intt\lambda_\zeta(\rho_N) \nabla\cdot\mathfrak v_N\,\nabla\cdot w_N\,\dx  
+\kappa^2(1-\kappa)\intt\lambda_\zeta(\rho_N) |\nabla\cdot\mathfrak v_N|^2\,\dx \\
&= (1-\kappa)\intt\lambda_\zeta(\rho_N) |\nabla\cdot w_N-\kappa\nabla\cdot\mathfrak v_N|^2\,\dx .
\end{align*}

We now rewrite the remaining terms in \eqref{eq:app-w-tested}. First, since $w_N=u_N+\kappa\mathfrak v_N$, the pressure term satisfies
\begin{align*}
-\intt p_\tau(\rho_N)\nabla\cdot w_N\,\dx &= -\intt p_\tau(\rho_N)\nabla\cdot u_N\,\dx -\kappa\intt p_\tau(\rho_N)\nabla\cdot\mathfrak v_N\,\dx \\
&= \frac{\dd}{\dt}\intt\lt[\pi(\rho_N)+\tau \pi_{10}(\rho_N)\rt]\dx +\kappa\intt\nabla p_\tau(\rho_N)\cdot\mathfrak v_N\,\dx .
\end{align*}
Here $p_\tau(\rho)=p(\rho)+\tau\rho^{10}$, and $\pi+\tau \pi_{10}$ is the corresponding pressure potential.

Next, by \eqref{eq:app-Bohm} and $\nabla(2s_\zeta(\rho_N))=\mathfrak v_N$, we have
\begin{align*}
\intt\mathfrak K_\zeta(\rho_N)\cdot\mathfrak v_N\,\dx &= -2\intt\nu_\zeta(\rho_N) |\nabla^2(2s_\zeta(\rho_N))|^2\,\dx  -\intt\lambda_\zeta(\rho_N) |\Delta(2s_\zeta(\rho_N))|^2\,\dx \\
&= -4\intt\lt[ 2\nu_\zeta(\rho_N)|\nabla^2s_\zeta(\rho_N)|^2 + \lambda_\zeta(\rho_N)|\Delta s_\zeta(\rho_N)|^2 \rt]\dx .
\end{align*}
Moreover, the continuity equation gives
\[
\intt\mathfrak K_\zeta(\rho_N)\cdot u_N\,\dx = -\frac12\frac{\dd}{\dt} \intt|\nabla Y_\zeta(\rho_N)|^2\,\dx .
\]
Thus,
\begin{align*}
\tau\intt\mathfrak K_\zeta(\rho_N)\cdot w_N\,\dx &= \tau\intt\mathfrak K_\zeta(\rho_N)\cdot u_N\,\dx + \tau\kappa\intt\mathfrak K_\zeta(\rho_N)\cdot\mathfrak v_N\,\dx \\
&= -\frac{\tau}{2}\frac{\dd}{\dt} \intt|\nabla Y_\zeta(\rho_N)|^2\,\dx  -4\tau\kappa\intt\lt[ 2\nu_\zeta(\rho_N)|\nabla^2s_\zeta(\rho_N)|^2 + \lambda_\zeta(\rho_N)|\Delta s_\zeta(\rho_N)|^2 \rt]\dx .
\end{align*}
Accordingly, we define
\[
\calD_{{\rm cap},N}^{\kappa} := 4\tau\kappa\intt\lt[ 2\nu_\zeta(\rho_N)|\nabla^2s_\zeta(\rho_N)|^2 + \lambda_\zeta(\rho_N)|\Delta s_\zeta(\rho_N)|^2 \rt]\dx .
\]
For the damping term, using $w_N=u_N+\kappa\mathfrak v_N$ and $\rho_N\nabla J_\zeta'(\rho_N)=\frac{\mathfrak v_N}2$, we get
\begin{align*}
-\tau\intt u_N\cdot w_N\,\dx &= -\tau\intt |u_N|^2\,\dx -\tau\kappa\intt u_N\cdot\mathfrak v_N\,\dx \\
&= -\tau\intt |u_N|^2\,\dx -2\tau\kappa\intt\rho_Nu_N\cdot\nabla J_\zeta'(\rho_N)\,\dx \\
&= -\tau\intt |u_N|^2\,\dx -2\tau\kappa\frac{\dd}{\dt}\intt J_\zeta(\rho_N)\,\dx .
\end{align*}
We also compute the kinetic energy. Multiplying \eqref{eq:aux-kinetic} by $\frac{|v|^2}2$ yields
\begin{align*}
\frac{\dd}{\dt}\frac12\inttr |v|^2f_N\,\dx\dv &= -\inttr(\rho_N)_\e v\cdot(v-u_{N,\e}^\e)f_N\,\dx\dv \\
&= -\inttr(\rho_N)_\e|v-u_{N,\e}^\e|^2f_N\,\dx\dv -\intt(\rho_Nu_N)_\e\cdot F_N\,\dx \\
&= -\calD_{{\rm drag},N} -\intt(\rho_Nu_N)_\e\cdot F_N\,\dx ,
\end{align*}
where
\[
F_N:=\intr(v-u_{N,\e}^\e)f_N\,\dv, \quad u_{N,\e}^\e:=\frac{(\rho_Nu_N)_\e}{(\rho_N)_\e}, \quad \calD_{{\rm drag},N}:= \inttr(\rho_N)_\e|v-u_{N,\e}^\e|^2f_N\,\dx\dv.
\]
 
Combining all of the above estimates, we obtain
\begin{align}\label{eq:app-pre-entropy}
\begin{aligned}
&\frac{\dd}{\dt}\Bigg\{ \frac12\inttr |v|^2f_N\,\dx\dv +\intt\lt[ \frac12\rho_N|w_N|^2 +\frac{\kappa(1-\kappa)}2\rho_N|\mathfrak v_N|^2 +\pi(\rho_N) +\tau \pi_{10}(\rho_N) \rt]\dx \\
&\hspace{28mm} +\frac{\tau}{2}\intt|\nabla Y_\zeta(\rho_N)|^2\,\dx +2\tau\kappa\intt J_\zeta(\rho_N)\,\dx \Bigg\} \\
&\quad + 2(1-\kappa)\intt\nu_\zeta(\rho_N) |D w_N-\kappa\nabla\mathfrak v_N|^2\,\dx   + (1-\kappa)\intt\lambda_\zeta(\rho_N) |\nabla\cdot w_N-\kappa\nabla\cdot\mathfrak v_N|^2\,\dx \\
&\quad + 2\kappa\intt\nu_\zeta(\rho_N)|A w_N|^2\,\dx + \kappa\intt\nabla p_\tau(\rho_N)\cdot\mathfrak v_N\,\dx + \calD_{{\rm drag},N} + \calD_{{\rm damp},N} + \calD_{{\rm cap},N}^{\kappa} + \epsilon_N\calG_N(t) \\
&\quad \quad = \calR_{{\rm kin},N}^{\kappa},
\end{aligned}
\end{align}
where
\[
\calD_{{\rm damp},N}:=\tau\intt |u_N|^2\,\dx,
\]
and the kinetic coupling remainder is
\[
\calR_{{\rm kin},N}^{\kappa} := \intt\rho_N(F_N)_\e\cdot w_N\,\dx - \intt(\rho_Nu_N)_\e\cdot F_N\,\dx = \kappa\intt\rho_N\mathfrak v_N\cdot(F_N)_\e\,\dx .
\]
Here we used the symmetry of the mollifier and the identity
$w_N-u_N=\kappa\mathfrak v_N$.

It remains to estimate the remainder. The mollified drag bound gives
\[
\intt\rho_N|(F_N)_\e|^2\,\dx
\le
C_\e\calD_{{\rm drag},N}.
\]
Thus, by Young's inequality,
\[
|\calR_{{\rm kin},N}^{\kappa}| \le \kappa\lt(\intt\rho_N|\mathfrak v_N|^2\,\dx\rt)^{\frac12} \lt(\intt\rho_N|(F_N)_\e|^2\,\dx\rt)^{\frac12}  \le \frac12 \calD_{{\rm drag},N} + C_\e\kappa^2\intt\rho_N|\mathfrak v_N|^2\,\dx .
\]
We absorb half of the drag dissipation into the left-hand side of \eqref{eq:app-pre-entropy}. Since $\kappa\in(0,\frac12)$ is fixed, the remaining term
\[
C_\e\kappa^2\intt\rho_N|\mathfrak v_N|^2\,\dx
\]
is controlled by the augmented energy
\[
\frac{\kappa(1-\kappa)}2\intt\rho_N|\mathfrak v_N|^2\,\dx .
\]
Applying Gr\"onwall's lemma then gives, for every $T>0$,
\begin{align*}
&\sup_{0\le t\le T}\Bigg\{ \frac12\inttr |v|^2f_N\,\dx\dv +\intt\lt[ \frac12\rho_N|w_N|^2 +\frac{\kappa(1-\kappa)}2\rho_N|\mathfrak v_N|^2 +\pi(\rho_N) +\tau \pi_{10}(\rho_N) \rt]\dx \\
&\hspace{25mm} +\frac{\tau}{2}\intt|\nabla Y_\zeta(\rho_N)|^2\,\dx +2\tau\kappa\intt J_\zeta(\rho_N)\,\dx \Bigg\} \\
&\quad +\int_0^T\Bigg[ 2(1-\kappa)\intt\nu_\zeta(\rho_N)|D w_N-\kappa\nabla\mathfrak v_N|^2\,\dx +(1-\kappa)\intt\lambda_\zeta(\rho_N)|\nabla\cdot w_N-\kappa\nabla\cdot\mathfrak v_N|^2\,\dx \\
&\hspace{25mm} +2\kappa\intt\nu_\zeta(\rho_N)|A w_N|^2\,\dx +\kappa\intt\nabla p_\tau(\rho_N)\cdot\mathfrak v_N\,\dx \\
&\hspace{25mm} +\frac12\calD_{{\rm drag},N}(t) +\calD_{{\rm damp},N}(t) +\calD_{{\rm cap},N}^{\kappa}(t) +\epsilon_N\calG_N(t) \Bigg]\dt \\
&\quad \quad \le C_{\e,\zeta,\tau,T}.
\end{align*}
This proves Lemma \ref{lem:aux-kappa-entropy}.

%
%
%
%
%
%
%
%
%
%
%

%
%
%
%


\end{document}